%% file: 3StageCriteria.tex
\documentclass{amsart}

\usepackage{amssymb}
\usepackage{amsmath}
\usepackage{amsthm}
\usepackage{amsfonts}

\usepackage{a4wide}
\usepackage{multirow}

\newtheorem{theorem}{Theorem}[section]
\newtheorem{proposition}{Proposition}[section]
\newtheorem{corollary}{Corollary}[section]
\newtheorem{lemma}{Lemma}[section]

\theoremstyle{definition}
\newtheorem{definition}{Definition}[section]
\newtheorem{remark}{Remark}[section]

\numberwithin{equation}{section}

\begin{document}


\title
[Three-stage \(p\)-class towers]
{Criteria for three-stage towers of \(p\)-class fields}

\author{Daniel C. Mayer}
\address{Naglergasse 53\\8010 Graz\\Austria}
\email{algebraic.number.theory@algebra.at}
\urladdr{http://www.algebra.at}

\thanks{Research supported by the Austrian Science Fund (FWF): P 26008-N25}


\date{November 24, 2016}

\begin{abstract}
Let \(p\) be a prime
and \(K\) be a number field
with non-trivial \(p\)-class group \(\mathrm{Cl}_p{K}\).
A crucial step in identifying the Galois group \(G=\mathrm{G}_p^\infty{K}\)
of the maximal unramified pro-\(p\) extension of \(K\)
is to determine its two-stage approximation \(\mathfrak{M}=\mathrm{G}_p^2{K}\),
that is the second derived quotient \(\mathfrak{M}\simeq G/G^{\prime\prime}\).
The family \(\tau_1{K}\) of abelian type invariants of the \(p\)-class groups \(\mathrm{Cl}_p{L}\)
of all unramified cyclic extensions \(L/K\) of degree \(p\)
is called the \textit{index-\(p\) abelianization data} (IPAD) of \(K\).
It is able to specify a finite batch of contestants
for the second \(p\)-class group \(\mathfrak{M}\) of \(K\).
In this paper we introduce two different kinds of \textit{generalized} IPADs
for obtaining more sophisticated results.
The \textit{multi-layered} IPAD \((\tau_1{K},\tau_2{K})\)
includes data on unramified abelian extensions \(L/K\) of degree \(p^2\)
and enables sharper bounds for the order of \(\mathfrak{M}\)
in the case \(\mathrm{Cl}_p{K}\simeq (p,p,p)\),
where current implementations of the \(p\)-group generation algorithm
fail to produce explicit contestants for \(\mathfrak{M}\),
due to memory limitations.
The \textit{iterated} IPAD of \textit{second order} \(\tau^{(2)}{K}\)
contains information on non-abelian unramified extensions \(L/K\) of degree \(p^2\), or even \(p^3\),
and admits the identification of the \(p\)-class tower group \(G\) 
for various infinite series of quadratic fields \(K=\mathbb{Q}(\sqrt{d})\)
with \(\mathrm{Cl}_p{K}\simeq (p,p)\)
possessing a \(p\)-class field tower of exact length \(\ell_p{K}=3\)
as a striking novelty.
\end{abstract}

\maketitle



\section{Introduction}
\label{s:Intro}
\noindent
In a previous article
\cite{Ma7},
we provided a systematic and rigorous introduction of the concepts of
abelian type invariants and iterated IPADs of higher order.
These ideas were communicated together with impressive numerical applications
at the \(29\)th Journ\'ees Arithm\'etiques in Debrecen, July \(2015\)
\cite{Ma7b}.
The purpose and the organization of the present article,
which considerably extends the computational and theoretical results in
\cite{Ma7,Ma7b},
is as follows.

\textit{Index-\(p\) abelianization data} (IPADs)
are explained in \S\
\ref{s:IPAD}.
Our Main Theorem on three-stage towers of \(3\)-class fields
is communicated in \S\
\ref{s:Principalization}.
Basic definitions concerning the \textit{Artin transfer pattern}
\cite{Ma7,Ma8,Ma9}
are recalled in \S\
\ref{s:ArtinPattern}.
Then we generally put \(p=3\) and consider \(3\)-class tower groups
(\S\ \ref{s:QF}).
In \S\
\ref{s:AllIPADs},
we first restate a \textit{summary of all possible} IPADs of a number field \(K\)
with \(3\)-class group \(\mathrm{Cl}_3{K}\) of type \((3,3)\)
\cite[Thm. 3.1--3.2, pp. 290--291]{Ma7}
in a more succinct and elegant form
avoiding infinitely many exceptions,
and emphasizing the role of two distinguished components,
called the \textit{polarization} and \textit{co-polarization},
which are crucial for proving the finiteness of the batch of \textit{contestants}
for the second \(3\)-class group \(\mathfrak{M}=\mathrm{G}_3^2{K}\).
Up to now, this is the unique situation
where all IPADs can be given in a complete form,
except for the simple case of a number field \(K\)
with \(2\)-class group \(\mathrm{Cl}_2{K}\) of type \((2,2)\)
\cite[\S\ 9, pp. 501--503]{Ma1}.
We characterize all relevant finite \(3\)-groups
by IPADs of first and second order in \S\S\
\ref{ss:Coclass1},
\ref{ss:TreeQ},
\ref{ss:TreeU},
\ref{ss:SporadicCoclass2}.
These groups constitute the \textit{candidates} for \(3\)-class tower groups \(\mathrm{G}_3^\infty{K}\)
of quadratic fields \(K=\mathbb{Q}(\sqrt{d})\)
with \(3\)-class group \(\mathrm{Cl}_3{K}\) of type \((3,3)\).
In \S\
\ref{ss:RQFa},
results for the dominant scenario with \(3\)-principalization \(\varkappa_1{K}\) of type \(\mathrm{a}\) are given.
In \S\S\
\ref{ss:NFE6E14},
\ref{ss:NFE8E9},
we provide evidence of unexpected phenomena revealed by \textit{real} quadratic fields \(K\)
with types \(\varkappa_1{K}\)
in Scholz and Taussky's section \(\mathrm{E}\)
\cite[p. 36]{SoTa}.
Their \(3\)-class tower can be of length \(2\le\ell_3{K}\le 3\)
and a sharp decision is possible by means of \textit{iterated} IPADs \textit{of second order}.
We point out that \textit{imaginary} quadratic fields with type \(\mathrm{E}\)
must always have a tower of exact length \(\ell_3{K}=3\)
\cite{Ma8,BuMa}.
In \S\S\
\ref{ss:RQFH4Spor},
\ref{ss:IQFH4Spor},
resp. \S\S\
\ref{ss:RQFG19Spor},
\ref{ss:IQFG19Spor},
results for quadratic fields \(K\)
with \(3\)-principalization type \(\mathrm{H}.4\), \(\varkappa_1{K}\sim(4111)\),
resp. \(\mathrm{G}.19\), \(\varkappa_1{K}\sim(2143)\), are proved.

In the last section \S\
\ref{s:IQF3x3x3}
on \textit{multi-layered} IPADs,
it is our endeavour to point out that
the rate of growth of successive derived quotients \(\mathrm{G}_p^n{K}\simeq G/G^{(n)}\), \(n\ge 2\),
of the \(p\)-class tower group \(G=\mathrm{G}_p^\infty{K}\) is still far from being known
for imaginary quadratic fields \(K\) with \(p\)-class rank \(\varrho\ge 3\),
where the criterion of Koch and Venkov
\cite{KoVe}
ensures an infinite \(p\)-class tower with \(\ell_p{K}=\infty\).




\section{Index-\(p\) abelianization data}
\label{s:IPAD}
\noindent
Let \(p\) be a prime number.
According to the Artin reciprocity law of class field theory
\cite{Ar1},
the unramified cyclic extensions \(L/K\) of relative degree \(p\)
of a number field \(K\) with non-trivial \(p\)-class group \(\mathrm{Cl}_p{K}\)
are in a bijective correspondence to the subgroups of index \(p\) in \(\mathrm{Cl}_p{K}\).
Their number is given by \(\frac{p^\varrho-1}{p-1}\)
if \(\varrho\) denotes the \(p\)-class rank of \(K\)
\cite[Thm. 3.1]{Ma5}.
The reason for this fact is that
the Galois group \(\mathrm{G}_p^1{K}:=\mathrm{Gal}(\mathrm{F}_p^1{K}/K)\)
of the maximal unramified \textit{abelian} \(p\)-extension \(\mathrm{F}_p^1{K}/K\),
which is called the first Hilbert \(p\)-class field of \(K\),
is isomorphic to the \(p\)-class group \(\mathrm{Cl}_p{K}\).
The fields \(L\) are contained in \(\mathrm{F}_p^1{K}\)
and each group \(\mathrm{Gal}(\mathrm{F}_p^1{K}/L)\)
is of index \(p\) in \(\mathrm{G}_p^1{K}\simeq\mathrm{Cl}_p{K}\).

It was also Artin's idea
\cite{Ar2}
to leave the abelian setting of class field theory
and to consider the second Hilbert \(p\)-class field
\(\mathrm{F}_p^2{K}=\mathrm{F}_p^1(\mathrm{F}_p^1{K})\),
that is the maximal unramified \textit{metabelian} \(p\)-extension of \(K\),
and its Galois group \(\mathfrak{M}:=\mathrm{G}_p^2{K}:=\mathrm{Gal}(\mathrm{F}_p^2{K}/K)\),
the so-called second \(p\)-class group of \(K\)
\cite{Ma1},
\cite[p. 41]{SoTa},
for proving the \textit{principal ideal theorem}
that \(\mathrm{Cl}_p{K}\) becomes trivial when it is extended to \(\mathrm{Cl}_p(\mathrm{F}_p^1{K})\)
\cite{Fw}.
Since \(K\le L\le\mathrm{F}_p^1{K}\le\mathrm{F}_p^1{L}\le\mathrm{F}_p^2{K}\) is
a non-decreasing tower of normal extensions
for any assigned unramified abelian \(p\)-extension \(L/K\),
the \(p\)-class group of \(L\),
\(\mathrm{Cl}_p{L}\simeq\mathrm{Gal}(\mathrm{F}_p^1{L}/L)
\simeq\mathrm{Gal}(\mathrm{F}_p^2{K}/L)\bigm/\mathrm{Gal}(\mathrm{F}_p^2{K}/\mathrm{F}_p^1{L})\),
is isomorphic to the \textit{abelianization} \(H/H^\prime\)
of the subgroup \(H:=\mathrm{Gal}(\mathrm{F}_p^2{K}/L)\) of the
second \(p\)-class group \(\mathrm{Gal}(\mathrm{F}_p^2{K}/K)\)
which corresponds to \(L\)
and whose commutator subgroup is given by \(H^\prime=\mathrm{Gal}(\mathrm{F}_p^2{K}/\mathrm{F}_p^1{L})\).

In particular, the structure of the \(p\)-class groups \(\mathrm{Cl}_p{L}\) of all
unramified cyclic extensions \(L/K\) of relative degree \(p\)
can be interpreted as the \textit{abelian type invariants} of all \textit{abelianizations} \(H/H^\prime\)
of subgroups \(H=\mathrm{Gal}(\mathrm{F}_p^2{K}/L)\) of \textit{index} \(p\)
in the second \(p\)-class group \(\mathrm{Gal}(\mathrm{F}_p^2{K}/K)\),
which has been dubbed the \textit{index-\(p\) abelianization data}, briefly IPAD, \(\tau_1{K}\) of \(K\)
by Boston, Bush, and Hajir
\cite{BBH}.
This kind of information would have been incomputable and thus useless about twenty years ago.
However, with the availability of computational algebra systems like PARI/GP
\cite{PARI}
and MAGMA
\cite{BCP,BCFS,MAGMA}
it became possible to compute the class groups \(\mathrm{Cl}_p{L}\),
collect their structures in the IPAD \(\tau_1{K}\),
reinterpret them as abelian quotient invariants of subgroups \(H\) of \(\mathrm{G}_p^2{K}\),
and to use this information for characterizing a batch of finitely many \(p\)-groups,
occasionally even a unique \(p\)-group,
as contestants for the second \(p\)-class group \(\mathfrak{M}=\mathrm{G}_p^2{K}\) of \(K\),
which in turn is a \textit{two-stage approximation} of the (potentially infinite)
pro-\(p\) group \(G:=\mathrm{G}_p^\infty{K}:=\mathrm{Gal}(\mathrm{F}_p^\infty{K}/K)\)
of the maximal unramified pro-\(p\) extension \(\mathrm{F}_p^\infty{K}\) of \(K\),
that is its \textit{Hilbert \(p\)-class tower}.

As we proved in the main theorem of
\cite[Thm. 5.4]{Ma9},
the IPAD is usually unable to permit a decision about the \textit{length} \(\ell:=\ell_p{K}\)
of the \(p\)-class tower of \(K\)
when non-metabelian candidates for \(\mathrm{G}_p^\infty{K}\) exist.
For solving such problems, \textit{iterated} IPADs \(\tau^{(2)}{K}\) of \textit{second order} are required.



\section{The \(p\)-principalization type}
\label{s:Principalization}
\noindent
Until very recently, the length \(\ell\) of the \(p\)-class tower
\[K<\mathrm{F}_p^1{K}<\mathrm{F}_p^2{K}<\ldots\mathrm{F}_p^\ell{K}=\mathrm{F}_p^{\ell+1}{K}=\ldots=\mathrm{F}_p^\infty{K}\]
over a quadratic field \(K=\mathbb{Q}(\sqrt{d})\) with \(p\)-class rank \(\varrho=2\),
that is, with \(p\)-class group \(\mathrm{Cl}_p{K}\) of type \((p^u,p^v)\), \(u\ge v\ge 1\),
was an open problem.
Apart from the proven impossibility of an abelian tower with \(\ell=1\)
\cite[Thm. 4.1.(1)]{Ma1},
it was unknown which values \(\ell\ge 2\) can occur and
whether \(\ell=\infty\) is possible or not.
In contrast, it is known that \(\ell=1\) for any number field \(K\) with \(p\)-class rank \(\varrho=1\),
i.e., with non-trivial cyclic \(p\)-class group \(\mathrm{Cl}_p{K}\),
and that \(\ell=\infty\) for an imaginary quadratic field with \(p\)-class rank \(\varrho\ge 3\),
when \(p\) is odd
\cite{KoVe}.

The finite batch of contestants for \(\mathfrak{M}=\mathrm{G}_p^2{K}\), specified by the IPAD \(\tau_1{K}\),
can be narrowed down further if the \(p\)-principalization type of \(K\) is known.
That is the family \(\varkappa_1{K}\) of all kernels \(\ker{T_{K,L}}\) of \(p\)-class transfers \(T_{K,L}:\mathrm{Cl}_p{K}\to\mathrm{Cl}_p{L}\)
from \(K\) to unramified cyclic superfields \(L\) of degree \(p\) over \(K\).
In view of the open problem for the length of the \(p\)-class tower,
there arose the question whether each possible \(p\)-principalization type \(\varkappa_1{K}\)
of a quadratic field \(K\) with \(\mathrm{Cl}_p{K}\) of type \((p,p)\)
is associated with a fixed value of the tower length \(\ell_p{K}\).

For \(p=3\) and \(\mathrm{Cl}_3{K}\) of type \((3,3)\),
there exist \(23\) distinct \(3\)-principalization types
\cite[Tbl. 6--7]{Ma2},
designated by \(\mathrm{X.n}\),
where \(\mathrm{X}\) denotes a letter in
\(\lbrace\mathrm{A,D,E,F,G,H,a,b,c,d}\rbrace\)
and \(\mathrm{n}\) denotes a certain integer in
\(\lbrace 1,\ldots,25\rbrace\), more explicitly:
\begin{center}
\begin{tabular}{rrrrrrrrrrrrrr}
\(\mathrm{A}.1\), & \(\mathrm{D}.5\), & \(\mathrm{D}.10\), & \(\mathrm{E}.6\), & \(\mathrm{E}.8\), & \(\mathrm{E}.9\), & \(\mathrm{E}.14\), & \(\mathrm{F}.7\), & \(\mathrm{F}.11\), & \(\mathrm{F}.12\), & \(\mathrm{F}.13\), & \(\mathrm{G}.16\), & \(\mathrm{G}.19\), & \(\mathrm{H}.4\), \\
\(\mathrm{a}.1\), & \(\mathrm{a}.2\), & \(\mathrm{a}.3\), & \(\mathrm{b}.10\), & \(\mathrm{c}.18\), & \(\mathrm{c}.21\), & \(\mathrm{d}.19\), & \(\mathrm{d}.23\), & \(\mathrm{d}.25\). \\
\end{tabular}
\end{center}


In this article,
we establish the last but one step for the proof of the following
solution to the open problem for \(p=3\) and quadratic fields \(K\) with \(\mathrm{Cl}_3{K}\simeq(3,3)\).

\begin{theorem}
\label{thm:MainTheorem}
(Main theorem on the length of the \(3\)-class tower for \(3\)-class rank two)
\begin{enumerate}
\item
For each of the \(13\) types of \(3\)-principalization \(\mathrm{X.n}\) with upper case letter \(\mathrm{X}\ne\mathrm{A}\),
there exists an imaginary quadratic field \(K=\mathbb{Q}(\sqrt{d})\), \(d<0\), of that type such that \(2\le\ell_3{K}\le 3\).
\item
For each of the \(22\) types of \(3\)-principalization \(\mathrm{X.n}\ne\mathrm{A.1}\),
there exists a real quadratic field \(K=\mathbb{Q}(\sqrt{d})\), \(d>0\), of that type such that \(2\le\ell_3{K}\le 3\).
\end{enumerate}
\end{theorem}


\begin{remark}
\label{rmk:ExceptionalType}
Type \(\mathrm{A.1}\) must be excluded for quadratic base fields \(K\), according to
\cite[Cor. 4.2]{Ma1}.
It occurs, however, with \(\ell=2\) for cyclic cubic fields with two primes dividing the conductor
\cite{Ta1}.
\end{remark}


Concerning the steps for the proof, we provide information in the form of Table
\ref{tbl:CapitulationTypes}.
An asterisk indicates the present paper.
The last step has been completed in collaboration with M. F. Newman but has not been published yet
\cite{MaNm}.
Only the types \(\mathrm{G}.16\) and \(\mathrm{G}.19\) must be distinguished by their integer identifier,
otherwise the types denoted by the same letter behave completely similar.
Additionally, we give the smallest logarithmic order \(\mathrm{lo}(G):=\log_3\lvert G\rvert\).



\renewcommand{\arraystretch}{1.2}

\begin{table}[ht]
\caption{Steps of the proof with references}
\label{tbl:CapitulationTypes}
\begin{center}
\begin{tabular}{|c|cccccc|cccc|cc|}
\hline
 Type               & \(\mathrm{D}\) & \(\mathrm{E}\) & \(\mathrm{F}\) & \(\mathrm{G}.16\) & \(\mathrm{G}.19\) & \(\mathrm{H}\) & \(\mathrm{a}\) & \(\mathrm{b}\) & \(\mathrm{c}\) & \(\mathrm{d}\) && Base Fields \\
\hline
 \(\ell_3{K}\)      & \(2\)          & \(3\)          & \(\ge 3\)      & \(\ge 3\)         & \(\ge 3\)         & \(\ge 3\)      &                &                &                &                & \multirow{3}{*}{\(\Biggr\rbrace\)} & imaginary \\
 \(\mathrm{lo}(G)\) & \(5\)          & \(\ge 8\)      & \(\ge 20\)     & \(\ge 11\)        & \(\ge 11\)        & \(\ge 8\)      &                &                &                &                && quadratic \\
 Ref.               & \cite{SoTa}    & \cite{BuMa}    & \cite{MaNm}    & \cite{MaNm}       & \cite{MaNm}       & \(\ast\)       &                &                &                &                && fields \\
\hline
 \(\ell_3{K}\)      & \(2\)          & \(2\) or \(3\) & \(\ge 3\)      & \(\ge 3\)         & \(\ge 3\)         & \(\ge 3\)      & \(2\)          & \(\ge 3\)      & \(3\)          & \(\ge 3\)      & \multirow{3}{*}{\(\Biggr\rbrace\)} & real \\
 \(\mathrm{lo}(G)\) & \(5\)          & \(\ge 7\)      & \(\ge 10\)     & \(\ge 9\)         & \(\ge 7\)         & \(\ge 7\)      & \(\ge 4\)      & \(\ge 10\)     & \(\ge 7\)      & \(\ge 10\)     && quadratic \\
 Ref.               & \cite{Ma14}    & \(\ast\)       & \cite{MaNm}    & \cite{MaNm}       & \(\ast\)          & \(\ast\)       & \(\ast\)       & \cite{MaNm}    & \cite{Ma10a,Ma10} & \cite{MaNm} && fields \\
\hline
\end{tabular}
\end{center}
\end{table}



\begin{remark}
\label{rmk:MainTheorem}
None of the types sets in with a length \(\ell\ge 4\).
Type \(\mathrm{D}\) behaves completely rigid with \(\ell=2\), fixed class \(3\), and coclass \(2\).
Type \(\mathrm{a}\) is also confined to \(\ell=2\) but admits unbounded nilpotency class with fixed coclass \(1\).
For type \(\mathrm{E}\), we have \(\ell=3\) with unbounded class and coclass for imaginary fields,
and the unique exact dichotomy \(\ell\in\lbrace 2,3\rbrace\) for real fields.
For type \(\mathrm{c}\), the length \(\ell=3\) is fixed with unbounded class and coclass for real fields.
The most extensive flexibility is revealed by fields of the types \(\mathrm{F,G,H}\) and \(\mathrm{b,d}\),
where any finite unbounded length \(\ell\ge 3\) can occur with variable class and coclass.
We expect that an actually infinite tower with \(\ell=\infty\) is impossible for \(\mathrm{Cl}_3{K}\simeq(3,3)\).
\end{remark}



\section{The Artin transfer pattern}
\label{s:ArtinPattern}

Let \(p\) be a prime number
and \(G\) be a pro-\(p\) group with finite abelianization \(G/G^\prime\),
more precisely, assume that the commutator subgroup \(G^\prime\)
is of index \((G:G^\prime)=p^v\) with an integer exponent \(v\ge 0\).

\begin{definition}
\label{dfn:Layers}
For each integer \(0\le n\le v\), let
\(\mathrm{Lyr}_n{G}:=\lbrace G^\prime\le H\le G\mid (G:H)=p^n\rbrace\)
be the \textit{\(n\)th layer} of normal subgroups of \(G\) containing \(G^\prime\).
\end{definition}


\begin{definition}
\label{dfn:TTTandTKT}
For any intermediate group \(G^\prime\le H\le G\), we denote by
\(T_{G,H}:\,G\to H/H^\prime\)
the \textit{Artin transfer} homomorphism from \(G\) to \(H/H^\prime\)
\cite[Dfn. 3.1]{Ma9},
and by
\(\tilde{T}_{G,H}:\,G/G^\prime\to H/H^\prime\)
the induced transfer.
\begin{enumerate}
\item
Let
\(\tau(G):=\lbrack\tau_0{G};\ldots;\tau_v{G}\rbrack\)
be the \textit{multi-layered transfer target type} (TTT) of \(G\), where
\(\tau_n{G}:=(H/H^\prime)_{H\in\mathrm{Lyr}_n{G}}\) for each \(0\le n\le v\).
\item
Let
\(\varkappa(G):=\lbrack\varkappa_0{G};\ldots;\varkappa_v{G}\rbrack\)
be the \textit{multi-layered transfer kernel type} (TKT) of \(G\), where
\(\varkappa_n{G}:=(\ker\tilde{T}_{G,H})_{H\in\mathrm{Lyr}_n{G}}\) for each \(0\le n\le v\).
\end{enumerate}
\end{definition}


\begin{definition}
\label{dfn:ArtinPattern}
The pair
\(\mathrm{AP}(G):=\left(\tau(G),\varkappa(G)\right)\)
is called the (restricted) \textit{Artin pattern} of \(G\).
\end{definition}


\begin{definition}
\label{dfn:IPADandIPOD}
The first order approximation
\(\tau^{(1)}{G}:=\lbrack\tau_0{G};\tau_1{G}\rbrack\)
of the TTT, resp.
\(\varkappa^{(1)}{G}:=\lbrack\varkappa_0{G};\varkappa_1{G}\rbrack\)
of the TKT,
is called the \textit{index-\(p\) abelianization data} (IPAD),
resp. \textit{index-\(p\) obstruction data} (IPOD), of \(G\).
\end{definition}


\begin{definition}
\label{dfn:IteratedIPAD}
\(\tau^{(2)}{G}:=\lbrack\tau_0{G};(\tau^{(1)}{H})_{H\in\mathrm{Lyr}_1{G}}\rbrack\)
is called \textit{iterated} IPAD \textit{of} \(2^{\text{nd}}\) \textit{order} of \(G\).
\end{definition}



\begin{remark}
\label{rmk:ArtinPattern}
For the \textit{complete} Artin pattern \(\mathrm{AP}_c(G)\) see
\cite[Dfn. 5.3]{Ma9}.
\begin{enumerate}
\item
Since the \(0\)th layer (top layer), \(\mathrm{Lyr}_0{G}=\lbrace G\rbrace\),
consists of the group \(G\) alone,
and \(T_{G,G}:\,G\to G/G^\prime\) is the natural projection onto the commutator quotient
with kernel \(\ker(T_{G,G})=G^\prime\),
we usually omit the trivial top layer \(\varkappa_0{G}\)
and identify the IPOD \(\varkappa^{(1)}{G}\) with the first layer \(\varkappa_1{G}\) of the TKT.
\item
In the case of an elementary abelianization of rank two, \((G:G^\prime)=p^2\),
we also identify the TKT \(\varkappa(G)\) with its first layer \(\varkappa_1{G}\),
since the \(2\)nd layer (bottom layer), \(\mathrm{Lyr}_2{G}=\lbrace G^\prime\rbrace\),
consists of the commutator subgroup \(G^\prime\) alone,
and the kernel of \(T_{G,G^\prime}:\,G\to G^\prime/G^{\prime\prime}\)
is always \textit{total}, that is \(\ker(T_{G,G^\prime})=G\),
according to the \textit{principal ideal theorem}
\cite{Fw}.
\end{enumerate}
\end{remark}



\section{All possible IPADs of 3-groups of type \((3,3)\)}
\label{s:AllIPADs}
\noindent
Since the abelian type invariants of certain IPAD components of an assigned \(3\)-group \(G\)
depend on the parity of the nilpotency class \(c\) or coclass \(r\),
a more economic notation,
which avoids the tedious distinction of the cases odd or even,
is provided by the following definition
\cite[\S\ 3]{Ma3}.

\begin{definition}
\label{dfn:NearlyHomocyclic}
For an integer \(n\ge 2\),
the \textit{nearly homocyclic abelian \(3\)-group} \(\mathrm{A}(3,n)\) of order \(3^n\)
is defined by its type invariants \((q+r,q)\hat{=}(3^{q+r},3^q)\),
where the quotient \(q\ge 1\) and the remainder \(0\le r<2\)
are determined uniquely by the Euclidean division \(n=2q+r\).
Two degenerate cases are included by putting
\(\mathrm{A}(3,1):=(1)\hat{=}(3)\) the cyclic group \(C_3\) of order \(3\), and
\(\mathrm{A}(3,0):=(0)\hat{=}1\) the trivial group of order \(1\).
\end{definition}



In the following theorem and in the whole remainder of the article,
we use the \textit{identifiers} of finite \(3\)-groups up to order \(3^8\)
as they are defined in the SmallGroups Library
\cite{BEO1,BEO2}.
They are of the shape \(\langle\ order,\ counter\ \rangle\),
where the counter is motivated by the way
how the output of descendant computations is arranged
in the \(p\)-group generation algorithm
by Newman
\cite{Nm}
and O'Brien
\cite{Ob}.

\begin{theorem}
\label{thm:AllIPADs}
(Complete classification of all IPADs with \(\tau_0\simeq (3,3)\)
\cite{Ma3})
\noindent
Let \(G\) be a pro-\(3\) group
with abelianization \(G/G^\prime\) of type \((3,3)\)
and metabelianization \(\mathfrak{M}=G/G^{\prime\prime}\)
of nilpotency class \(c=\mathrm{cl}(\mathfrak{M})\ge 2\),
defect \(0\le k=k(\mathfrak{M})\le 1\),
and coclass \(r=\mathrm{cc}(\mathfrak{M})\ge 1\).
Assume that \(\mathfrak{M}\) does not belong to
the finitely many exceptions in the list below.
Then the IPAD \(\tau^{(1)}{G}=\lbrack\tau_0{G};\tau_1{G}\rbrack\) of \(G\)
in terms of nearly homocyclic abelian \(3\)-groups
is given by
\begin{equation}
\label{eqn:PolCoPol}
\begin{aligned}
\tau_0{G} &= (1^2);\\
\tau_1{G} &= \left(\overbrace{\mathrm{A}(3,c-k)}^{\text{polarization}},
                   \overbrace{\mathrm{A}(3,r+1)}^{\text{co-polarization}},T_3,T_4\right),
\end{aligned}
\end{equation}
\noindent
where the polarized first component of \(\tau_1{G}\) depends on the class \(c\) and defect \(k\),
the co-polarized second component increases with the coclass \(r\),
and the third and fourth component are completely stable for \(r\ge 3\)
but depend on the coclass tree of \(\mathfrak{M}\) for \(1\le r\le 2\) in the following manner
\begin{equation}
\label{eqn:Component3And4}
(T_3,T_4)=
\begin{cases}
(\mathrm{A}(3,r+1)^2)   & \text{ if } r=2,\ \mathfrak{M}\in\mathcal{T}^2\langle 243,8\rangle \text{ or } r=1,\\
(1^3,\mathrm{A}(3,r+1)) & \text{ if } r=2,\ \mathfrak{M}\in\mathcal{T}^2\langle 243,6\rangle,\\
((1^3)^2)               & \text{ if } r=2,\ \mathfrak{M}\in\mathcal{T}^2\langle 243,3\rangle \text{ or } r\ge 3.
\end{cases}
\end{equation}
\noindent
Anomalies of finitely many, precisely \(13\), exceptional groups are summarized in the following list.
\begin{equation}
\label{eqn:Anomalies}
\begin{aligned}
 \tau_1{G} = \left((1)^4\right)          & \text{ for } \mathfrak{M} \simeq \langle 9,2\rangle,\ c=1,\ r=1, \\
 \tau_1{G} = \left(1^2,(2)^3\right)      & \text{ for } \mathfrak{M} \simeq \langle 27,4\rangle,\ c=2,\ r=1, \\
 \tau_1{G} = \left(1^3,(1^2)^3\right)    & \text{ for } \mathfrak{M} \simeq \langle 81,7\rangle,\ c=3,\ r=1, \\
 \tau_1{G} = \left((1^3)^3,21\right)     & \text{ for } \mathfrak{M} \simeq \langle 243,4\rangle,\ c=3,\ r=2, \\
 \tau_1{G} = \left(1^3,(21)^3\right)     & \text{ for } \mathfrak{M} \simeq \langle 243,5\rangle,\ c=3,\ r=2, \\
 \tau_1{G} = \left((1^3)^2,(21)^2\right) & \text{ for } \mathfrak{M} \simeq \langle 243,7\rangle,\ c=3,\ r=2, \\
 \tau_1{G} = \left((21)^4\right)         & \text{ for } \mathfrak{M} \simeq \langle 243,9\rangle,\ c=3,\ r=2, \\
 \tau_1{G} = \left((1^3)^3,21\right)     & \text{ for } \mathfrak{M} \simeq \langle 729,44\ldots 47\rangle,\ c=4,\ r=2, \\
 \tau_1{G} = \left((21)^4\right)         & \text{ for } \mathfrak{M} \simeq \langle 729,56\ldots 57\rangle,\ c=4,\ r=2. \\
\end{aligned}
\end{equation}
\end{theorem}

The \textit{polarization} and the \textit{co-polarization} we had in our mind
when we spoke about a \textit{bi-polarization} in
\cite[Dfn. 3.2, p. 430]{Ma4}.
Meanwhile, we have provided yet another proof for the existence of
\textit{stable} and \textit{polarized} IPAD components with the aid of a \textit{natural partial order}
on the Artin transfer patterns distributed over a descendant tree
\cite[Thm. 6.1--6.2]{Ma9}.

\begin{proof}
Equations
(\ref{eqn:PolCoPol})
and
(\ref{eqn:Component3And4})
are a succinct form of information which summarizes
all statements about the first TTT layer \(\tau_1{G}\)
in the formulas (19), (20) and (22) of
\cite[Thm. 3.2, p. 291]{Ma7}
omitting the claims on the second TTT layer \(\tau_2{G}\).
Here we do not need the restrictions arising from lower bounds for the nilpotency class
\(c=\mathrm{cl}(\mathfrak{M})\)
in the cited theorem,
since the remaining cases for small values of \(c\) can be taken from
\cite[Thm. 3.1, p. 290]{Ma7},
with the exception of the following \(13\) anomalies in formula
(\ref{eqn:Anomalies}):

The \textit{abelian} group \(\langle 9,2\rangle\simeq (3,3)\),
the \textit{extra special} group \(\langle 27,4\rangle\),
and the group \(\langle 81,7\rangle\simeq\mathrm{Syl}_3(A_9)\)
do not fit into the general rules for \(3\)-groups of coclass \(1\).
These three groups appear in the top region of the tree diagram in the Figures
\ref{fig:AbsFreqTyp33TreeCc1}
and
\ref{fig:MinDiscTyp33TreeCc1}.

The four \textit{sporadic} groups \(\langle 243,n\rangle\) with \(n\in\lbrace 4,5,7,9\rbrace\)
and the six \textit{sporadic} groups \(\langle 729,n\rangle\) with \(n\in\lbrace 44,\ldots,47,56,57\rbrace\)
do not belong to any coclass-\(2\) tree,
as shown in Figure
\ref{fig:SporCc2},
whence the conditions in equation
(\ref{eqn:Component3And4})
cannot be applied to them.

On the other hand,
there is no need to list the groups
\(\langle 27,3\rangle\) and \(\langle 81,8\ldots 10\rangle\)
in formula (14),
the groups \(\langle 243,n\rangle\) with \(n\in\lbrace 3,6,8\rbrace\)
in formula (15),
and the groups \(\langle 729,n\rangle\) with \(n\in\lbrace 34,\ldots,39\rbrace\)
in formula (16) of
\cite[Thm. 3.1, p. 290]{Ma7},
since they perfectly fit into the general pattern.
\end{proof}



\begin{remark}
\label{rmk:AllIPADs}
The reason why we exclude the second TTT layer \(\tau_2{G}\) from Theorem
\ref{thm:AllIPADs},
while it is part of
\cite[Thm. 3.1--3.2, pp. 290--291]{Ma7},
is that we want to reduce the exceptions of the general pattern to a \textit{finite} list,
whereas the \textit{irregular case}
of the abelian quotient invariants of the commutator subgroup \(G^\prime\),
which forms the single component of \(\tau_2{G}\),
occurs for each even value of the coclass \(r=\mathrm{cc}(\mathfrak{M})\equiv 0\pmod{2}\)
and thus \textit{infinitely} often.
\end{remark}



\begin{theorem}
\label{thm:Contestants}
(Finiteness of the batch of contestants for the second \(p\)-class group \(\mathfrak{M}=\mathrm{G}_p^2{K}\)) \\
If \(p=3\), \(\tau_0=(1^2)\), and
\(\tau_1\) denotes an assigned family \((\tau_1(i))_{1\le i\le 4}\) of four abelian type invariants,
then the set \(\mathrm{Cnt}_p^2(\tau_0,\tau_1)\) of all
(isomorphism classes of) finite metabelian \(p\)-groups \(\mathfrak{M}\)
such that \(\tau_0{\mathfrak{M}}=\mathfrak{M}/\mathfrak{M}^\prime\simeq\tau_0\)
and \(\tau_1{\mathfrak{M}}=(H/H^\prime)_{H\in\mathrm{Lyr}_1{\mathfrak{M}}}\simeq\tau_1\)
is finite.
\end{theorem}

\begin{proof}
We have \(\mathrm{Cnt}_p^2(\tau_0,\tau_1)=\emptyset\), when \(\tau_1\) is \textit{malformed}
\cite[Dfn. 5.1, p. 294]{Ma7}.
For \(p=3\) and \(\tau_0=(1^2)\), Theorem
\ref{thm:AllIPADs}
ensures the validity of the following general \textbf{Polarization Principle}:
There exist a few components of a non-malformed family \(\tau_1\)
which determine the nilpotency class \(c:=\mathrm{cl}(\mathfrak{M})\) and the coclass \(r:=\mathrm{cc}(\mathfrak{M})\)
of a finite metabelian \(p\)-group \(\mathfrak{M}\)
with \(\tau_1{\mathfrak{M}}=(H/H^\prime)_{H\in\mathrm{Lyr}_1{\mathfrak{M}}}\simeq\tau_1\).
Together with the Coclass Theorems
\cite[\S\ 5, p. 164, and Eqn. (10), p. 168]{Ma6},
the polarization principle proves the claim.
\end{proof}



\section{Tables and Figures of possible 3-groups \(\mathrm{G}_3^2{K}\) and \(\mathrm{G}_3^\infty{K}\)}
\label{s:TablesFigures}

\subsection{Tables}
\label{ss:Tables}
\noindent
In this article, we shall frequently deal with finite \(3\)-groups \(G\)
of huge orders \(\lvert G\rvert\ge 3^9\)
for which no identifiers are available in the SmallGroups database
\cite{BEO1,BEO2}.
For instance in Table
\ref{tbl:3GroupsG19Spor},
and in the Figures
\ref{fig:TreeH4Spor}
and
\ref{fig:TreeG19Spor}.
A work-around for these cases is provided by the \textit{relative identifiers}
of the ANUPQ (Australian National University \(p\)-Quotient) package
\cite{GNO}
which is implemented in our licence of the computational algebra system MAGMA
\cite{BCP,BCFS,MAGMA}
and in the open source system GAP
\cite{GAP}.

\begin{definition}
\label{dfn:RelId}
Let \(p\) be a prime number
and \(G\) be a finite \(p\)-group
with nuclear rank \(\nu\ge 1\)
\cite[eqn. (28), p. 178]{Ma6}
and immediate descendant numbers \(N_1,\ldots,N_\nu\)
\cite[eqn. (34), p. 180]{Ma6}.
Then we denote the \(i\)th \textit{immediate descendant} of \textit{step size} \(s\) of \(G\)
by the symbol
\begin{equation}
\label{eqn:RelId}
G-\#s;i
\end{equation}
\noindent
for each \(1\le s\le\nu\) and \(1\le i\le N_s\).


Recall that a group with nuclear rank \(\nu=0\)
is a \textit{terminal leaf} without any descendants.
\end{definition}



All numerical results in this article have been computed by means of the computational algebra system MAGMA
\cite{BCP,BCFS,MAGMA}.
The \(p\)-group generation algorithm by Newman
\cite{Nm}
and O'Brien
\cite{Ob}
was used for the recursive construction of descendant trees \(\mathcal{T}(R)\) of finite \(p\)-groups \(G\).
The tree root (starting group) \(R\) was taken to be
\(\langle 9,2\rangle\) for Table
\ref{tbl:3GroupsCc1}
and the Figures
\ref{fig:AbsFreqTyp33TreeCc1},
\ref{fig:MinDiscTyp33TreeCc1},
\ref{fig:SporCc2},
\(\langle 243,6\rangle\) for Table
\ref{tbl:3GroupsTreeQ}
and Figure
\ref{fig:TreeQSecE},
\(\langle 243,8\rangle\) for Table
\ref{tbl:3GroupsTreeU}
and Figure
\ref{fig:TreeUSecE},
\(\langle 243,4\rangle\) for Table
\ref{tbl:3GroupsH4Spor}
and Figure
\ref{fig:TreeH4Spor}, and
\(\langle 243,9\rangle\) for Table
\ref{tbl:3GroupsG19Spor}
and Figure
\ref{fig:TreeG19Spor}.
For computing group theoretic invariants of each tree vertex \(G\),
we implemented the Artin transfers \(T_{G,H}\)
from a finite \(p\)-group \(G\) of type \(G/G^\prime\simeq (p,p)\) to its maximal subgroups \(H\unlhd G\)
in a MAGMA program script as described in
\cite[\S\ 4.1]{Ma9}.



\subsection{Figures}
\label{ss:Figures}
\noindent
Basic definitions, facts, and notation concerning descendant trees of finite \(p\)-groups
are summarized briefly in
\cite[\S\ 2, pp. 410--411]{Ma4},
\cite{Ma4a}.
They are discussed thoroughly in the broadest detail in the initial sections of
\cite{Ma6}.
Trees are crucial for the recent theory of \(p\)-class field towers
\cite{Ma12,Ma15,Ma15b},
in particular for describing the mutual location of
\(\mathrm{G}_3^2{K}\) and \(\mathrm{G}_3^\infty{K}\).

Generally, the vertices of coclass trees in the Figures
\ref{fig:AbsFreqTyp33TreeCc1},
\ref{fig:MinDiscTyp33TreeCc1},
\ref{fig:TreeQSecE},
\ref{fig:TreeUSecE},
of the sporadic part of a coclass graph in Figure
\ref{fig:SporCc2},
and of the descendant trees in the Figures
\ref{fig:TreeH4Spor},
\ref{fig:TreeG19Spor}
represent isomorphism classes of finite \(3\)-groups.
Two vertices are connected by a directed edge \(G\to H\) if
\(H\) is isomorphic to the last lower central quotient \(G/\gamma_c(G)\),
where \(c=\mathrm{cl}(G)\) denotes the nilpotency class of \(G\),
and either \(\lvert G\rvert=3\lvert H\rvert\), that is,
\(\gamma_c(G)\) is cyclic of order \(3\),
or \(\lvert G\rvert=9\lvert H\rvert\), that is,
\(\gamma_c(G)\) is bicyclic of type \((3,3)\).
See also
\cite[\S\ 2.2, p. 410--411]{Ma4}
and
\cite[\S\ 4, p. 163--164]{Ma6}.

The vertices of the tree diagrams in Figure
\ref{fig:AbsFreqTyp33TreeCc1}
and
\ref{fig:MinDiscTyp33TreeCc1}
are classified by using various symbols:
\begin{enumerate}
\item
big full discs {\Large \(\bullet\)} represent metabelian groups \(\mathfrak{M}\)
with defect \(k(\mathfrak{M})=0\),
\item
small full discs {\footnotesize \(\bullet\)} represent metabelian groups \(\mathfrak{M}\)
with defect \(k(\mathfrak{M})=1\).
\end{enumerate}
\noindent
In the Figures
\ref{fig:TreeQSecE},
\ref{fig:TreeUSecE},
and
\ref{fig:SporCc2},
\begin{enumerate}
\item
big full discs {\Large \(\bullet\)} represent metabelian groups \(\mathfrak{M}\)
with bicyclic centre of type \((3,3)\) and defect \(k(\mathfrak{M})=0\)
\cite[\S\ 3.3.2, p. 429]{Ma4},
\item
small full discs {\footnotesize \(\bullet\)} represent metabelian groups \(\mathfrak{M}\)
with cyclic centre of order \(3\) and defect \(k(\mathfrak{M})=1\),
\item
small contour squares {\tiny \(\square\)} represent non-metabelian groups \(\mathfrak{G}\).
\end{enumerate}
\noindent
In the Figures
\ref{fig:TreeH4Spor}
and
\ref{fig:TreeG19Spor},
\begin{enumerate}
\item
big contour squares \(\square\) represent groups \(\mathfrak{G}\) with relation rank \(d_2\mathfrak{G})\le 3\),
\item
small contour squares {\tiny \(\square\)} represent groups \(\mathfrak{G}\) with relation rank \(d_2(\mathfrak{G})\ge 4\).
\end{enumerate}
\noindent
A symbol \(n\ast\) adjacent to a vertex
denotes the multiplicity of a batch of \(n\) siblings,
that is, immediate descendants sharing a common parent.
The groups of particular importance are labelled by a number in angles,
which is the identifier in the SmallGroups Library
\cite{BEO1,BEO2}
of GAP
\cite{GAP}
and MAGMA
\cite{MAGMA}.
We omit the orders, which are given on the left hand scale.
The IPOD \(\varkappa_1\)
\cite[Thm. 2.5, Tbl. 6--7]{Ma2},
in the bottom rectangle concerns all
vertices located vertically above.
The first, resp. second, component \(\tau_1(1)\), resp. \(\tau_1(2)\), of the IPAD
\cite[Dfn. 3.3, p. 288]{Ma7}
in the left rectangle
concerns vertices \(G\) on the same horizontal level with defect \(k(G)=0\).
The periodicity with length \(2\) of branches,
\(\mathcal{B}(j)\simeq\mathcal{B}(j+2)\) for \(j\ge 4\), resp. \(j\ge 7\),
sets in with branch \(\mathcal{B}(4)\), resp. \(\mathcal{B}(7)\),
having a root of order \(3^4\), resp. \(3^7\),
in Figure
\ref{fig:AbsFreqTyp33TreeCc1}
and
\ref{fig:MinDiscTyp33TreeCc1},
resp.
\ref{fig:TreeQSecE}
and
\ref{fig:TreeUSecE}.
The metabelian skeletons of the Figures
\ref{fig:TreeQSecE}
and
\ref{fig:TreeUSecE}
were drawn by Nebelung
\cite[p. 189 ff]{Ne},
the complete trees were given by Ascione and coworkers
\cite{AHL},
\cite[Fig. 4.8, p. 76, and Fig. 6.1, p. 123]{As}.



We define two kinds of \textit{arithmetically structured graphs} \(\mathcal{G}\) of finite \(p\)-groups
by mapping each vertex \(V\in\mathcal{G}\) of the graph to statistical number theoretic information,
e.g. the distribution of second \(p\)-class groups \(\mathfrak{M}=\mathrm{G}_p^2{K}\) or \(p\)-class tower groups \(G=\mathrm{G}_p^\infty{K}\),
with respect to a given kind of number fields \(K\),
for instance real quadratic fields \(K=K(d):=\mathbb{Q}(\sqrt{d})\) with discriminant \(d>0\).



\begin{definition}
\label{dfn:ASDT}
Let \(p\) be a prime
and \(\mathcal{G}\) be a subgraph of a descendant tree \(\mathcal{T}\) of finite \(p\)-groups.
\begin{itemize}
\item
The mapping
\begin{equation}
\label{eqn:MinDisc}
\mathrm{MD}: \mathcal{G}\to\mathbb{N}\cup\lbrace\infty\rbrace,\quad V\mapsto\inf\lbrace d\mid\mathrm{G}_p^2{K(d)}\simeq V\rbrace
\end{equation}
\noindent
is called the \textit{distribution} of \textit{minimal discriminants} on \(\mathcal{G}\).
\item
For an assigned upper bound \(B>0\),
the mapping
\begin{equation}
\label{eqn:AbsFreq}
\mathrm{AF}: \mathcal{G}\to\mathbb{N}\cup\lbrace 0\rbrace,\quad V\mapsto\#\lbrace d<B\mid\mathrm{G}_p^2{K(d)}\simeq V\rbrace
\end{equation}
\noindent
is called the \textit{distribution} of \textit{absolute frequencies} on \(\mathcal{G}\).
\end{itemize}
\noindent
For both mappings, the subset of the graph \(\mathcal{G}\)
consisting of vertices \(V\) with \(\mathrm{MD}(V)\ne\infty\), resp. \(\mathrm{AF}(V)\ne 0\),
is called the \textit{support} of the distribution.
The trivial values outside of the support will be ignored in the sequel.
\end{definition}

Whereas Figure
\ref{fig:AbsFreqTyp33TreeCc1}
displays an \(\mathrm{AF}\)-distribution,
the Figures
\ref{fig:MinDiscTyp33TreeCc1},
\ref{fig:TreeQSecE}
and
\ref{fig:TreeUSecE}
show \(\mathrm{MD}\)-distributions.
The Figures
\ref{fig:SporCc2},
\ref{fig:TreeH4Spor}
and
\ref{fig:TreeG19Spor}
contain both distributions simultaneously.



\section{\(3\)-class towers of quadratic fields and iterated IPADs of second order}
\label{s:QF}



\subsection{\(3\)-groups \(G\) of coclass \(\mathrm{cc}(G)=1\)}
\label{ss:Coclass1}
\noindent
Table
\ref{tbl:3GroupsCc1}
shows the designation of the transfer kernel type
\cite{Ne},
the IPOD \(\varkappa_1{G}\),
and the \textit{iterated multi-layered} IPAD of \textit{second order},
\begin{equation}
\label{eqn:MltLyrIPAD2ndOrd}
\tau^{(2)}_\ast{G}=\lbrack\tau_0{G};\lbrack\tau_0{H};\tau_1{H};\tau_2{H}\rbrack_{H\in\mathrm{Lyr}_1{G}}\rbrack,
\end{equation}
\noindent
for \(3\)-groups \(G\) of maximal class up to order \(\lvert G\rvert=3^8\),
characterized by the logarithmic order, \(\mathrm{lo}\), i.e. \(\mathrm{lo}(G):=\log_3\lvert G\rvert\),
and the SmallGroup identifier, \(\mathrm{id}\)
\cite{BEO1,BEO2}.

The groups in Table
\ref{tbl:3GroupsCc1}
are represented by vertices of the tree diagrams in Figure
\ref{fig:AbsFreqTyp33TreeCc1}
and
\ref{fig:MinDiscTyp33TreeCc1}.



\subsection{Real Quadratic Fields of Types \(\mathrm{a}.1\), \(\mathrm{a}.2\) and \(\mathrm{a}.3\)}
\label{ss:RQFa}
\noindent
Sound numerical investigations of real quadratic fields \(K=\mathbb{Q}(\sqrt{d})\) with fundamental discriminant \(d>0\)
started in \(1982\), when Heider and Schmithals
\cite{HeSm}
showed the first examples of a Galois cohomology structure of Moser's type \(\alpha\)
on unit groups of unramified cyclic cubic extensions, \(L/K\) which are dihedral of degree \(6\) over \(\mathbb{Q}\)
\cite[Prop. 4.2, p. 482]{Ma1},
and of IPODs \(\varkappa_1{K}\) with type \(\mathrm{a}.1\) (\(d=62\,501\)), type \(\mathrm{a}.2\) (\(d=72\,329\)), and type \(\mathrm{a}.3\) (\(d=32\,009\)),
in the notation of Nebelung
\cite{Ne}.
See Figure
\ref{fig:MinDiscTyp33TreeCc1}.

Our extension in \(1991\)
\cite{Ma0}
merely produced further examples for these occurrences of type \(\mathrm{a}\).
In the \(15\) years from \(1991\) to \(2006\) we consequently were convinced that this type
with at least three total \(3\)-principalizations is the only possible type of real quadratic fields.

The absolute frequencies in
\cite[Tbl. 2, p. 496]{Ma1}
and
\cite[Tbl. 6.1, p. 451]{Ma3},
which should be corrected by the Corrigenda in the Appendix,
and the extended statistics in Figure
\ref{fig:AbsFreqTyp33TreeCc1}
underpin the \textit{striking dominance} of type \(\mathrm{a}\).
The distribution of the second \(3\)-class groups \(\mathfrak{M}=\mathrm{G}_3^2{K}\) with the smallest order \(3^4\), resp. \(3^6\),
alone reaches \(79.7\%\) for the accumulated types \(\mathrm{a}.2\) and \(\mathrm{a}.3\) together, resp. \(6.4\%\) for type \(\mathrm{a}.1\).

So it is not astonishing that the first exception \(d=214\,712\) without any total \(3\)-principalizations
did not show up earlier than in \(2006\)
\cite[Tbl. 4, p. 498]{Ma1},
\cite[Tbl. 6.3, p. 452]{Ma3}.
See Figure
\ref{fig:SporCc2}.

\newpage


\renewcommand{\arraystretch}{1.2}

\begin{table}[ht]
\caption{IPOD \(\varkappa_1{G}\) and iterated IPAD \(\tau^{(2)}_\ast{G}\) of \(3\)-groups \(G\) of coclass \(\mathrm{cc}(G)=1\)}
\label{tbl:3GroupsCc1}
\begin{center}
\begin{tabular}{|c|c||cc||c|ccccc|}
\hline
    lo &                                           id & type & \(\varkappa_1{G}\) & \(\tau_0{G}\) &       & \(\tau_0{H}\) &   \(\tau_1{H}\) &     \(\tau_2{H}\) &               \\
\hline
 \(2\) &            \(2\)                             &  a.1 &           \(0000\) &       \(1^2\) & \(\lbrack\) &   \(1\) &           \(0\) &             \(0\) & \(\rbrack^4\) \\
\hline
\hline
 \(3\) &            \(3\)                             &  a.1 &           \(0000\) &       \(1^2\) & \(\lbrack\) & \(1^2\) &       \((1)^4\) &             \(0\) & \(\rbrack^4\) \\
 \hline
 \(3\) &            \(4\)                             &  A.1 &           \(1111\) &       \(1^2\) &             & \(1^2\) &       \((1)^4\) &             \(0\) &               \\
       &                                              &      &                    &               & \(\lbrack\) &   \(2\) &           \(1\) &             \(0\) & \(\rbrack^3\) \\
\hline
\hline
 \(4\) &   \(\mathbf{7}\)                             &  a.3 &           \(2000\) &       \(1^2\) &             & \(1^3\) &  \((1^2)^{13}\) &      \((1)^{13}\) &               \\
       &                                              &      &                    &               &             & \(1^2\) &     \((1^2)^4\) &             \(1\) &               \\
       &                                              &      &                    &               & \(\lbrack\) & \(1^2\) &   \(1^2,(2)^3\) &             \(1\) & \(\rbrack^2\) \\
\hline
 \(4\) &   \(\mathbf{8}\)                             &  a.3 &           \(2000\) &       \(1^2\) &             &  \(21\) &   \(1^2,(2)^3\) &         \((1)^4\) &               \\
       &                                              &      &                    &               &             & \(1^2\) &     \((1^2)^4\) &             \(1\) &               \\
       &                                              &      &                    &               & \(\lbrack\) & \(1^2\) &   \(1^2,(2)^3\) &             \(1\) & \(\rbrack^2\) \\
\hline
 \(4\) &            \(9\)                             &  a.1 &           \(0000\) &       \(1^2\) &             &  \(21\) &   \(1^2,(2)^3\) &         \((1)^4\) &               \\
       &                                              &      &                    &               & \(\lbrack\) & \(1^2\) &     \((1^2)^4\) &             \(1\) & \(\rbrack^3\) \\
\hline
 \(4\) &  \(\mathbf{10}\)                             &  a.2 &           \(1000\) &       \(1^2\) &             &  \(21\) &   \(1^2,(2)^3\) &         \((1)^4\) &               \\
       &                                              &      &                    &               & \(\lbrack\) & \(1^2\) &   \(1^2,(2)^3\) &             \(1\) & \(\rbrack^3\) \\
\hline
\hline
 \(5\) &           \(25\)                             &  a.3 &           \(2000\) &       \(1^2\) &             & \(2^2\) &      \((21)^4\) &  \(1^2,(2)^{12}\) &               \\
       &                                              &      &                    &               & \(\lbrack\) & \(1^2\) &  \(21,(1^2)^3\) &           \(1^2\) & \(\rbrack^3\) \\
\hline
 \(5\) &           \(26\)                             &  a.1 &           \(0000\) &                                                   \multicolumn{6}{|c|}{IPAD like id \(25\)} \\
 \(5\) &           \(27\)                             &  a.2 &           \(1000\) &                                                   \multicolumn{6}{|c|}{IPAD like id \(25\)} \\
\hline
 \(5\) &  \(28\ldots 30\)                             &  a.1 &           \(0000\) &       \(1^2\) &             &  \(21\) &      \((21)^4\) &     \(1^2,(2)^3\) &               \\
       &                                              &      &                    &               & \(\lbrack\) & \(1^2\) &  \(21,(1^2)^3\) &           \(1^2\) & \(\rbrack^3\) \\
\hline
\hline
 \(6\) &           \(95\)                             &  a.1 &           \(0000\) &       \(1^2\) &             &  \(32\) &  \(2^2,(31)^3\) &  \((21)^4,(3)^9\) &               \\
       &                                              &      &                    &               & \(\lbrack\) & \(1^2\) & \(2^2,(1^2)^3\) &            \(21\) & \(\rbrack^3\) \\
\hline
 \(6\) &  \(\mathbf{96}\)                             &  a.2 &           \(1000\) &                                                   \multicolumn{6}{|c|}{IPAD like id \(95\)} \\
 \(6\) & \(\mathbf{97/98}\)                      &  a.3 &           \(2000\) &                                                   \multicolumn{6}{|c|}{IPAD like id \(95\)} \\
\hline
 \(6\) & \(\mathbf{99\ldots 101}\)                    &  a.1 &           \(0000\) &       \(1^2\) &             & \(2^2\) &  \(2^2,(31)^3\) &  \((21)^4,(3)^9\) &               \\
       &                                              &      &                    &               & \(\lbrack\) & \(1^2\) & \(2^2,(1^2)^3\) &            \(21\) & \(\rbrack^3\) \\
\hline
\hline
 \(7\) &          \(386\)                             &  a.1 &           \(0000\) &       \(1^2\) &             & \(3^2\) &      \((32)^4\) & \(2^2,(31)^{12}\) &               \\
       &                                              &      &                    &               & \(\lbrack\) & \(1^2\) &  \(32,(1^2)^3\) &           \(2^2\) & \(\rbrack^3\) \\
\hline
 \(7\) &          \(387\)                             &  a.2 &           \(1000\) &                                                  \multicolumn{6}{|c|}{IPAD like id \(386\)} \\
 \(7\) &          \(388\)                             &  a.3 &           \(2000\) &                                                  \multicolumn{6}{|c|}{IPAD like id \(386\)} \\
\hline
 \(7\) & \(389\ldots 391\)                            &  a.1 &           \(0000\) &       \(1^2\) &             &  \(32\) &      \((32)^4\) & \(2^2,(31)^{12}\) &               \\
       &                                              &      &                    &               & \(\lbrack\) & \(1^2\) &  \(32,(1^2)^3\) &           \(2^2\) & \(\rbrack^3\) \\
\hline
\hline
 \(8\) &         \(2221\)                             &  a.1 &           \(0000\) &       \(1^2\) &             &  \(43\) &  \(3^2,(42)^3\) & \((32)^4,(41)^9\) &               \\
       &                                              &      &                    &               & \(\lbrack\) & \(1^2\) & \(3^2,(1^2)^3\) &            \(32\) & \(\rbrack^3\) \\
\hline
 \(8\) & \(\mathbf{2222}\)                            &  a.2 &           \(1000\) &                                                 \multicolumn{6}{|c|}{IPAD like id \(2221\)} \\
 \(8\) & \(\mathbf{2223/2224}\)                  &  a.3 &           \(2000\) &                                                 \multicolumn{6}{|c|}{IPAD like id \(2221\)} \\
\hline
 \(8\) & \(\mathbf{2225\ldots 2227}\)                 &  a.1 &           \(0000\) &       \(1^2\) &             & \(3^2\) &  \(3^2,(42)^3\) & \((32)^4,(41)^9\) &               \\
       &                                              &      &                    &               & \(\lbrack\) & \(1^2\) & \(3^2,(1^2)^3\) &            \(32\) & \(\rbrack^3\) \\
\hline
\end{tabular}
\end{center}
\end{table}



The most extensive computation of data concerning
unramified cyclic cubic extensions \(L/K\)
of the \(481\,756\) real quadratic fields \(K=\mathbb{Q}(\sqrt{d})\)
with discriminant \(0<d<10^9\) and \(3\)-class rank \(\varrho_3{K}=2\)
has been achieved by M. R. Bush in \(2015\)
\cite{Bu}.
In the following,
we focus on the partial results for \(3\)-class groups of type \((3,3)\),
since they extend our own results of \(2010\)
\cite{Ma1,Ma3}.

\begin{proposition}
\label{prp:a1a2a3}
(IPADs of fields with type \(\mathrm{a}\) up to \(d<10^9\)
\cite{Bu}) \\
In the range \(0<d<10^9\) with \(415\,698\) fundamental discriminants \(d\)
of real quadratic fields \(K=\mathbb{Q}(\sqrt{d})\) having \(3\)-class group of type \((3,3)\),
there exist precisely\\
\(208\,236\) cases (\(50.1\%\)) with IPAD \(\tau^{(1)}{K}=\lbrack 1^2;21,(1^2)^3\rbrack\),\\
\(122\,955\) cases (\(29.6\%\)) with IPAD \(\tau^{(1)}{K}=\lbrack 1^2;1^3,(1^2)^3\rbrack\),\\
\(26\,678\) cases (\(6.4\%\)) with IPAD \(\tau^{(1)}{K}=\lbrack 1^2;2^2,(1^2)^3\rbrack\), and\\
\(11\,780\) cases (\(2.8\%\)) with IPAD \(\tau^{(1)}{K}=\lbrack 1^2;32,(1^2)^3\rbrack\).
\end{proposition}

\begin{proof}
The results were computed with PARI/GP
\cite{PARI},
double-checked with MAGMA
\cite{MAGMA},
and kindly communicated to us by M. R. Bush, privately
\cite{Bu}.
\end{proof}



For establishing the connection between IPADs and IPODs we need the following bridge.

\begin{corollary}
\label{cor:a1a2a3}
(Associated IPODs of fields with type \(\mathrm{a}\))
\begin{enumerate}
\item
A real quadratic field \(K\) with IPAD \(\tau^{(1)}{K}=\lbrack 1^2;21,(1^2)^3\rbrack\)
has IPOD either \(\varkappa_1{K}=(1000)\) of type \(\mathrm{a}.2\) or \(\varkappa_1{K}=(2000)\) of type \(\mathrm{a}.3\).
\item
A real quadratic field \(K\) with IPAD \(\tau^{(1)}{K}=\lbrack 1^2;1^3,(1^2)^3\rbrack\)
has IPOD \(\varkappa_1{K}=(2000)\) of type \(\mathrm{a}.3\),
more precisely \(\mathrm{a}.3^\ast\), in view of the exceptional IPAD.
\item
A real quadratic field \(K\) with IPAD \(\tau^{(1)}{K}=\lbrack 1^2;2^2,(1^2)^3\rbrack\)
has IPOD \(\varkappa_1{K}=(0000)\) of type \(\mathrm{a}.1\).
\item
A real quadratic field \(K\) with IPAD \(\tau^{(1)}{K}=\lbrack 1^2;32,(1^2)^3\rbrack\)
has IPOD either \(\varkappa_1{K}=(1000)\) of type \(\mathrm{a}.2\) or \(\varkappa_1{K}=(2000)\) of type \(\mathrm{a}.3\).
\end{enumerate}
\end{corollary}

\begin{proof}
Here, we again make use of the selection rule
\cite[Thm. 3.5, p. 420]{Ma4}
that only every other branch of the tree \(\mathcal{T}^1\langle 3^2,2\rangle\)
is admissible for second \(3\)-class groups \(\mathfrak{M}=\mathrm{G}_3^2{K}\) of (real) quadratic fields \(K\).

According to Table
\ref{tbl:3GroupsCc1},
three (isomorphism classes of) groups \(G\)
share the common IPAD \(\tau^{(1)}{G}=\lbrack 1^2;21,(1^2)^3\rbrack\),
namely \(\langle 81,8\ldots 10\rangle\),
whereas the IPAD \(\tau^{(1)}{G}=\lbrack 1^2;1^3,(1^2)^3\rbrack\)
unambiguously leads to the group \(\langle 81,7\rangle\)
with IPOD \(\varkappa_1{G}=(2000)\).

In Theorem
\ref{thm:IPODa1Mainline}
we shall show that the mainline group \(\langle 81,9\rangle\)
cannot occur as the second \(3\)-class group
of a real quadratic field.
Among the remaining two possible groups,
\(\langle 81,8\rangle\) has IPOD \(\varkappa_1{G}=(2000)\) and
\(\langle 81,10\rangle\) has IPOD \(\varkappa_1{G}=(1000)\).

The IPAD \(\tau^{(1)}{G}=\lbrack 1^2;2^2,(1^2)^3\rbrack\)
leads to three groups \(\langle 729,99\ldots 101\rangle\)
with IPOD \(\varkappa_1{G}=(0000)\) and defect of commutativity \(k=1\)
\cite[\S\ 3.1.1, p. 412]{Ma4}.

Concerning the IPAD \(\tau^{(1)}{G}=\lbrack 1^2;32,(1^2)^3\rbrack\),
Table
\ref{tbl:3GroupsCc1}
yields four groups \(G\)
with SmallGroup identifiers \(\langle 729,95\ldots 98\rangle\).
The mainline group \(\langle 729,95\rangle\) is discouraged by Theorem
\ref{thm:IPODa1Mainline},
\(\langle 729,96\rangle\) has IPOD \(\varkappa_1{G}=(1000)\),
and the two groups \(\langle 729,97\ldots 98\rangle\) have IPOD \(\varkappa_1{G}=(2000)\).

By the Artin reciprocity law
\cite{Ar1,Ar2},
the Artin pattern \(\mathrm{AP}(K)\) of the field \(K\)
coincides with the Artin pattern \(\mathrm{AP}(\mathfrak{M})\)
of its second \(3\)-class group \(\mathfrak{M}=\mathrm{G}_3^2{K}\).
\end{proof}



\begin{remark}
\label{rmk:Statistics}
The huge statistical ensembles underlying the computations of Bush
\cite{Bu}
admit a prediction of sound and reliable tendencies in the population of the \lq\lq ground state\rq\rq.
If we compare the smaller range \(d<10^7\) in
\cite{Ma1}
with the extended range \(d<10^9\) in
\cite{Bu},
then we have a decrease

\(\frac{1382}{2576}\approx 53.6\%\) \(\searrow\) \(\frac{208236}{415698}\approx 50.1\%\)
by \(3.5\%\) for the union of types \(\mathrm{a}.2\) and \(\mathrm{a}.3\),

\noindent
and increases

\(\frac{698}{2576}\approx 27.1\%\) \(\nearrow\) \(\frac{122955}{415698}\approx 29.6\%\)
by \(2.5\%\) for type \(\mathrm{a}.3\)\({}^\ast\),
and

\(\frac{150}{2576}\approx 5.8\%\) \(\nearrow\) \(\frac{26678}{415698}\approx 6.4\%\)
by \(0.6\%\) for type \(\mathrm{a}.1\).

\noindent
Of course, the accumulation of all types \(\mathrm{a}.2\), \(\mathrm{a}.3\), and \(\mathrm{a}.3\)\({}^\ast\)
with absolute frequencies\\
\(1382+698=2080\),
resp.
\(208236+122955=331191\),
shows a resultant decrease

\(\frac{2080}{2576}\approx 80.7\%\) \(\searrow\) \(\frac{331191}{415698}\approx 79.7\%\)
by \(1.0\%\).

\noindent
For the union of the \lq\lq first excited states\rq\rq\ of types \(\mathrm{a}.2\) and \(\mathrm{a}.3\), we have a stagnation

\(\frac{72}{2576}\approx 2.8\%\) \(\approx\) \(\frac{11780}{415698}\approx 2.8\%\)
at the same percentage.

\noindent
Unfortunately, the exact absolute frequency of the ground state of type a.2, resp. type a.3,
is unknown for the extended range \(d<10^9\).
It could be computed using Theorem
\ref{thm:IPODaGS}.
However, meanwhile we succeeded in separating all states of type a.2 and type a.3 up to \(d<10^8\)
by immediately figuring out the \(3\)-principalization type with MAGMA V2.22-1. In
\cite{Ma14},
we compare the results of this most recent tour de force of computing with
asymptotic densities predicted by Boston, Bush and Hajir (communicated privately and yet unpublished, similar to
\cite{BBH}).
\end{remark}



\noindent
Figure
\ref{fig:AbsFreqTyp33TreeCc1}
visualizes \(3\)-groups of section \S\
\ref{ss:Coclass1}
which arise as \(3\)-class tower groups \(G=\mathrm{G}_3^\infty{K}\)
of real quadratic fields \(K=\mathbb{Q}(\sqrt{d})\), \(d>0\),
with \(3\)-principalization types \(\mathrm{a}.1\), \(\mathrm{a}.2\) and \(\mathrm{a}.3\)
and the corresponding \textit{absolute frequencies and percentages}
(relative frequencies with respect to the total number of \(415\,698\)
real quadratic fields with discriminants in the range \(0<d<B\) for \(B=10^9\))
which were given in Proposition
\ref{prp:a1a2a3}.



{\tiny

\begin{figure}[hb]
\caption{Distribution of absolute frequencies of \(\mathrm{G}_3^2{K}\) on the coclass tree \(\mathcal{T}^1\langle 9,2\rangle\)}
\label{fig:AbsFreqTyp33TreeCc1}

\input{AbsFreqTyp33TreeCc1}

\end{figure}
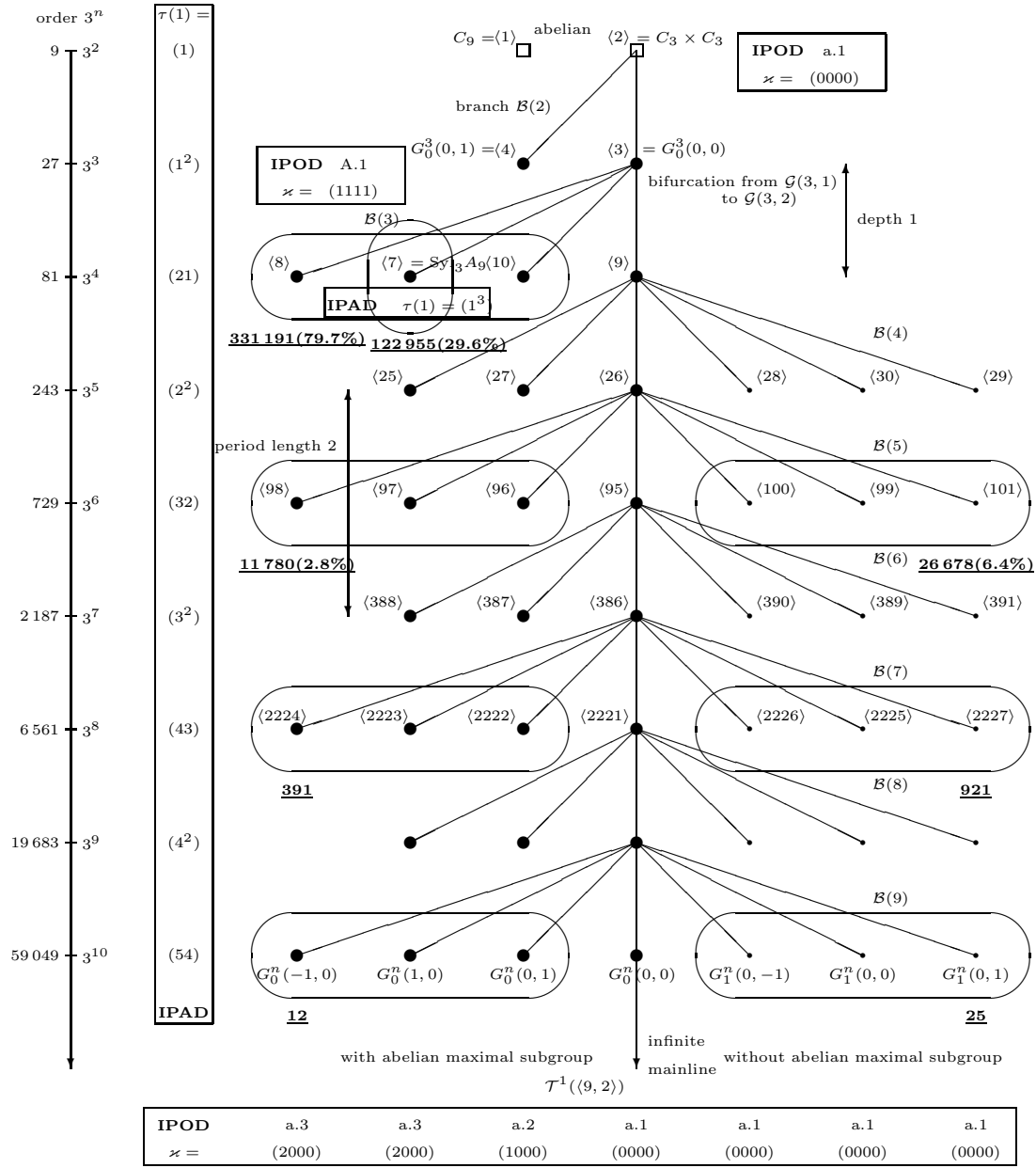

}



\noindent
Figure
\ref{fig:MinDiscTyp33TreeCc1}
visualizes \(3\)-groups of section \S\
\ref{ss:Coclass1}
which arise as \(3\)-class tower groups \(G=\mathrm{G}_3^\infty{K}\)
of real quadratic fields \(K=\mathbb{Q}(\sqrt{d})\), \(d>0\),
with \(3\)-principalization types \(\mathrm{a}.1\), \(\mathrm{a}.2\) and \(\mathrm{a}.3\)
and the corresponding \textit{minimal discriminants} in the sense of Definition
\ref{dfn:ASDT}.



{\tiny

\begin{figure}[hb]
\caption{Distribution of minimal discriminants for \(\mathrm{G}_3^2{K}\) on the coclass tree \(\mathcal{T}^1\langle 9,2\rangle\)}
\label{fig:MinDiscTyp33TreeCc1}

\input{MinDiscTyp33TreeCc1}

\end{figure}

}



As mentioned in
\cite{Ma7},
we have the following criterion for distinguishing subtypes of type \(\mathrm{a}\):

\begin{theorem}
\label{thm:IPODaGS}
(The \lq\lq ground state\rq\rq\ of type \(\mathrm{a}\)
\cite[\S\ 3.2.5, pp. 423--424]{Ma4}) \\
The second \(3\)-class groups \(\mathfrak{M}=\mathrm{G}_3^2{K}\) with the smallest order \(3^4\)
possessing type \(\mathrm{a}.2\) or \(\mathrm{a}.3\) can be separated
by means of the iterated IPAD of second order
\(\tau^{(2)}{\mathfrak{M}}=\lbrack\tau_0{\mathfrak{M}};\lbrack\tau_0{H};\tau_1{H}\rbrack_{H\in\mathrm{Lyr}_1{\mathfrak{M}}}\rbrack\).
\end{theorem}

\begin{proof}
This is essentially
\cite[Thm. 6.1, p. 296]{Ma7}
but can also be seen directly by comparing the column \(\tau_1{H}\) with the IPAD
for the rows with \(\mathrm{lo}=4\) and \(\mathrm{id}\in\lbrace 7,\ldots,10\rbrace\) in Table
\ref{tbl:3GroupsCc1}.
Here the column \(\tau_2{H}\), containing the second layer of the IPAD, does not permit a distinction.
\end{proof}



Unfortunately, we also must state a negative result:

\begin{theorem}
\label{thm:IPODaES}
(\lq\lq Excited states\rq\rq\ of type \(\mathrm{a}\)
\cite[\S\ 3.2.5, pp. 423--424]{Ma4}) \\
Even the multi-layered IPAD
\(\tau^{(2)}_\ast{\mathfrak{M}}=\lbrack\tau_0{\mathfrak{M}};\lbrack\tau_0{H};\tau_1{H};\tau_2{H}\rbrack_{H\in\mathrm{Lyr}_1{\mathfrak{M}}}\rbrack\)
of second order is unable to separate the second \(3\)-class groups \(\mathfrak{M}=\mathrm{G}_3^2{K}\)
with order \(3^6\) and type \(\mathrm{a}.2\) or \(\mathrm{a}.3\).
It is also unable to distinguish between the three candidates for \(\mathfrak{M}\) of type \(\mathrm{a}.1\),
and between the two candidates for \(\mathfrak{M}\) of type \(\mathrm{a}.3\), both for orders \(\lvert\mathfrak{M}\rvert\ge 3^6\).
\end{theorem}

\begin{proof}
This is a consequence of comparing both columns \(\tau_1{H}\) and \(\tau_2{H}\)
for the rows with \(\mathrm{lo}\in\lbrace 6,8\rbrace\) and \(\mathrm{id}\in\lbrace 95,\ldots,101\rbrace\),
resp. \(\mathrm{id}\in\lbrace 2221,\ldots,2227\rbrace\) in Table
\ref{tbl:3GroupsCc1}.
According to the selection rule
\cite[Thm. 3.5, p. 420]{Ma4},
only every other branch of the tree \(\mathcal{T}^1\langle 3^2,2\rangle\)
is admissible for second \(3\)-class groups \(\mathfrak{M}=\mathrm{G}_3^2{K}\) of (real) quadratic fields \(K\).
\end{proof}



\begin{theorem}
\label{thm:IPODa}
(Two-stage \(3\)-class towers of type \(\mathrm{a}\))
\noindent
For each (real) quadratic field \(K\)
with second \(3\)-class group \(\mathfrak{M}=\mathrm{G}_3^2{K}\) of maximal class
the \(3\)-class tower has exact length \(\ell_3{K}=2\).
\end{theorem}

\begin{proof}
Let \(G\) be a \(3\)-group of maximal class.
Then \(G\) is metabelian by
\cite[Thm. 3.7, proof, p. 421]{Ma4}
or directly by
\cite[Thm. 6, p. 26]{Bl}.
Suppose that \(H\) is a non-metabelian \(3\)-group of derived length \(\mathrm{dl}(H)\ge 3\)
such that \(H/H^{\prime\prime}\simeq G\).
According to
\cite[Thm. 5.4]{Ma9},
the Artin patterns \(\mathrm{AP}(H)\) and \(\mathrm{AP}(G)\) coincide,
in particular, both groups share a common IPOD \(\varkappa_1{H}=\varkappa_1{G}\),
which contains at least three total kernels, indicated by zeros, \(\varkappa_1=(\ast000)\)
\cite{Ma2}.
However, this is a contradiction already,
since any non-metabelian \(3\)-group,
which necessarily must be of coclass at least \(2\),
is descendant of one of the five groups
\(\langle 243,n\rangle\) with \(n\in\lbrace 3,4,6,8,9\rbrace\)
whose IPODs possess at most two total kernels,
and a descendant cannot have an IPOD with more total kernels than its parent, by
\cite[Thm. 5.2]{Ma9}.
Consequently, the cover \(\mathrm{cov}(G)\) of \(G\) in the sense of
\cite[Dfn. 5.1]{Ma10}
consists of the single element \(G\).

Finally, we apply this result to class field theory:
Since \(\mathfrak{M}=\mathrm{G}_3^2{K}\) is assumed to be of coclass \(\mathrm{cc}(\mathfrak{M})=1\),
we obtain \(G=\mathrm{G}_3^\infty{K}\in\mathrm{cov}(\mathfrak{M})=\lbrace \mathfrak{M}\rbrace\)
and the length of the \(3\)-class tower is given by \(\ell_3{K}=\mathrm{dl}(G)=\mathrm{dl}(\mathfrak{M})=2\).
\end{proof}

\begin{remark}
\label{rmk:IPODa}
To the very best of our knowledge, Theorem
\ref{thm:IPODa}
does not appear in the literature,
although we are convinced that it is well known to experts,
since it can also be proved purely group theoretically with the aid of a theorem by Blackburn
\cite[Thm. 4, p. 26]{Bl}.
Here we prefer to give a new proof which uses the structure of descendant trees.
\end{remark}



\begin{theorem}
\label{thm:IPODa1Mainline}
(The forbidden mainline of coclass \(1\))
\noindent
The mainline vertices of the coclass-\(1\) tree
cannot occur as second \(3\)-class groups \(\mathfrak{M}=\mathrm{G}_3^2{K}\) of (real) quadratic fields \(K\) (of type \(\mathrm{a}.1\)).
\end{theorem}

\begin{proof}
Since periodicity sets in with branch \(\mathcal{B}(4)\) in the Figures
\ref{fig:AbsFreqTyp33TreeCc1}
and
\ref{fig:MinDiscTyp33TreeCc1},
and MAGMA shows that the groups \(\langle 3^3,3\rangle\) and \(\langle 3^4,9\rangle\) have \(p\)-multiplicator rank \(4\),
all mainline vertices \(V\) must have \(p\)-multiplicator rank \(\mu(V)=4\)
and thus relation rank \(d_2{V}=\mu(V)=4\). However,
a real quadratic field \(K\) has torsion free Dirichlet unit rank \(r=1\)
and certainly does not contain the (complex) primitive third roots of unity.
According to the corrected version
\cite[Thm. 5.1]{Ma10}
of the Shafarevich theorem
\cite{Sh},
the relation rank \(d_2{G}\) of the \(3\)-tower group \(G=\mathrm{G}_3^\infty{K}\),
which coincides with the second \(3\)-class group \(\mathfrak{M}=\mathrm{G}_3^2{K}\) by Theorem
\ref{thm:IPODa},
is bounded by \(2=\varrho\le d_2{G}\le\varrho+r=2+1=3\),
where \(\varrho=2\) denotes the \(3\)-class rank of \(K\).
\end{proof}



\subsection{\(3\)-groups \(G\) of coclass \(\mathrm{cc}(G)=2\) arising from \(\langle 3^5,6\rangle\)}
\label{ss:TreeQ}
\noindent
Table
\ref{tbl:3GroupsTreeQ}
shows the designation of the transfer kernel type
\cite{SoTa,Ne},
the IPOD \(\varkappa_1{G}\),
and the iterated multi-layered IPAD of \(2\)nd order,
\[\tau^{(2)}_\ast{G}=\lbrack\tau_0{G};\lbrack\tau_0{H};\tau_1{H};\tau_2{H}\rbrack_{H\in\mathrm{Lyr}_1{G}}\rbrack,\]
for \(3\)-groups \(G\) on the coclass tree \(\mathcal{T}^2\langle 3^5,6\rangle\) up to order \(\lvert G\rvert=3^8\),
characterized by the logarithmic order, \(\mathrm{lo}\),
and the SmallGroup identifier, \(\mathrm{id}\),
\cite{BEO1,BEO2}.
To enable a brief reference for relative identifiers we put \(Q:=\langle 3^6,49\rangle\),
since this group was called the non-CF group \(Q\) by Ascione
\cite{AHL,As}.

The groups in Table
\ref{tbl:3GroupsTreeQ}
are represented by vertices of the tree diagram in Figure
\ref{fig:TreeQSecE}.



\renewcommand{\arraystretch}{1.2}

\begin{table}[hb]
\caption{IPOD \(\varkappa_1{G}\) and iterated IPAD \(\tau^{(2)}_\ast{G}\) of \(3\)-groups \(G\) on \(\mathcal{T}^2\langle 3^5,6\rangle\)}
\label{tbl:3GroupsTreeQ}
\begin{center}
\begin{tabular}{|c|c||cc||c|ccccc|}
\hline
    lo &                        id          & type & \(\varkappa_1{G}\) & \(\tau_0{G}\) &       & \(\tau_0{H}\) &            \(\tau_1{H}\) &            \(\tau_2{H}\) &               \\
\hline
 \(5\) &                     \(6\)          & c.18 &           \(0122\) &       \(1^2\) &             & \(1^3\) &      \((1^3)^4,(1^2)^9\) &           \((1^2)^{13}\) &               \\
       &                                    &      &                    &               & \(\lbrack\) &  \(21\) &           \(1^3,(21)^3\) &              \((1^2)^4\) & \(\rbrack^3\) \\
\hline
\hline
 \(6\) &                    \(48\)          & H.4  &           \(2122\) &       \(1^2\) &             & \(2^2\) &             \((21^2)^4\) &        \(1^3,(21)^{12}\) &               \\
       &                                    &      &                    &               &             & \(1^3\) & \(21^2,(1^3)^3,(1^2)^9\) &   \(1^3,(21)^3,(1^2)^9\) &               \\
       &                                    &      &                    &               & \(\lbrack\) &  \(21\) &          \(21^2,(21)^3\) &           \(1^3,(21)^3\) & \(\rbrack^2\) \\
\hline
 \(6\) &                  \(Q=49\)          & c.18 &           \(0122\) &                                                                   \multicolumn{6}{|c|}{IPAD like id \(48\)} \\
 \(6\) &                    \(50\)          & E.14 &           \(3122\) &                                                                   \multicolumn{6}{|c|}{IPAD like id \(48\)} \\
 \(6\) &                    \(51\)          & E.6  &           \(1122\) &                                                                   \multicolumn{6}{|c|}{IPAD like id \(48\)} \\
\hline
\hline
 \(7\) &                   \(284\)          & c.18 &           \(0122\) &       \(1^2\) &             & \(2^2\) &             \((21^2)^4\) &      \(21^2,(2^2)^{12}\) &               \\
       &                                    &      &                    &               &             & \(1^3\) &      \((21^2)^4,(1^2)^9\)&  \(21^2,(1^3)^3,(21)^9\) &               \\
       &                                    &      &                    &               & \(\lbrack\) &  \(21\) &          \(21^2,(21)^3\) &          \(21^2,(21)^3\) & \(\rbrack^2\) \\
\hline
 \(7\) &                   \(285\)          & c.18 &           \(0122\) &       \(1^2\) &             &  \(32\) &        \(2^21,(31^2)^3\) &      \((21^2)^4,(31)^9\) &               \\
       &                                    &      &                    &               &             & \(1^3\) & \(2^21,(\mathbf{1}^3)^3,(1^2)^9\) & \(21^2,(2^2)^3,(1^2)^9\) &      \\
       &                                    &      &                    &               & \(\lbrack\) &  \(21\) & \(2^21,(\mathbf{2}1)^3\) &         \(21^2,(2^2)^3\) & \(\rbrack^2\) \\
\hline
 \(7\) &          \(286/287\)          & H.4  &           \(2122\) &                                                                  \multicolumn{6}{|c|}{IPAD like id \(285\)} \\
 \(7\) &          \(\mathbf{288}\)          & E.6  &           \(1122\) &                                                                  \multicolumn{6}{|c|}{IPAD like id \(285\)} \\
 \(7\) & \(\mathbf{289/290}\)          & E.14 &           \(3122\) &                                                                  \multicolumn{6}{|c|}{IPAD like id \(285\)} \\
\hline
 \(7\) &                   \(291\)          & c.18 &           \(0122\) &       \(1^2\) &             & \(2^2\) &             \((21^2)^4\) &  \(21^2,(2^2)^3,(31)^9\) &               \\
       &                                    &      &                    &               &             & \(1^3\) &  \(21^2,(1^3)^3,(1^2)^9\)&  \(21^2,(12)^3,(1^2)^9\) &               \\
       &                                    &      &                    &               &             &  \(21\) &          \(21^2,(31)^3\) &          \(21^2,(21)^3\) &               \\
       &                                    &      &                    &               &             &  \(21\) &          \(21^2,(21)^3\) &          \(21^2,(21)^3\) &               \\
\hline
\hline
 \(8\) &                   \(613\)          & c.18 &           \(0122\) &       \(1^2\) &             &  \(32\) &        \(2^21,(31^2)^3\) &      \((2^21)^4,(32)^9\) &               \\
       &                                    &      &                    &               &             & \(1^3\) & \(2^21,(\mathbf{2}1^2)^3,(1^2)^9\) & \(2^21,(1^3)^3,(2^2)^3,(21)^6\) & \\
       &                                    &      &                    &               & \(\lbrack\) &  \(21\) & \(2^21,(\mathbf{3}1)^3\) &         \(2^21,(2^2)^3\) & \(\rbrack^2\) \\
\hline
 \(8\) &          \(614/615\)          & H.4  &           \(2122\) &                                                                  \multicolumn{6}{|c|}{IPAD like id \(613\)} \\
 \(8\) &          \(\mathbf{616}\)          & E.6  &           \(1122\) &                                                                  \multicolumn{6}{|c|}{IPAD like id \(613\)} \\
 \(8\) & \(\mathbf{617/618}\)          & E.14 &           \(3122\) &                                                                  \multicolumn{6}{|c|}{IPAD like id \(613\)} \\
\hline
\end{tabular}
\end{center}
\end{table}



\begin{theorem} 
\label{thm:E6E14GS}
(Smallest possible \(3\)-tower groups \(G=\mathrm{G}_3^\infty{K}\) of type \(\mathrm{E}.6\) or \(\mathrm{E}.14\)
\cite{Ma7}) \\
Let \(G\) be a finite \(3\)-group
with IPAD of first order \(\tau^{(1)}{G}=\lbrack\tau_0{G};\tau_1{G}\rbrack\),
where \(\tau_0{G}=1^2\) and \(\tau_1{G}=(32,1^3,(21)^2)\) is given in ordered form.

If the IPOD of \(G\) is of type \(\mathrm{E}.6\), \(\varkappa_1{G}=(1122)\),
resp. \(\mathrm{E}.14\), \(\varkappa_1{G}=(3122)\sim (4122)\),
then the IPAD of second order \(\tau^{(2)}{G}=\lbrack\tau_0{G};(\tau_0{H_i};\tau_1{H_i})_{1\le i\le 4}\rbrack\),
where the maximal subgroups of index \(3\) in \(G\) are denoted by \(H_1,\ldots,H_4\),
determines the isomorphism type of \(G\) in the following way:
\begin{enumerate}
\item
\(\tau^{(1)}{H_2}=\lbrack 1^3;2^21,(\mathbf{1}^3)^3,(1^2)^9\rbrack\) if and only if
\(\tau^{(1)}{H_i}=\lbrack 21;2^21,(\mathbf{2}1)^3\rbrack\) for \(i\in\lbrace 3,4\rbrace\) \\
if and only if \(G\simeq\langle 3^7,\mathbf{288}\rangle\),
resp. \(G\simeq\langle 3^7,\mathbf{289}\rangle\) or \(G\simeq\langle 3^7,\mathbf{290}\rangle\),
\item
\(\tau^{(1)}{H_2}=\lbrack 1^3;2^21,(\mathbf{2}1^2)^3,(1^2)^9\rbrack\) if and only if
\(\tau^{(1)}{H_i}=\lbrack 21;2^21,(\mathbf{3}1)^3\rbrack\) for \(i\in\lbrace 3,4\rbrace\) \\
if and only if \(G\simeq\langle 3^8,\mathbf{616}\rangle\),
resp. \(G\simeq\langle 3^8,\mathbf{617}\rangle\) or \(G\simeq\langle 3^8,\mathbf{618}\rangle\),
\end{enumerate}
\noindent
whereas the component \(\tau^{(1)}{H_1}=\lbrack 32;2^21,(31^2)^3\rbrack\) is fixed
and does not admit a distinction.
\end{theorem}

\begin{proof}
This is essentially
\cite[Thm. 6.2, pp. 297--298]{Ma7}.
It is also an immediate consequence of Table
\ref{tbl:3GroupsTreeQ},
which has been computed with MAGMA
\cite{MAGMA}.
As a termination criterion we can now use the more precise
\cite[Thm. 5.1]{Ma9}
instead of
\cite[Cor. 3.0.1, p. 771]{BuMa}.
\end{proof}



\noindent
Figure
\ref{fig:TreeQSecE}
visualizes \(3\)-groups which arise as second \(3\)-class groups \(\mathfrak{M}=\mathrm{G}_3^2{K}\)
of real quadratic fields \(K=\mathbb{Q}(\sqrt{d})\), \(d>0\),
with \(3\)-principalization types \(\mathrm{E}.6\) and \(\mathrm{E}.14\)
in section \S\
\ref{ss:TreeQ}
and the corresponding minimal discriminants.



{\tiny

\begin{figure}[ht]
\caption{Distribution of minimal discriminants for \(\mathrm{G}_3^2{K}\) on the coclass tree \(\mathcal{T}^2\langle 243,6\rangle\)}
\label{fig:TreeQSecE}

\input{TreeQSecE}

\end{figure}
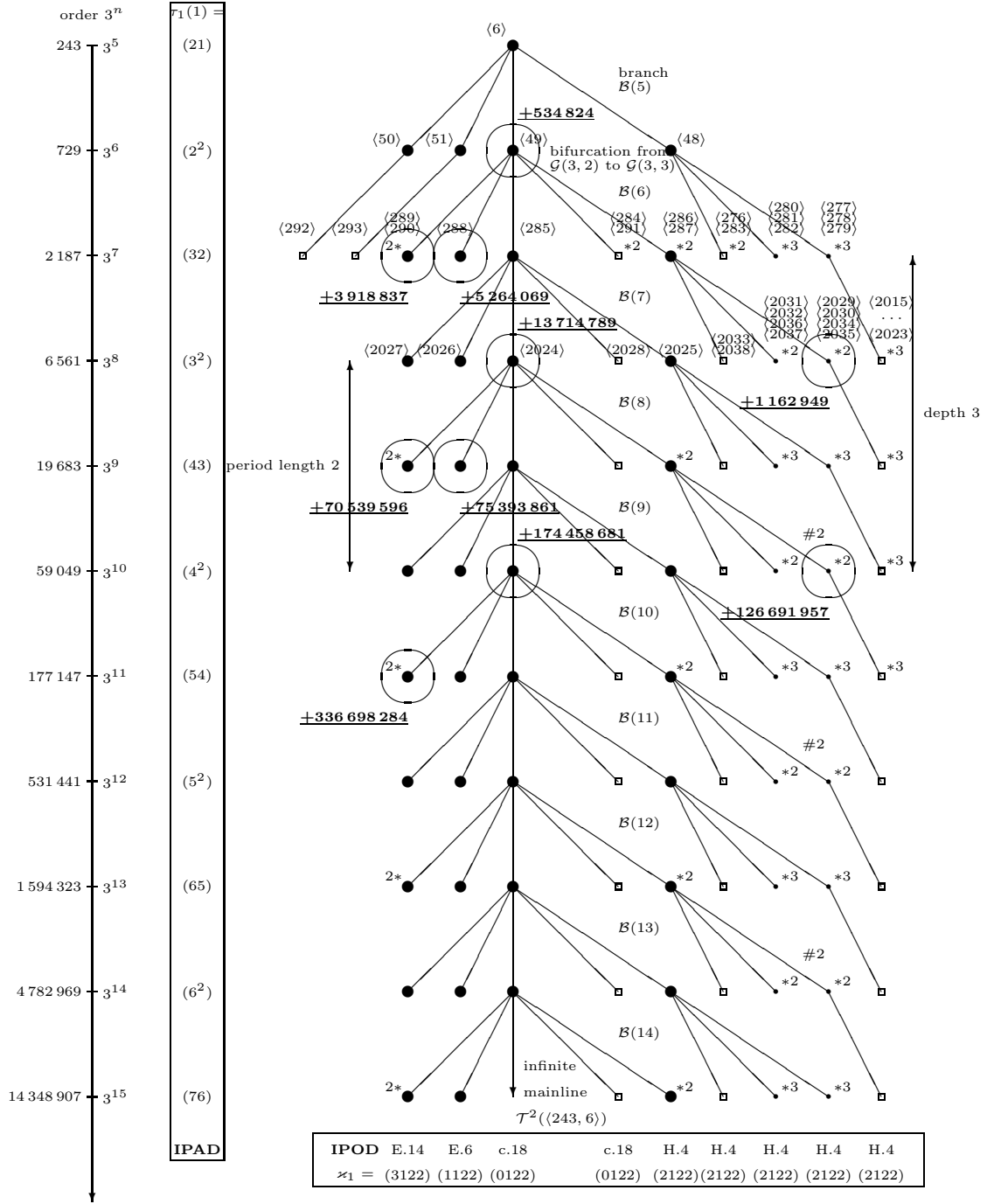

}



\subsection{Parametrized IPADs of second order for the coclass tree \(\mathcal{T}^2{\langle 3^5,6\rangle}\)}
\label{ss:IPAD2TreeQ}
\noindent
Let \(G\in\mathcal{T}^2{\langle 3^5,6\rangle}\) be a descendant of coclass \(\mathrm{cc}(G)=2\) of the root \(\langle 3^5,6\rangle\).
Denote by \(c:=\mathrm{cl}(G)\) the nilpotency class of \(G\),
by \(t:=\mathrm{dl}(G)-2\) the indicator of a three-stage group, and
by \(k:=k(G)\), resp \(k:=k(\pi G)\), the defect of commutativity of \(G\) itself if \(t=0\),
and of the metabelian parent \(\pi{G}\) if \(t=1\).

\begin{theorem}
\label{thm:IPAD2TreeQ}
In dependence on the parameters \(c\), \(t\) and \(k\),
the IPAD of second order of \(G\) has the form
\begin{equation}
\label{eqn:IPAD2TreeQ}
\begin{aligned}
\tau^{(2)}{G} = \lbrack\ 1^2;\ (\ \mathrm{A}(3,c-k-t);\ & \mathrm{A}(3,c-1-k-t)\times C_3,\ (\mathrm{B}(3,c-1-k-t)\times C_3)^3\ ), \\
                                                     (\ 21;\ & \mathrm{A}(3,c-1-k-t)\times C_3,\ (21)^3\ ), \\
                                                    (\ 1^3;\ & \mathrm{A}(3,c-1-k-t)\times C_3,\ (1^3)^3,\ (1^2)^9\ ), \\
                                                     (\ 21;\ & \mathrm{A}(3,c-1-k-t)\times C_3,\ (21)^3\ )\ \rbrack,
\end{aligned}
\end{equation}
where a \textbf{variant} of the \textbf{nearly homocyclic} abelian \(3\)-group of order \(n\ge 2\) in Definition
\ref{dfn:NearlyHomocyclic},
which can also be defined by \quad \(\mathrm{A}(3,0):=1\), \quad \(\mathrm{A}(3,1):=C_3\), \quad and 
\begin{equation}
\label{eqn:NrlHom}
\mathrm{A}(3,n):=
\begin{cases}
C_{3^{m+1}}\times C_{3^m} & \text{ if } n=2m+1 \text{ odd}, \\
C_{3^m}\times C_{3^m} & \text{ if } n=2m \text{ even},
\end{cases}
\end{equation}
is given by \quad \(\mathrm{B}(3,2):=C_{3^2}\) \quad and
\begin{equation}
\label{eqn:VarNrlHom}
\mathrm{B}(3,n):=
\begin{cases}
C_{3^{m+1}}\times C_{3^m} & \text{ if } n=2m+1 \text{ odd}, \\
C_{3^{m+2}}\times C_{3^m} & \text{ if } n=2m+2 \text{ even}.
\end{cases}
\end{equation}
\end{theorem}



\subsection{Number fields with IPOD of type \(\mathrm{E}.6\) or \(\mathrm{E}.14\)}
\label{ss:NFE6E14}
\noindent
Let \(K\) be a number field
with \(3\)-class group \(\mathrm{Cl}_3{K}\simeq C_3\times C_3\)
and first layer \(\mathrm{Lyr}_1{K}=\lbrace L_1,\ldots,L_4\rbrace\) of unramified abelian extensions.

\begin{theorem}
\label{thm:NFE6E14}
(Criteria for \(\ell_3{K}\in\lbrace 2,3\rbrace\).)
Let the IPOD of \(K\) be
of type \(\mathrm{E}.6\), \(\varkappa_1{K}\sim (1313)\),
resp. \(\mathrm{E}.14\), \(\varkappa_1{K}\sim (2313)\).
If \(\tau_1{K}\sim\left(\mathrm{A}(3,c),\ 21,\ 1^3,\ 21\right)\) with \(c\ge 4\), then
\begin{itemize}
\item
\(\ell_3{K}=2\) \(\iff\) \(\tau_1{L_3}\sim\left(\mathrm{A}(3,c-1)\times C_3,\ \mathbf{(1^3)^3},\ (1^2)^9\right)\) \\
\phantom{\(\ell_3{K}=2\)} \(\iff\) \(\tau_1{L_j}\sim\left(\mathrm{A}(3,c-1)\times C_3,\ \mathbf{(21)^3}\right)\) for \(j\in\lbrace 2,4\rbrace\),
\item
\(\ell_3{K}=3\) \(\iff\) \(\tau_1{L_3}\sim\left(\mathrm{A}(3,c-1)\times C_3,\ \mathbf{(21^2)^3},\ (1^2)^9\right)\) \\
\phantom{\(\ell_3{K}=3\)} \(\iff\) \(\tau_1{L_j}\sim\left(\mathrm{A}(3,c-1)\times C_3,\ \mathbf{(31)^3}\right)\) for \(j\in\lbrace 2,4\rbrace\).
\end{itemize}
\end{theorem}

\begin{proof}
Exemplarily, we conduct the proof for \(c=5\),
which is the most important situation for our computational applications. \\
Searching for the Artin pattern \(\mathrm{AP}_1=(\tau_1,\varkappa_1)\) with
\(\tau_1\sim\left(32,21,1^3,21\right)\) and
\(\varkappa_1\sim (1313)\), resp. \((2313)\),
in the descendant tree \(\mathcal{T}(R)\) with abelian root \(R:=\langle 3^2,2\rangle\simeq C_3\times C_3\),
unambiguously leads to the unique metabelian descendant with path
\(R\leftarrow\langle 3^3,3\rangle\leftarrow\langle 3^5,6\rangle\leftarrow\langle 3^6,49\rangle\leftarrow\langle 3^7,288\rangle=:\mathfrak{M}\)
for type \(\mathrm{E}.6\),
resp. two descendants \(\langle 3^7,289/290\rangle\) for type \(\mathrm{E}.14\).
The bifurcation at the vertex \(\langle 3^6,49\rangle\) with nuclear rank two
leads to a unique non-metabelian descendant with path
\(R\leftarrow\langle 3^3,3\rangle\leftarrow\langle 3^5,6\rangle\leftarrow\langle 3^6,49\rangle\leftarrow\langle 3^8,616\rangle=:G\)
for type \(\mathrm{E}.6\),
resp. two descendants \(\langle 3^8,617/618\rangle\) for type \(\mathrm{E}.14\).
The cover of \(\mathfrak{M}=\mathrm{G}_3^2{K}\) is non-trivial but very simple,
since it contains two elements \(\mathrm{cov}(\mathfrak{M})=\lbrace\mathfrak{M},G\rbrace\) only.
The decision whether \(\ell_3{K}=2\) and \(\mathrm{G}_3^3{K}=\mathfrak{M}\)
or \(\ell_3{K}=3\) and \(\mathrm{G}_3^3{K}=G\) requires the iterated IPADs of second order
\(\tau^{(2)}\) of \(\mathfrak{M}\) and \(G\),
which are listed in Table
\ref{tbl:3GroupsTreeQ}.
The general form \(\mathrm{A}(3,c-1)\times C_3\) of the component of \(\tau^{(2)}\) which corresponds to the commutator subgroup
\(\mathfrak{M}^\prime\simeq G^\prime/G^{\prime\prime}\) is a consequence of
\cite[Thm. 8.8, p.461]{Ma3},
since in terms of the nilpotency class \(c\) and coclass \(r=2\) of \(\mathfrak{M}\) we have \(m-2=c-1\) and \(e-2=r-1\).
\end{proof}


\noindent
The proof of Theorem
\ref{thm:NFE6E14},
immediately justifies the following conclusions for \(c\le 5\).

\begin{corollary}
\label{cor:NFE6E14}
Under the assumptions of Theorem
\ref{thm:NFE6E14},
the second and third \(3\)-class groups of \(K\) are given by their SmallGroups identifier
\cite{BEO1,BEO2},
if \(c\le 5\). Independently of \(\ell_3{K}\), \\
if \(c=4\), then \(\mathrm{G}_3^2{K}\simeq\langle 3^6,51\rangle\) for type \(\mathrm{E}.6\), resp. \(\langle 3^6,50\rangle\) for type \(\mathrm{E}.14\), and \\
if \(c=5\), then \(\mathrm{G}_3^2{K}\simeq\langle 3^7,288\rangle\) for type \(\mathrm{E}.6\), resp. \(\langle 3^7,289/290\rangle\) for type \(\mathrm{E}.14\). \\
In the case of a \(3\)-class tower \(\mathrm{F}_3^\infty{K}\) of length \(\ell_3{K}=3\), \\
if \(c=4\), then \(\mathrm{G}_3^3{K}\simeq\langle 3^7,293\rangle\) for type \(\mathrm{E}.6\), resp. \(\langle 3^7,292\rangle\) for type \(\mathrm{E}.14\), and \\
if \(c=5\), then \(\mathrm{G}_3^3{K}\simeq\langle 3^8,616\rangle\) for type \(\mathrm{E}.6\), resp. \(\langle 3^8,617/618\rangle\) for type \(\mathrm{E}.14\). \\
\end{corollary}


The range \(0<d<10^7\) of fundamental discriminants \(d\)
of real quadratic fields \(K=\mathbb{Q}(\sqrt{d})\) of type \(\mathrm{E}\),
which underlies Theorem
\ref{thm:RQFE6E14GS}
in this section, resp.
\ref{thm:RQFE8E9GS}
in the next section,
is just sufficient to prove that each of the possible groups \(G\) in Theorem
\ref{thm:E6E14GS},
resp.
\ref{thm:E8E9GS},
is actually realized by the \(3\)-tower group \(\mathrm{G}_3^\infty{K}\) of some field \(K\).

\begin{proposition}
\label{prp:RQFE6E14GS}
(Fields \(\mathbb{Q}(\sqrt{d})\) with IPOD of type \(\mathrm{E}.6\) or \(\mathrm{E}.14\) for \(0<d<10^7\)
\cite{Ma1},
\cite{Ma3}.)
\noindent
In the range \(0<d<10^7\) of fundamental discriminants \(d\)
of real quadratic fields \(K=\mathbb{Q}(\sqrt{d})\),
there exist precisely \(\mathbf{3}\), resp. \(\mathbf{4}\), cases
with \(3\)-principalization type \(\mathrm{E}.6\), \(\varkappa_1{K}\sim (1313)\),
resp. \(\mathrm{E}.14\), \(\varkappa_1{K}\sim (2313)\).
\end{proposition}

\begin{proof}
The results of
\cite[Tbl. 6.5, p. 452]{Ma3},
where the entry in the last column freq. should be \(28\) instead of \(29\) in the first row
and \(4\) instead of \(3\) in the fourth row,
were computed in \(2010\) by means of the free number theoretic computer algebra system PARI/GP
\cite{PARI}
using an implementation of our own principalization algorithm in a PARI script,
as described in detail in
\cite[\S\ 5, pp. 446--450]{Ma3}.
The accumulated frequency \(7\)
for the second and third row
was recently split into \(3\) and \(4\)
with the aid of the computational algebra system MAGMA
\cite{MAGMA}.
See also
\cite[Tbl. 4, p. 498]{Ma1}.
\end{proof}

\begin{remark}
\label{rmk:RQFE6E14GS}
The minimal discriminant \(d=5\,264\,069\) of real quadratic fields \(K=\mathbb{Q}(\sqrt{d})\) of type \(\mathrm{E}.6\),
resp. \(d=3\,918\,837\) of type \(\mathrm{E}.14\),
is indicated in boldface font adjacent to an oval surrounding the vertex, resp. batch of two vertices,
which represents the associated second \(3\)-class group \(G_3^2{K}\),
on the branch \(\mathcal{B}(6)\) of the coclass tree \(\mathcal{T}^2\langle 243,6\rangle\)
in Figure
\ref{fig:TreeQSecE}.
\end{remark}



\begin{theorem} 
\label{thm:RQFE6E14GS}
(\(3\)-Class towers \(\mathrm{F}_3^\infty{\mathbb{Q}(\sqrt{d})}\) with IPOD of type \(\mathrm{E}.6\) or \(\mathrm{E}.14\) for \(0<d<10^7\))
\noindent
Among the \(3\) real quadratic fields \(K=\mathbb{Q}(\sqrt{d})\) with IPOD of type \(\mathrm{E}.6\) in Proposition
\ref{prp:RQFE6E14GS},
\begin{itemize}
\item
the \(\mathbf{2}\) fields (\(\mathbf{67}\%\)) with discriminants
\[d\in\lbrace 5\,264\,069,\ 6\,946\,573\rbrace\]
have the unique \(3\)-class tower group 
\(G\simeq\langle 3^8,\mathbf{616}\rangle\)
and \(3\)-tower length \(\ell_3{K}=\mathbf{3}\),
\item
the \textbf{single} field (\(\mathbf{33}\%\)) with discriminant
\[d=7\,153\,097\]
has the unique \(3\)-class tower group 
\(G\simeq\langle 3^7,\mathbf{288}\rangle\)
and \(3\)-tower length \(\ell_3{K}=\mathbf{2}\).
\end{itemize}
\noindent
Among the \(4\) real quadratic fields \(K=\mathbb{Q}(\sqrt{d})\) with IPOD of type \(\mathrm{E}.14\) in Proposition
\ref{prp:RQFE6E14GS},
\begin{itemize}
\item
the \(\mathbf{3}\) fields (\(\mathbf{75}\%\)) with discriminants
\[d\in\lbrace 3\,918\,837,\ 8\,897\,192,\ 9\,991\,432\rbrace\]
have \(3\)-class tower group
\(G\simeq\langle 3^7,\mathbf{289}\rangle\)
or \(G\simeq\langle 3^7,\mathbf{290}\rangle\)
and \(3\)-tower length \(\ell_3{K}=\mathbf{2}\),
\item
the \textbf{single} field (\(\mathbf{25}\%\)) with discriminant
\[d=9\,433\,849\]
has \(3\)-class tower group 
\(G\simeq\langle 3^8,\mathbf{617}\rangle\)
or \(G\simeq\langle 3^8,\mathbf{618}\rangle\)
and \(3\)-tower length \(\ell_3{K}=\mathbf{3}\).
\end{itemize}
\noindent
\end{theorem}

\begin{proof}
Since all these real quadratic fields \(K=\mathbb{Q}(\sqrt{d})\) have
\(3\)-capitulation type \(\varkappa_1{K}=(1122)\) or \((3122)\) and \(1^{\mathrm{st}}\) IPAD  
\(\tau^{(1)}{K}=\lbrack 1^2;\mathbf{32},1^3,(21)^2\rbrack\),
and the \(4\) fields with \(d\in\lbrace 3\,918\,837,\ 7\,153\,097,\)
\(\ 8\,897\,192,\ 9\,991\,432\rbrace\)
have \(2^{\mathrm{nd}}\) IPAD
\[\tau_1{L_1}=(2^21,(31^2)^3),\ \tau_1{L_2}=(2^21,\mathbf{(1^3)^3},(1^2)^9),\ \tau_1{L_3}=(2^21,\mathbf{(21)^3}),\ \tau_1{L_4}=(2^21,\mathbf{(21)^3}),\]
whereas the \(3\) fields with \(d\in\lbrace 5\,264\,069,\ 6\,946\,573,\ 9\,433\,849\rbrace\)
have \(2^{\mathrm{nd}}\) IPAD
\[\tau_1{L_1}=(2^21,(31^2)^3),\ \tau_1{L_2}=(2^21,\mathbf{(21^2)^3},(1^2)^9),\ \tau_1{L_3}=(2^21,\mathbf{(31)^3}),\ \tau_1{L_4}=(2^21,\mathbf{(31)^3}),\]
the claim is a consequence of Theorem
\ref{thm:E6E14GS}.
\end{proof}

\begin{remark}
\label{rmk:9433849}
The computation of the \(3\)-principalization type \(\mathrm{E}.14\) of the field with \(d=9\,433\,849\)
resisted all attempts with MAGMA versions up to V2.21-7.
Due to essential improvements in the change from relative to absolute number fields,
made by the staff of the Computational Algebra Group at the University of Sydney,
it actually became feasible to figure it out with V2.21-8
\cite{MAGMA} for UNIX/LINUX machines
or V2.22-3 for any operating system.
\end{remark}



\subsection{\(3\)-groups \(G\) of coclass \(\mathrm{cc}(G)=2\) arising from \(\langle 3^5,8\rangle\)}
\label{ss:TreeU}
\noindent
Table
\ref{tbl:3GroupsTreeU}
shows the designation of the transfer kernel type,
the IPOD \(\varkappa_1{G}\),
and the iterated multi-layered IPAD of second order,
\[\tau^{(2)}_\ast{G}=\lbrack\tau_0{G};\lbrack\tau_0{H};\tau_1{H};\tau_2{H}\rbrack_{H\in\mathrm{Lyr}_1{G}}\rbrack,\]
for \(3\)-groups \(G\) on the coclass tree \(\mathcal{T}^2\langle 3^5,8\rangle\) up to order \(\lvert G\rvert=3^8\),
characterized by the logarithmic order, \(\mathrm{lo}\),
and the SmallGroup identifier, \(\mathrm{id}\)
\cite{BEO1,BEO2}.
To enable a brief reference for relative identifiers we put \(U:=\langle 3^6,54\rangle\),
since this group was called the non-CF group \(U\) by Ascione
\cite{AHL,As}.

The groups in Table
\ref{tbl:3GroupsTreeU}
are represented by vertices of the tree diagram in Figure
\ref{fig:TreeUSecE}.



\renewcommand{\arraystretch}{1.2}

\begin{table}[ht]
\caption{IPOD \(\varkappa_1{G}\) and iterated IPAD \(\tau^{(2)}_\ast{G}\) of \(3\)-groups \(G\) on \(\mathcal{T}^2\langle 3^5,8\rangle\)}
\label{tbl:3GroupsTreeU}
\begin{center}
\begin{tabular}{|c|c||cc||c|ccccc|}
\hline
    lo &                        id          & type & \(\varkappa_1{G}\) & \(\tau_0{G}\) &       & \(\tau_0{H}\) &            \(\tau_1{H}\) &            \(\tau_2{H}\) &               \\
\hline
 \(5\) &                     \(8\)          & c.21 &           \(2034\) &       \(1^2\) & \(\lbrack\) &  \(21\) &           \(1^3,(21)^3\) &              \((1^2)^4\) & \(\rbrack^4\) \\
\hline
\hline
 \(6\) &                    \(52\)          & G.16 &           \(2134\) &       \(1^2\) &             & \(2^2\) &             \((21^2)^4\) &        \(1^3,(21)^{12}\) &               \\
       &                                    &      &                    &               & \(\lbrack\) &  \(21\) &          \(21^2,(21)^3\) &           \(1^3,(21)^3\) & \(\rbrack^3\) \\
\hline
 \(6\) &                    \(53\)          & E.9  &           \(2434\) &                                                                   \multicolumn{6}{|c|}{IPAD like id \(52\)} \\
 \(6\) &                  \(U=54\)          & c.21 &           \(2034\) &                                                                   \multicolumn{6}{|c|}{IPAD like id \(52\)} \\
 \(6\) &                    \(55\)          & E.8  &           \(2234\) &                                                                   \multicolumn{6}{|c|}{IPAD like id \(52\)} \\
\hline
\hline
 \(7\) &          \(301/305\)          & G.16 &           \(2134\) &       \(1^2\) &             &  \(32\) &        \(2^21,(31^2)^3\) &      \((21^2)^4,(31)^9\) &               \\
       &                                    &      &                    &               & \(\lbrack\) &  \(21\) & \(2^21,(\mathbf{2}1)^3\) &         \(21^2,(2^2)^3\) & \(\rbrack^3\) \\
\hline
 \(7\) & \(\mathbf{302/306}\)          & E.9  &           \(2334\) &                                                                  \multicolumn{6}{|c|}{IPAD like id \(301\)} \\
 \(7\) &                   \(303\)          & c.21 &           \(2034\) &                                                                  \multicolumn{6}{|c|}{IPAD like id \(301\)} \\
 \(7\) &          \(\mathbf{304}\)          & E.8  &           \(2234\) &                                                                  \multicolumn{6}{|c|}{IPAD like id \(301\)} \\
\hline
 \(7\) &                   \(307\)          & c.21 &           \(2034\) &       \(1^2\) &             & \(2^2\) &             \((21^2)^4\) &      \(21^2,(2^2)^{12}\) &               \\
       &                                    &      &                    &               &             &  \(21\) &          \(21^2,(31)^3\) &          \(21^2,(21)^3\) &               \\
       &                                    &      &                    &               & \(\lbrack\) &  \(21\) &          \(21^2,(21)^3\) &          \(21^2,(21)^3\) & \(\rbrack^2\) \\
\hline
 \(7\) &                   \(308\)          & c.21 &           \(2034\) &       \(1^2\) &             & \(2^2\) &             \((21^2)^4\) &  \(21^2,(2^2)^3,(31)^9\) &               \\
       &                                    &      &                    &               &             &  \(21\) &          \(21^2,(31)^3\) &          \(21^2,(21)^3\) &               \\
       &                                    &      &                    &               & \(\lbrack\) &  \(21\) &          \(21^2,(21)^3\) &          \(21^2,(21)^3\) & \(\rbrack^2\) \\
\hline
\hline
 \(8\) &          \(619/623\)          & G.16 &           \(2134\) &       \(1^2\) &             &  \(32\) &        \(2^21,(31^2)^3\) &      \((2^21)^4,(32)^9\) &               \\
       &                                    &      &                    &               & \(\lbrack\) &  \(21\) & \(2^21,(\mathbf{3}1)^3\) &         \(2^21,(2^2)^3\) & \(\rbrack^3\) \\
\hline
 \(8\) & \(\mathbf{620/624}\)          & E.9  &           \(2334\) &                                                                  \multicolumn{6}{|c|}{IPAD like id \(619\)} \\
 \(8\) &                   \(621\)          & c.21 &           \(2034\) &                                                                  \multicolumn{6}{|c|}{IPAD like id \(619\)} \\
 \(8\) &          \(\mathbf{622}\)          & E.8  &           \(2234\) &                                                                  \multicolumn{6}{|c|}{IPAD like id \(619\)} \\
\hline
\end{tabular}
\end{center}
\end{table}



\begin{theorem} 
\label{thm:E8E9GS}
(Smallest possible \(3\)-tower groups \(G=\mathrm{G}_3^\infty{K}\) of type \(\mathrm{E}.8\) or \(\mathrm{E}.9\)
\cite{Ma7})

\noindent
Let \(G\) be a finite \(3\)-group
with IPAD of first order \(\tau^{(1)}{G}=\lbrack\tau_0{G};\tau_1{G}\rbrack\),
where \(\tau_0{G}=1^2\) and \(\tau_1{G}=(21,32,(21)^2)\) is given in ordered form.

If the IPOD of \(G\) is of type \(\mathrm{E}.8\), \(\varkappa_1{G}=(2234)\),
resp. \(\mathrm{E}.9\), \(\varkappa_1{G}=(2334)\sim (2434)\),
then the IPAD of second order \(\tau^{(2)}{G}=\lbrack\tau_0{G};(\tau_0{H_i};\tau_1{H_i})_{1\le i\le 4}\rbrack\),
where the maximal subgroups of index \(3\) in \(G\) are denoted by \(H_1,\ldots,H_4\),
determines the isomorphism type of \(G\) in the following way:
\begin{enumerate}
\item
\(\tau^{(1)}{H_i}=\lbrack 21;2^21,(\mathbf{2}1)^3\rbrack\) for \(i\in\lbrace 1,3,4\rbrace\)\\
if and only if \(G\simeq\langle 3^7,\mathbf{304}\rangle\),
resp. \(G\simeq\langle 3^7,\mathbf{302}\rangle\) or \(G\simeq\langle 3^7,\mathbf{306}\rangle\),
\item
\(\tau^{(1)}{H_i}=\lbrack 21;2^21,(\mathbf{3}1)^3\rbrack\) for \(i\in\lbrace 1,3,4\rbrace\)\\
if and only if \(G\simeq\langle 3^8,\mathbf{622}\rangle\),
resp. \(G\simeq\langle 3^8,\mathbf{620}\rangle\) or \(G\simeq\langle 3^8,\mathbf{624}\rangle\),
\end{enumerate}
\noindent
whereas the component \(\tau^{(1)}{H_2}=\lbrack 32;2^21,(31^2)^3\rbrack\) is fixed
and does not admit a distinction.
\end{theorem}

\begin{proof}
This is essentially
\cite[Thm. 6.3, pp. 298--299]{Ma7}.
It is also an immediate consequence of Table
\ref{tbl:3GroupsTreeU},
which has been computed with MAGMA
\cite{MAGMA}.
As a termination criterion we can now use the more precise
\cite[Thm. 5.1]{Ma9}
instead of
\cite[Cor. 3.0.1, p. 771]{BuMa}.
\end{proof}



\noindent
Figure
\ref{fig:TreeUSecE}
visualizes \(3\)-groups which arise as second \(3\)-class groups \(\mathfrak{M}=\mathrm{G}_3^2{K}\)
of real quadratic fields \(K=\mathbb{Q}(\sqrt{d})\), \(d>0\),
with \(3\)-principalization types \(\mathrm{E}.8\) and \(\mathrm{E}.9\)
in section \S\
\ref{ss:TreeU}
and the corresponding minimal discriminants.



{\tiny

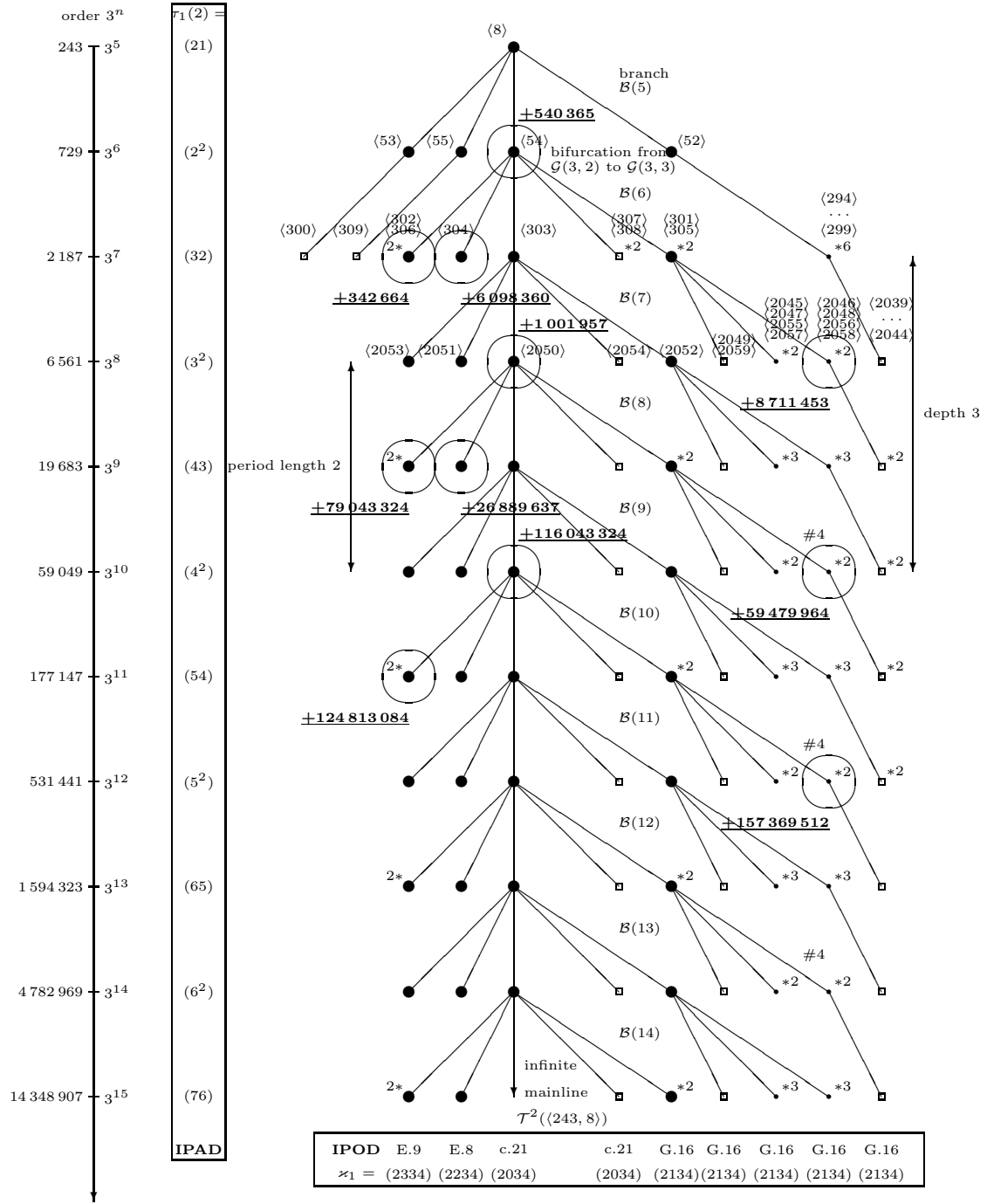
\begin{figure}[hb]
\caption{Distribution of minimal discriminants for \(\mathrm{G}_3^2{K}\) on the coclass tree \(\mathcal{T}^2\langle 243,8\rangle\)}
\label{fig:TreeUSecE}

\input{TreeUSecE}

\end{figure}

}



\subsection{Parametrized IPADs of second order for the coclass tree \(\mathcal{T}^2{\langle 3^5,8\rangle}\)}
\label{ss:IPAD2TreeU}
\noindent
Let \(G\in\mathcal{T}^2{\langle 3^5,8\rangle}\) be a descendant of coclass \(\mathrm{cc}(G)=2\) of the root \(\langle 3^5,8\rangle\).
Denote by \(c:=\mathrm{cl}(G)\) the nilpotency class of \(G\),
by \(t:=\mathrm{dl}(G)-2\) the indicator of a three-stage group, and
by \(k:=k(G)\), resp \(k:=k(\pi G)\), the defect of commutativity of \(G\) itself if \(t=0\),
and of the metabelian parent \(\pi{G}\) if \(t=1\).

\begin{theorem}
\label{thm:IPAD2TreeU}
In dependence on the parameters \(c\), \(t\) and \(k\),
the IPAD of second order of \(G\) has the form
\begin{equation}
\label{eqn:IPAD2TreeU}
\begin{aligned}
\tau^{(2)}{G} = \lbrack\ 1^2;\ (\ \mathrm{A}(3,c-k-t);\ & \mathrm{A}(3,c-1-k-t)\times C_3,\ (\mathrm{B}(3,c-1-k-t)\times C_3)^3\ ), \\
                                                (\ 21;\ & \mathrm{A}(3,c-1-k-t)\times C_3,\ (21)^3\ )^3\ \rbrack,
\end{aligned}
\end{equation}
where a \textbf{variant} \(\mathrm{B}(3,n)\) of the \textbf{nearly homocyclic} abelian \(3\)-group \(\mathrm{A}(3,n)\) of order \(n\ge 2\)
is defined as in Formula
\eqref{eqn:VarNrlHom}
of Theorem
\ref{thm:IPAD2TreeQ}.
\end{theorem}



\subsection{Number fields with IPOD of type \(\mathrm{E}.8\) or \(\mathrm{E}.9\)}
\label{ss:NFE8E9}
\noindent
Let \(K\) be a number field
with \(3\)-class group \(\mathrm{Cl}_3{K}\simeq C_3\times C_3\)
and first layer \(\mathrm{Lyr}_1{K}=\lbrace L_1,\ldots,L_4\rbrace\) of unramified abelian extensions.

\begin{theorem}
\label{thm:NFE8E9}
(Criteria for \(\ell_3{K}\in\lbrace 2,3\rbrace\).)
\noindent
Let the IPOD of \(K\) be
of type \(\mathrm{E}.8\), \(\varkappa_1{K}\sim (1231)\),
resp. \(\mathrm{E}.9\), \(\varkappa_1{K}\sim (2231)\).
If \(\tau_1{K}\sim\left(\mathrm{A}(3,c),\ 21,\ 21,\ 21\right)\) with \(c\ge 4\), then
\begin{itemize}
\item
\(\ell_3{K}=2\) \(\iff\) \(\tau_1{L_j}\sim\left(\mathrm{A}(3,c-1)\times C_3,\ \mathbf{(21)^3}\right)\) for \(2\le j\le 4\),
\item
\(\ell_3{K}=3\) \(\iff\) \(\tau_1{L_j}\sim\left(\mathrm{A}(3,c-1)\times C_3,\ \mathbf{(31)^3}\right)\) for \(2\le j\le 4\).
\end{itemize}
\end{theorem}

\begin{proof}
Exemplarily, we conduct the proof for \(c=5\),
which is the most important situation for our computational applications. \\
Searching for the Artin pattern \(\mathrm{AP}_1=(\tau_1,\varkappa_1)\) with
\(\tau_1\sim\left(32,21,21,21\right)\) and
\(\varkappa_1\sim (1231)\), resp. \((2231)\),
in the descendant tree \(\mathcal{T}(R)\) with abelian root \(R:=\langle 3^2,2\rangle\simeq C_3\times C_3\),
unambiguously leads to the unique metabelian descendant with path
\(R\leftarrow\langle 3^3,3\rangle\leftarrow\langle 3^5,8\rangle\leftarrow\langle 3^6,54\rangle\leftarrow\langle 3^7,304\rangle=:\mathfrak{M}\)
for type \(\mathrm{E}.8\),
resp. two descendants \(\langle 3^7,302/306\rangle\) for type \(\mathrm{E}.9\).
The bifurcation at the vertex \(\langle 3^6,54\rangle\) with nuclear rank two
leads to a unique non-metabelian descendant with path
\(R\leftarrow\langle 3^3,3\rangle\leftarrow\langle 3^5,8\rangle\leftarrow\langle 3^6,54\rangle\leftarrow\langle 3^8,622\rangle=:G\)
for type \(\mathrm{E}.8\),
resp. two descendants \(\langle 3^8,620/624\rangle\) for type \(\mathrm{E}.9\).
The cover of \(\mathfrak{M}=\mathrm{G}_3^2{K}\) is non-trivial but very simple,
since it contains two elements \(\mathrm{cov}(\mathfrak{M})=\lbrace\mathfrak{M},G\rbrace\) only.
The decision whether \(\ell_3{K}=2\) and \(\mathrm{G}_3^3{K}=\mathfrak{M}\)
or \(\ell_3{K}=3\) and \(\mathrm{G}_3^3{K}=G\) requires the iterated IPADs of second order
\(\tau^{(2)}\) of \(\mathfrak{M}\) and \(G\),
which are listed in Table
\ref{tbl:3GroupsTreeU}.
The general form \(\mathrm{A}(3,c-1)\times C_3\) of the component of \(\tau^{(2)}\) which corresponds to the commutator subgroup
\(\mathfrak{M}^\prime\simeq G^\prime/G^{\prime\prime}\) is a consequence of
\cite[Thm. 8.8, p.461]{Ma3},
since in terms of the nilpotency class \(c\) and coclass \(r=2\) of \(\mathfrak{M}\) we have \(m-2=c-1\) and \(e-2=r-1\).
\end{proof}


\noindent
The proof of Theorem
\ref{thm:NFE8E9},
immediately justifies the following conclusions for \(c\le 5\).

\begin{corollary}
\label{cor:NFE8E9}
Under the assumptions of Theorem
\ref{thm:NFE8E9},
the second and third \(3\)-class groups of \(K\) are given by their SmallGroups identifier
\cite{BEO1,BEO2},
if \(c\le 5\). Independently of \(\ell_3{K}\), \\
if \(c=4\), then \(\mathrm{G}_3^2{K}\simeq\langle 3^6,55\rangle\) for type \(\mathrm{E}.8\), resp. \(\langle 3^6,53\rangle\) for type \(\mathrm{E}.9\), and \\
if \(c=5\), then \(\mathrm{G}_3^2{K}\simeq\langle 3^7,304\rangle\) for type \(\mathrm{E}.8\), resp. \(\langle 3^7,302/306\rangle\) for type \(\mathrm{E}.9\). \\
In the case of a \(3\)-class tower \(\mathrm{F}_3^\infty{K}\) of length \(\ell_3{K}=3\), \\
if \(c=4\), then \(\mathrm{G}_3^3{K}\simeq\langle 3^7,309\rangle\) for type \(\mathrm{E}.8\), resp. \(\langle 3^7,300\rangle\) for type \(\mathrm{E}.9\), and \\
if \(c=5\), then \(\mathrm{G}_3^3{K}\simeq\langle 3^8,622\rangle\) for type \(\mathrm{E}.8\), resp. \(\langle 3^8,620/624\rangle\) for type \(\mathrm{E}.9\). \\
\end{corollary}


\begin{proposition}
\label{prp:RQFE8E9GS}
(Fields \(\mathbb{Q}(\sqrt{d})\) with IPOD of type \(\mathrm{E}.8\) or \(\mathrm{E}.9\) for \(0<d<10^7\)
\cite{Ma1},
\cite{Ma3})
\noindent
In the range \(0<d<10^7\) of fundamental discriminants \(d\)
of real quadratic fields \(K=\mathbb{Q}(\sqrt{d})\),
there exist precisely \(\mathbf{3}\), resp. \(\mathbf{11}\), cases
with \(3\)-principalization type \(\mathrm{E}.8\), \(\varkappa_1{K}\sim (1231)\),
resp. \(\mathrm{E}.9\), \(\varkappa_1{K}\sim (2231)\).
\end{proposition}

\begin{proof}
The results of
\cite[Tbl. 6.7, p. 453]{Ma3}
were computed in \(2010\) by means of PARI/GP
\cite{PARI}
using an implementation of our principalization algorithm, as described in
\cite[\S\ 5, pp. 446--450]{Ma3}.
The accumulated frequency \(14\)
in the last column \lq\lq freq.\rq\rq\
for the second and third row
was recently split into \(3\) and \(11\)
with the aid of MAGMA
\cite{MAGMA}.
See also
\cite[Tbl. 4, p. 498]{Ma1}.
\end{proof}

\begin{remark}
\label{rmk:RQFE8E9GS}
The minimal discriminant \(d=6\,098\,360\) of real quadratic fields \(K=\mathbb{Q}(\sqrt{d})\) of type \(\mathrm{E}.8\),
resp. \(d=342\,664\) of type \(\mathrm{E}.9\),
is indicated in boldface font adjacent to an oval surrounding the vertex, resp. batch of two vertices,
which represents the associated second \(3\)-class group \(G_3^2{K}\),
on the branch \(\mathcal{B}(6)\) of the coclass tree \(\mathcal{T}^2\langle 243,8\rangle\)
in Figure
\ref{fig:TreeUSecE}.
\end{remark}



\begin{theorem} 
\label{thm:RQFE8E9GS}
(\(3\)-Class towers \(\mathrm{F}_3^\infty{\mathbb{Q}(\sqrt{d})}\) with IPOD of type \(\mathrm{E}.8\) or \(\mathrm{E}.9\) for \(0<d<10^7\))
\noindent
Among the \(3\) real quadratic fields \(K=\mathbb{Q}(\sqrt{d})\) with IPOD of type \(\mathrm{E}.8\) in Proposition
\ref{prp:RQFE8E9GS},
\begin{itemize}
\item
the \(\mathbf{2}\) fields (\(\mathbf{67}\%\)) with discriminants
\[d\in\lbrace 6\,098\,360,\ 7\,100\,889\rbrace\] 
have the unique \(3\)-class tower group 
\(G\simeq\langle 3^8,\mathbf{622}\rangle\)
and \(3\)-tower length \(\ell_3{K}=\mathbf{3}\),
\item
the \textbf{single} field (\(\mathbf{33}\%\)) with discriminant
\[d=8\,632\,716\]
has the unique \(3\)-class tower group 
\(G\simeq\langle 3^7,\mathbf{304}\rangle\)
and \(3\)-tower length \(\ell_3{K}=\mathbf{2}\).
\end{itemize}
\noindent
Among the \(11\) real quadratic fields \(K=\mathbb{Q}(\sqrt{d})\) with IPOD of type \(\mathrm{E}.9\) in Proposition
\ref{prp:RQFE8E9GS},
\begin{itemize}
\item
the \(\mathbf{7}\) fields (\(\mathbf{64}\%\)) with discriminants
\[d\in\lbrace 342\,664,\ 1\,452\,185,\ 1\,787\,945,\ 4\,861\,720,\ 5\,976\,988,\ 8\,079\,101,\ 9\,674\,841\rbrace\]
have \(3\)-class tower group
\(G\simeq\langle 3^8,\mathbf{620}\rangle\)
or \(G\simeq\langle 3^8,\mathbf{624}\rangle\)
and \(3\)-tower length \(\ell_3{K}=\mathbf{3}\),
\item
the \(\mathbf{4}\) fields (\(\mathbf{36}\%\)) with discriminants
\[d\in\lbrace 4\,760\,877,\ 6\,652\,929,\ 7\,358\,937,\ 9\,129\,480\rbrace\]
have \(3\)-class tower group 
\(G\simeq\langle 3^7,\mathbf{302}\rangle\)
or \(G\simeq\langle 3^7,\mathbf{306}\rangle\)
and \(3\)-tower length \(\ell_3{K}=\mathbf{2}\).
\end{itemize}
\end{theorem}

\begin{proof}
Since all these real quadratic fields \(K=\mathbb{Q}(\sqrt{d})\) have
\(3\)-capitulation type \(\varkappa_1{K}=(2234)\) or \((2334)\) and \(1^{\mathrm{st}}\) IPAD  
\(\tau^{(1)}{K}=\lbrack 1^2;\mathbf{32},(21)^3\rbrack\),
and the \(5\) fields with \(d\in\lbrace 4\,760\,877,\ 6\,652\,929,\ 7\,358\,937,\)
\(\ 8\,632\,716,\ 9\,129\,480\rbrace\)
have \(2^{\mathrm{nd}}\) IPAD
\[\tau_1{L_1}=(2^21,(31^2)^3),\ \tau_1{L_2}=(2^21,\mathbf{(21)^3}),\ \tau_1{L_3}=(2^21,\mathbf{(21)^3}),\ \tau_1{L_4}=(2^21,\mathbf{(21)^3}),\]
whereas the \(9\) fields with
\(d\in\lbrace 342\,664,\ 1\,452\,185,\ 1\,787\,945,\ 4\,861\,720,\ 5\,976\,988,\ 6\,098\,360,\ 7\,100\,889,\)
\(\ 8\,079\,101,\ 9\,674\,841\rbrace\)
have \(2^{\mathrm{nd}}\) IPAD
\[\tau_1{L_1}=(2^21,(31^2)^3),\ \tau_1{L_2}=(2^21,\mathbf{(31)^3}),\ \tau_1{L_3}=(2^21,\mathbf{(31)^3}),\ \tau_1{L_4}=(2^21,\mathbf{(31)^3}),\]
the claim is a consequence of Theorem
\ref{thm:E8E9GS}.
\end{proof}

\begin{remark}
\label{rmk9674841}
The \(3\)-principalization type \(\mathrm{E}.9\) of the field with \(d=9\,674\,841\)
could not be computed with MAGMA versions up to V2.21-7.
Finally, we succeeded to figure it out by means of V2.21-8
\cite{MAGMA}.
\end{remark}



Figure
\ref{fig:SporCc2}
visualizes sporadic \(3\)-groups of section \S\
\ref{ss:SporadicCoclass2}
which arise as second \(3\)-class groups \(\mathfrak{M}=\mathrm{G}_3^2{K}\)
of real quadratic fields \(K=\mathbb{Q}(\sqrt{d})\), \(d>0\),
with \(3\)-principalization types \(\mathrm{D}.10\), \(\mathrm{D}.5\), \(\mathrm{G}.19\) and \(\mathrm{H}.4\)
and the corresponding minimal discriminants, resp. absolute frequencies, which are given in section \S\
\ref{ss:RQFH4Spor}
and \S\
\ref{ss:RQFG19Spor}.



{\tiny

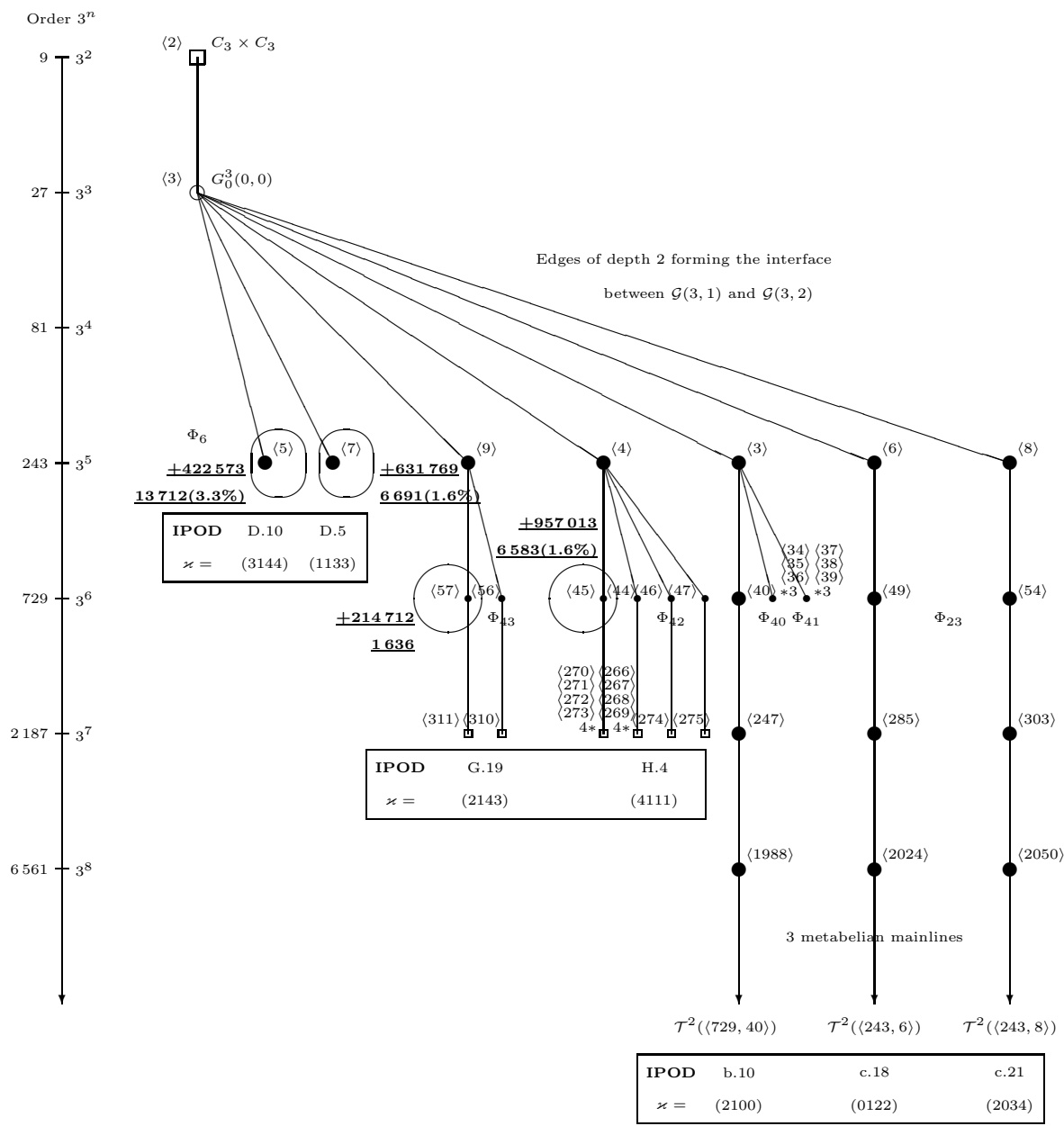
\begin{figure}[ht]
\caption{Distribution of second \(3\)-class groups \(\mathrm{G}_3^2{K}\) on the sporadic graph \(\mathcal{G}_0(3,2)\)}
\label{fig:SporCc2}

\input{SporCc2}

\end{figure}

}

















\subsection{Sporadic \(3\)-groups \(G\) of coclass \(\mathrm{cc}(G)=2\)}
\label{ss:SporadicCoclass2}

\renewcommand{\arraystretch}{1.1}

\begin{table}[ht]
\caption{IPOD \(\varkappa_1{G}\) and iterated IPAD \(\tau^{(2)}_\ast{G}\) of sporadic \(3\)-groups \(G\) of type H.4}
\label{tbl:3GroupsH4Spor}
\begin{center}
\begin{tabular}{|c|c||cc||c|ccccc|}
\hline
    lo &               id          & type & \(\varkappa_1{G}\) & \(\tau_0{G}\) &       & \(\tau_0{H}\) &            \(\tau_1{H}\) &          \(\tau_2{H}\) &               \\
\hline
 \(5\) &            \(4\)          &  H.4 &           \(4111\) &       \(1^2\) &             & \(1^3\) &      \((1^3)^4,(1^2)^9\) &         \((1^2)^{13}\) &               \\
       &                           &      &                    &               & \(\lbrack\) & \(1^3\) &   \(1^3,(21)^3,(1^2)^9\) &      \((1^2)^4,(2)^9\) & \(\rbrack^2\) \\
       &                           &      &                    &               &             &  \(21\) &           \(1^3,(21)^3\) &            \((1^2)^4\) &               \\
\hline
\hline
 \(6\) &         \(N=45\)          &  H.4 &           \(4111\) &       \(1^2\) &             & \(1^3\) & \(21^2,(1^3)^3,(1^2)^9\) &  \(1^3,(21)^3,(1^2)^9\) &              \\
       &                           &      &                    &               & \(\lbrack\) & \(1^3\) &       \(21^2,(21)^{12}\) &     \(21^2,(21)^{12}\) & \(\rbrack^2\) \\
       &                           &      &                    &               &             &  \(21\) &          \(21^2,(21)^3\) &       \(21^2,(2^2)^3\) &               \\
\hline
\hline
 \(7\) & \(\mathbf{270}\)          &  H.4 &           \(4111\) &       \(1^2\) &             & \(1^3\) &     \((21^2)^4,(1^2)^9\) & \(21^2,(1^3)^3,(21)^9\) &              \\
       &                           &      &                    &               & \(\lbrack\) & \(1^3\) &       \(21^2,(21)^{12}\) &     \(21^2,(21)^{12}\) & \(\rbrack^2\) \\
       &                           &      &                    &               &             &  \(21\) &          \(21^2,(21)^3\) &       \(21^2,(2^2)^3\) &               \\
\hline
 \(7\) & \(\mathbf{271/272}\) &  H.4 &           \(4111\) &       \(1^2\) &             & \(1^3\) & \(21^2,(1^3)^3,(1^2)^9\) & \(21^2,(2^2)^3,(1^2)^9\) &             \\
       &                           &      &                    &               & \(\lbrack\) & \(1^3\) &       \(21^2,(21)^{12}\) &     \(21^2,(21)^{12}\) & \(\rbrack^2\) \\
       &                           &      &                    &               &             &  \(21\) &          \(21^2,(31)^3\) &        \(21^2,(21)^3\) &               \\
\hline
 \(7\) & \(\mathbf{273}\)          &  H.4 &           \(4111\) &       \(1^2\) &             & \(1^3\) & \(21^2,(1^3)^3,(1^2)^9\) & \(21^2,(21)^3,(1^2)^9\)&               \\
       &                           &      &                    &               &             & \(1^3\) &       \(21^2,(21)^{12}\) &     \(21^2,(21)^{12}\) &               \\
       &                           &      &                    &               &             & \(1^3\) &     \((21^2)^4,(2^2)^9\) &        \((21^2)^{13}\) &               \\
       &                           &      &                    &               &             &  \(21\) &          \(21^2,(21)^3\) &        \(21^2,(21)^3\) &               \\
\hline
\hline
 \(8\) & \(605/\mathbf{606}\)  &  H.4 &           \(4111\) &       \(1^2\) &             & \(1^3\) &     \((21^2)^4,(1^2)^9\) & \(2^21,(1^3)^3,(2^2)^3,(21)^6\) &      \\
       &                           &      &                    &               & \(\lbrack\) & \(1^3\) &     \((21^2)^4,(2^2)^9\) &   \(2^21,(21^2)^{12}\) & \(\rbrack^2\) \\
       &                           &      &                    &               &             &  \(21\) &          \(21^2,(31)^3\) &       \(2^21,(2^2)^3\) &               \\
\hline
\end{tabular}
\end{center}
\end{table}



Table
\ref{tbl:3GroupsH4Spor}
shows the designation of the transfer kernel type,
the IPOD \(\varkappa_1{G}\),
and the iterated multi-layered IPAD of second order,
\[\tau^{(2)}_\ast{G}=\lbrack\tau_0{G};\lbrack\tau_0{H};\tau_1{H};\tau_2{H}\rbrack_{H\in\mathrm{Lyr}_1{G}}\rbrack,\]
for sporadic \(3\)-groups \(G\) of type H.4 up to order \(\lvert G\rvert=3^8\),
characterized by the logarithmic order, \(\mathrm{lo}\),
and the SmallGroup identifier, \(\mathrm{id}\)
\cite{BEO1,BEO2}.
To enable a brief reference for relative identifiers we put
\(N:=\langle 3^6,45\rangle\),
since this group was called the non-CF group \(N\) by Ascione
\cite{AHL,As}.

The groups in Table
\ref{tbl:3GroupsH4Spor}
are represented by vertices of the tree diagram in Figure
\ref{fig:TreeH4Spor}.



\noindent
Figure
\ref{fig:TreeH4Spor}
visualizes sporadic \(3\)-groups of section \S\
\ref{ss:SporadicCoclass2}
which arise as \(3\)-class tower groups \(G=\mathrm{G}_3^\infty{K}\)
of real quadratic fields \(K=\mathbb{Q}(\sqrt{d})\), \(d>0\),
with \(3\)-principalization type \(\mathrm{H}.4\)
and the corresponding  minimal discriminants, resp. absolute frequencies, in Theorem
\ref{thm:RQFH4Spor}
and
\ref{thm:IQFH4Spor}.

The tree is infinite, according to Bartholdi, Bush
\cite{BaBu}
and
\cite[Cor. 6.2, p. 301]{Ma7}.

For \(d=+2\,852\,733\) and \(d=-6\,583\), we can only give the \textbf{conjectural} location of \(G\).



{\tiny

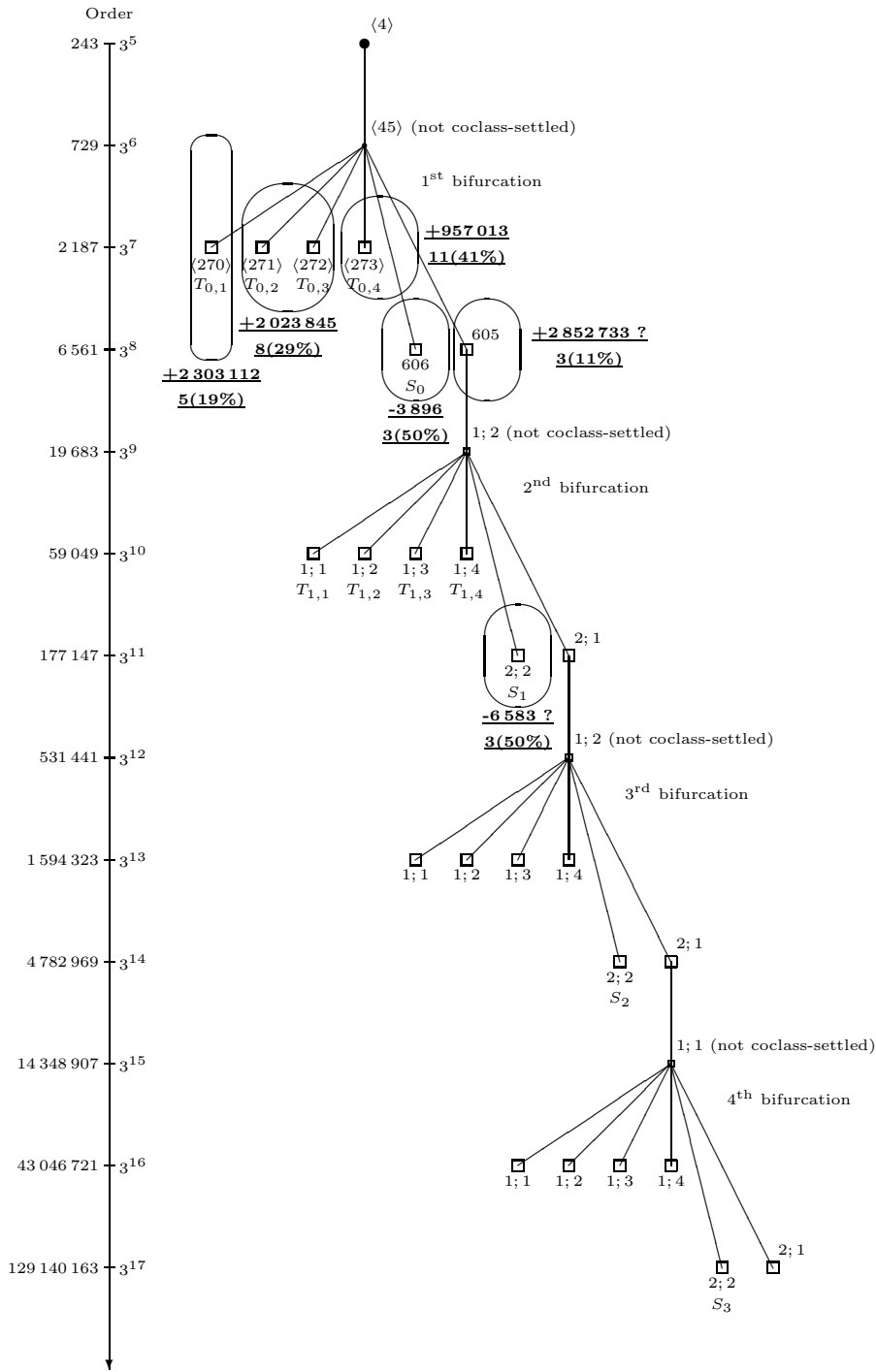
\begin{figure}[hb]
\caption{Distribution of \(3\)-class tower groups \(\mathrm{G}_3^\infty{K}\) on the descendant tree \(\mathcal{T}_\ast\langle 243,4\rangle\)}
\label{fig:TreeH4Spor}

\input{TreeH4}

\end{figure}

}



\subsection{Real Quadratic Fields of Type \(\mathrm{H}.4\)}
\label{ss:RQFH4Spor}

\begin{proposition}
\label{prp:RQFH4Spor}

(Fields of type \(\mathrm{H}.4\) up to \(d<10^7\)
\cite{Ma1},
\cite{Ma3})

\noindent
In the range \(0<d<10^7\) of fundamental discriminants \(d\)
of real quadratic fields \(K=\mathbb{Q}(\sqrt{d})\),
there exist precisely \(\mathbf{27}\) cases
with \(3\)-principalization type \(\mathrm{H}.4\), \(\varkappa_1{K}=(4111)\),
and IPAD \(\tau^{(1)}{K}=\lbrack 1^2;(1^3)^3,21\rbrack\).
They share the common second \(3\)-class group
\(\mathrm{G}_3^2{K}\simeq\langle 3^6,45\rangle\).

\end{proposition}

\begin{proof}
The results of
\cite[Tbl. 6.3, p. 452]{Ma3}
were computed in \(2010\) by means of PARI/GP
\cite{PARI}
using an implementation of our principalization algorithm, as described in
\cite[\S\ 5, pp. 446--450]{Ma3}.
The frequency \(27\)
in the last column \lq\lq freq.\rq\rq\
for the fourth row
concerns type \(\mathrm{H}.4\).
\end{proof}

\begin{remark}
\label{rmk:RQFH4Spor}
To discourage any misinterpretation,
we point out that there are four other real quadratic fields
\(K=\mathbb{Q}(\sqrt{d})\) with discriminants
\(d\in\lbrace 1\,162\,949,\ 2\,747\,001,\ 3\,122\,232,\ 4\,074\,493\rbrace\)
in the range \(0<d<10^7\) which possess the same \(3\)-principalization type \(\mathrm{H}.4\).
However their second \(3\)-class group
\(\mathrm{G}_3^2{K}\) is isomorphic to either
\(\langle 3^7,286\rangle-\#1;2\) or \(\langle 3^7,287\rangle-\#1;2\) of order \(3^8\),
which is not a sporadic group but is
located on the coclass tree \(\mathcal{T}^2\langle 3^5,6\rangle\),
and has a different IPAD \(\tau^{(1)}{K}=\lbrack 1^2;32,1^3,(21)^2\rbrack\).
The \(3\)-class towers of these fields are determined in
\cite{MaNm}.
\end{remark}



\begin{theorem} 
\label{thm:RQFH4Spor}

(\(3\)-Class towers of type \(\mathrm{H}.4\) up to \(d<10^7\))

\noindent
Among the \(27\) real quadratic fields \(K=\mathbb{Q}(\sqrt{d})\) with type \(\mathrm{H}.4\) in Proposition
\ref{prp:RQFH4Spor},

\begin{itemize}

\item
the \(\mathbf{11}\) fields (\(\mathbf{41}\%\)) with discriminants
\[
\begin{aligned}
d\in\lbrace    957\,013 &,& 1\,571\,953 &,& 1\,734\,184 &,& 3\,517\,689 &,& 4\,025\,909 &,& 4\,785\,845, \\
            4\,945\,973 &,& 5\,562\,969 &,& 6\,318\,733 &,& 7\,762\,296 &,& 8\,070\,637 & \rbrace
\end{aligned}
\]
have the unique \(3\)-class tower group 
\(G\simeq\langle 3^7,\mathbf{273}\rangle\)
and \(3\)-tower length \(\ell_3{K}=\mathbf{3}\),

\item
the \(\mathbf{8}\) fields (\(\mathbf{29}\%\)) with discriminants
\[
\begin{aligned}
d\in\lbrace 2\,023\,845 &,& 4\,425\,229 &,& 6\,418\,369 &,& 6\,469\,817 &,& \\
            6\,775\,224 &,& 6\,895\,612 &,& 7\,123\,493 &,& 9\,419\,261 & \rbrace
\end{aligned}
\]
have \(3\)-class tower group
\(G\simeq\langle 3^7,\mathbf{271}\rangle\)
or
\(G\simeq\langle 3^7,\mathbf{272}\rangle\)
and \(3\)-tower length \(\ell_3{K}=\mathbf{3}\),

\item
the \(\mathbf{5}\) fields (\(\mathbf{19}\%\)) with discriminants
\[d\in\lbrace 2\,303\,112,\ 3\,409\,817,\ 3\,856\,685,\ 5\,090\,485,\ 6\,526\,680\rbrace\]  
have the unique \(3\)-class tower group
\(G\simeq\langle 3^7,\mathbf{270}\rangle\)
and \(3\)-tower length \(\ell_3{K}=\mathbf{3}\),

\item
the \(\mathbf{3}\) fields (\(\mathbf{11}\%\)) with discriminants
\[d\in\lbrace 2\,852\,733,\ 8\,040\,029,\ 8\,369\,468\rbrace\]  
have a \(3\)-class tower group of order at least \(3^8\)
and \(3\)-tower length \(\ell_3{K}\in\lbrace\mathbf{3,4,\ldots}\rbrace\).

\end{itemize}

\noindent
Note that
\(\langle 3^7,\mathbf{270}\rangle=\langle 3^6,45\rangle-\#1;\mathbf{1}\),
\(\langle 3^7,\mathbf{271}\rangle=\langle 3^6,45\rangle-\#1;\mathbf{2}\),
\(\langle 3^7,\mathbf{272}\rangle=\langle 3^6,45\rangle-\#1;\mathbf{3}\), and
\(\langle 3^7,\mathbf{273}\rangle=\langle 3^6,45\rangle-\#1;\mathbf{4}\).

\end{theorem}

\begin{proof}
Extensions of absolute degrees \(6\) and \(18\) were constructed in steps with MAGMA
\cite{MAGMA},
using the class field package of C. Fieker
\cite{Fi}.
The resulting iterated IPADs of second order \(\tau^{(2)}{K}\) were used for the identification, according to Table
\ref{tbl:3GroupsH4Spor},
which is also contained in the more extensive theorem
\cite[Thm. 6.5, pp. 304--306]{Ma7}.
\end{proof}



\subsection{Imaginary Quadratic Fields of Type \(\mathrm{H}.4\)}
\label{ss:IQFH4Spor}

\begin{proposition}
\label{prp:IQFH4Spor}
(Fields of type \(\mathrm{H}.4\) down to \(d>-3\cdot 10^4\)
\cite{Ma},
\cite{Ma1}) \\
In the range \(-30\,000<d<0\) of fundamental discriminants \(d\)
of imaginary quadratic fields \(K=\mathbb{Q}(\sqrt{d})\),
there exist precisely \(\mathbf{6}\) cases
with \(3\)-principalization type \(\mathrm{H}.4\), \(\varkappa_1{K}=(4111)\),
and IPAD \(\tau^{(1)}{K}=\lbrack 1^2;(1^3)^3,21\rbrack\).
They share the common second \(3\)-class group
\(\mathrm{G}_3^2{K}\simeq\langle 3^6,45\rangle\).

\end{proposition}

\begin{proof}
In the table of suitable base fields
\cite[p. 84]{Ma},
the row Nr. \(4\) contains \(7\) discriminants \(-30\,000<d<0\)
of imaginary quadratic fields \(K=\mathbb{Q}(\sqrt{d})\) with type \(\mathrm{H}.4\).
It was computed in \(1989\) by means of our implementation of
the principalization algorithm by Scholz and Taussky, described in
\cite[pp. 80--83]{Ma}.
In \(1989\) already, we recognized that only for the discriminant \(d=-21\,668\)
one of the four absolute cubic subfields \(L_i\), \(1\le i\le 4\),
of the unramified cyclic cubic extensions \(N_i\) of \(K\)
has \(3\)-class number \(h_3{L_i}=9\),
which is not the case for the other \(6\) cases of type H.4 in the table
\cite[pp. 78--79]{Ma}.
According to
\cite[Prop. 4.4, p. 485]{Ma1}
or
\cite[Thm. 4.2, p. 489]{Ma1}
or
\cite{So},
the exceptional cubic field \(L_i\) is contained in a sextic field \(N_i\)
with \(3\)-class number \(h_3{N_i}=3\cdot (h_3{L_i})^2=243\),
which discourages an IPAD \(\tau^{(1)}{K}=\lbrack 1^2;(1^3)^3,21\rbrack\).
\end{proof}

\begin{remark}
\label{rmk:IQFH4Spor}
The imaginary quadratic field with discriminant \(d=-21\,668\)
possesses the same  \(3\)-principalization type \(\mathrm{H}.4\),
but its second \(3\)-class group
\(\mathrm{G}_3^2{K}\) is isomorphic to either
\(\langle 3^7,286\rangle-\#1;2\) or \(\langle 3^7,287\rangle-\#1;2\) of order \(3^8\),
and has the different IPAD \(\tau^{(1)}{K}=\lbrack 1^2;32,1^3,(21)^2\rbrack\).
Results for this field will be given in
\cite{MaNm}.
\end{remark}



\begin{theorem} 
\label{thm:IQFH4Spor}
(\(3\)-Class towers of type \(\mathrm{H}.4\) down to \(d>-3\cdot 10^4\)) \\
Among the \(6\) imaginary quadratic fields \(K=\mathbb{Q}(\sqrt{d})\) with type \(\mathrm{H}.4\) in Proposition
\ref{prp:IQFH4Spor},
\begin{itemize}
\item
the \(\mathbf{3}\) fields (\(\mathbf{50}\%\)) with discriminants
\[d\in\lbrace -3\,896,\ -25\,447,\ -27\,355\rbrace\]
have the unique \(3\)-class tower group 
\(G\simeq\langle 3^8,\mathbf{606}\rangle\)
and \(3\)-tower length \(\ell_3{K}=\mathbf{3}\),
\item
the \(\mathbf{3}\) fields (\(\mathbf{50}\%\)) with discriminants
\[d\in\lbrace -6\,583,\ -23\,428,\ -27\,991\rbrace\] 
have a \(3\)-class tower group of order at least \(3^{11}\)
and \(3\)-tower length \(\ell_3{K}\in\lbrace\mathbf{3,4,\ldots}\rbrace\).
\end{itemize}
\end{theorem}

\begin{proof}
Using the technique of Fieker
\cite{Fi},
extensions of absolute degrees \(6\) and \(54\) were constructed in two steps,
squeezing MAGMA
\cite{MAGMA}
close to its limits.
The resulting multi-layered iterated IPADs of second order \(\tau^{(2)}_\ast{K}\) were used for the identification,
according to Table
\ref{tbl:3GroupsH4Spor},
resp. the more detailed theorem
\cite[Thm. 6.5, pp. 304--306]{Ma7}.
\end{proof}



\renewcommand{\arraystretch}{1.2}

\begin{table}[ht]
\caption{IPOD \(\varkappa_1{G}\) and iterated IPAD \(\tau^{(2)}_\ast{G}\) of sporadic \(3\)-groups \(G\) of type G.19}
\label{tbl:3GroupsG19Spor}
\begin{center}
\begin{tabular}{|c|cc||cc||c|ccccc|}
\hline
    lo &                 id &  w.r.t. & type & \(\varkappa_1{G}\) & \(\tau_0{G}\) &      & \(\tau_0{H}\) &    \(\tau_1{H}\) &         \(\tau_2{H}\) &               \\
\hline
 \(5\) &              \(9\) &         & G.19 &           \(2143\) &       \(1^2\) & \(\lbrack\) & \(21\) &   \(1^3,(21)^3\) &           \((1^2)^4\) & \(\rbrack^4\) \\
\hline
\hline
 \(6\) &           \(W=57\) &         & G.19 &           \(2143\) &       \(1^2\) & \(\lbrack\) & \(21\) &   \(1^4,(21)^3\) &           \((1^3)^4\) & \(\rbrack^4\) \\
\hline
\hline
 \(7\) &   \(\mathbf{311}\) &         & G.19 &           \(2143\) &       \(1^2\) &             & \(21\) & \(1^4,(\mathbf{2}1^2)^3\) &       \(1^4,(1^3)^3\) &               \\
       &                    &         &      &                    &               &             & \(21\) &   \(1^4,(\mathbf{21})^3\) &           \((1^4)^4\) &               \\
       &                    &         &      &                    &               & \(\lbrack\) & \(21\) &   \(1^4,(\mathbf{21})^3\) &       \(1^4,(1^3)^3\) & \(\rbrack^2\) \\
\hline
\hline
 \(8\) &  \(\mathbf{625\ldots 630}\) &         & G.19 &           \(2143\) &       \(1^2\) & \(\lbrack\) & \(21\) & \(1^4,(\mathbf{21^2})^3\) &       \(1^5,(1^4)^3\) & \(\rbrack^4\) \\
\hline
\hline
 \(9\) &          \(\#1;2\) & \(\Phi\)& G.19 &           \(2143\) &       \(1^2\) & \(\lbrack\) & \(21\) & \(1^4,(21^2)^3\) &           \((1^5)^4\) & \(\rbrack^4\) \\
\hline
 \(9\) &          \(\#1;2\) & \(\Psi\)& G.19 &           \(2143\) &       \(1^2\) & \(\lbrack\) & \(21\) & \(1^4,(21^2)^3\) &      \(1^5,(21^3)^3\) & \(\rbrack^4\) \\
\hline
 \(9\) &          \(\#1;2\) &   \(Y\) & G.19 &           \(2143\) &       \(1^2\) & \(\lbrack\) & \(21\) & \(1^4,(21^2)^3\) &       \(1^6,(1^4)^3\) & \(\rbrack^4\) \\
\hline
 \(9\) &          \(\#1;2\) &   \(Z\) & G.19 &           \(2143\) &                                                  \multicolumn{6}{|c|}{IPAD like id \(Y-\#1;2\)} \\
\hline
 \(9\) &          \(\#1;3\) &   \(Z\) & G.19 &           \(2143\) &                                               \multicolumn{6}{|c|}{IPAD like id \(\Psi-\#1;2\)} \\
\hline
\hline
\(10\) & \(\#\mathbf{1;1}\) & \(Y_1\) & G.19 &           \(2143\) &       \(1^2\) &             & \(21\) & \(1^4,(\mathbf{3}1^2)^3\) &      \(21^5,(1^4)^3\) &               \\
       &                    &         &      &                    &               & \(\lbrack\) & \(21\) & \(1^4,(\mathbf{2}1^2)^3\) &      \(21^5,(1^4)^3\) & \(\rbrack^3\) \\
\hline
\(10\) &          \(\#1;1\) & \(Z_1\) & G.19 &           \(2143\) &       \(1^2\) &             & \(21\) & \(1^4,(2^21)^3\) &      \(21^5,(1^4)^3\) &               \\
       &                    &         &      &                    &               & \(\lbrack\) & \(21\) & \(1^4,(21^2)^3\) &      \(21^5,(1^4)^3\) & \(\rbrack^3\) \\
\hline
\(10\) &          \(\#2;7\) &   \(Z\) & G.19 &           \(2143\) &       \(1^2\) & \(\lbrack\) & \(21\) & \(1^4,(21^2)^3\) &      \(1^6,(21^3)^3\) & \(\rbrack^4\) \\
\hline
\hline
\(11\) & \(\#\mathbf{2;1/2}\) & \(Y_1\) & G.19 &    \(2143\) &       \(1^2\) & \(\lbrack\) & \(21\) & \(1^4,(\mathbf{3}1^2)^3\) &    \(2^21^4,(1^4)^3\) & \(\rbrack^4\) \\
\hline
\(11\) &  \(\#2;1\ldots 4\) & \(Z_1\) & G.19 &           \(2143\) &       \(1^2\) & \(\lbrack\) & \(21\) & \(1^4,(\mathbf{2^21})^3\) &    \(2^21^4,(1^4)^3\) & \(\rbrack^4\) \\
\hline
\(11\) &          \(\#1;1\) & \(Z_2\) & G.19 &           \(2143\) &       \(1^2\) &             & \(21\) & \(1^4,(2^21)^3\) &     \(21^5,(21^3)^3\) &               \\
       &                    &         &      &                    &               & \(\lbrack\) & \(21\) & \(1^4,(21^2)^3\) &     \(21^5,(21^3)^3\) & \(\rbrack^3\) \\
\hline
\(11\) &          \(\#1;5\) & \(Z_2\) & G.19 &           \(2143\) &       \(1^2\) &             & \(21\) & \(1^4,(21^2)^3\) &    \(1^6,(2^21^2)^3\) &               \\
       &                    &         &      &                    &               & \(\lbrack\) & \(21\) & \(1^4,(21^2)^3\) &      \(1^6,(21^3)^3\) & \(\rbrack^3\) \\
\hline
\hline
\(12\) &          \(\#2;1\) & \(Z_2\) & G.19 &           \(2143\) &       \(1^2\) & \(\lbrack\) & \(21\) & \(1^4,(\mathbf{2^21})^3\) &   \(2^21^4,(21^3)^3\) & \(\rbrack^4\) \\
\hline
\(12\) &         \(\#2;62\) & \(Z_2\) & G.19 &           \(2143\) &       \(1^2\) &             & \(21\) & \(1^4,(2^21)^3\) &   \(21^5,(2^21^2)^3\) &               \\
       &                    &         &      &                    &               & \(\lbrack\) & \(21\) & \(1^4,(21^2)^3\) &     \(21^5,(21^3)^3\) & \(\rbrack^3\) \\
\hline
\(12\) &         \(\#2;87\) & \(Z_2\) & G.19 &           \(2143\) &       \(1^2\) &             & \(21\) & \(1^4,(2^21)^3\) &     \(21^5,(21^3)^3\) &               \\
       &                    &         &      &                    &               &             & \(21\) & \(1^4,(21^2)^3\) &   \(21^5,(2^21^2)^3\) &               \\
       &                    &         &      &                    &               & \(\lbrack\) & \(21\) & \(1^4,(21^2)^3\) &     \(21^5,(21^3)^3\) & \(\rbrack^2\) \\
\hline
\hline
\(14\) & \(\#4;1\ldots 43\) & \(Z_2\) & G.19 &           \(2143\) &       \(1^2\) & \(\lbrack\) & \(21\) & \(1^4,(\mathbf{2^21})^3\) & \(2^21^4,(2^21^2)^3\) & \(\rbrack^4\) \\
\hline
\end{tabular}
\end{center}
\end{table}



Table
\ref{tbl:3GroupsG19Spor}
shows the designation of the transfer kernel type,
the IPOD \(\varkappa_1{G}\),
and the iterated multi-layered IPAD of second order,
\[\tau^{(2)}_\ast{G}=\lbrack\tau_0{G};\lbrack\tau_0{H};\tau_1{H};\tau_2{H}\rbrack_{H\in\mathrm{Lyr}_1{G}}\rbrack,\]
for sporadic \(3\)-groups \(G\) of type G.19 up to order \(\lvert G\rvert=3^{14}\),
characterized by the logarithmic order, \(\mathrm{lo}\),
and the SmallGroup identifier, \(\mathrm{id}\)
\cite{BEO1,BEO2},
resp. the relative identifier for \(\mathrm{lo}\ge 9\).
To enable a brief reference for relative identifiers we put\\
\(W:=\langle 3^6,57\rangle\),
since this group was called the non-CF group \(W\) by Ascione
\cite{As,AHL},\\
\(\Phi:=\langle 3^8,626\rangle\),
\(\Psi:=\langle 3^8,628\rangle\), and further\\
\(Y:=\langle 3^8,629\rangle\), 
\(Y_1:=Y-\#1;2\), and\\
\(Z:=\langle 3^8,630\rangle\),
\(Z_1:=Z-\#1;2\),
\(Z_2:=Z-\#2;7\).

The groups in Table
\ref{tbl:3GroupsG19Spor}
are represented by vertices of the tree diagram in Figure
\ref{fig:TreeG19Spor}.



\noindent
Figure
\ref{fig:TreeG19Spor}
visualizes sporadic \(3\)-groups
of Table
\ref{tbl:3GroupsG19Spor}
which arise as \(3\)-class tower groups \(G=\mathrm{G}_3^\infty{K}\)
of real quadratic fields \(K=\mathbb{Q}(\sqrt{d})\), \(d>0\),
with \(3\)-principalization type \(\mathrm{G}.19\)
and the corresponding minimal discriminants, resp. absolute frequencies in Theorem
\ref{thm:RQFG19SporExt}
and
\ref{thm:IQFG19Spor}.

The subtrees \(\mathcal{T}(W-\#2;i)\)
are finite and drawn completely for \(i\in\lbrace 1,3,5\rbrace\),
but they are omitted in the complicated cases \(i\in\lbrace 2,4,6\rbrace\),
where they reach beyond order \(3^{20}\).

For \(d=+24\,126\,593\), \(-12\,067\) and \(-54\,195\), we can only give the \textbf{conjectural} location of \(G\).



{\tiny

\begin{figure}[hb]
\caption{Distribution of \(3\)-class tower groups \(\mathrm{G}_3^\infty{K}\) on the descendant tree \(\mathcal{T}_\ast\langle 243,9\rangle\)}
\label{fig:TreeG19Spor}

\input{TreeG19}

\end{figure}
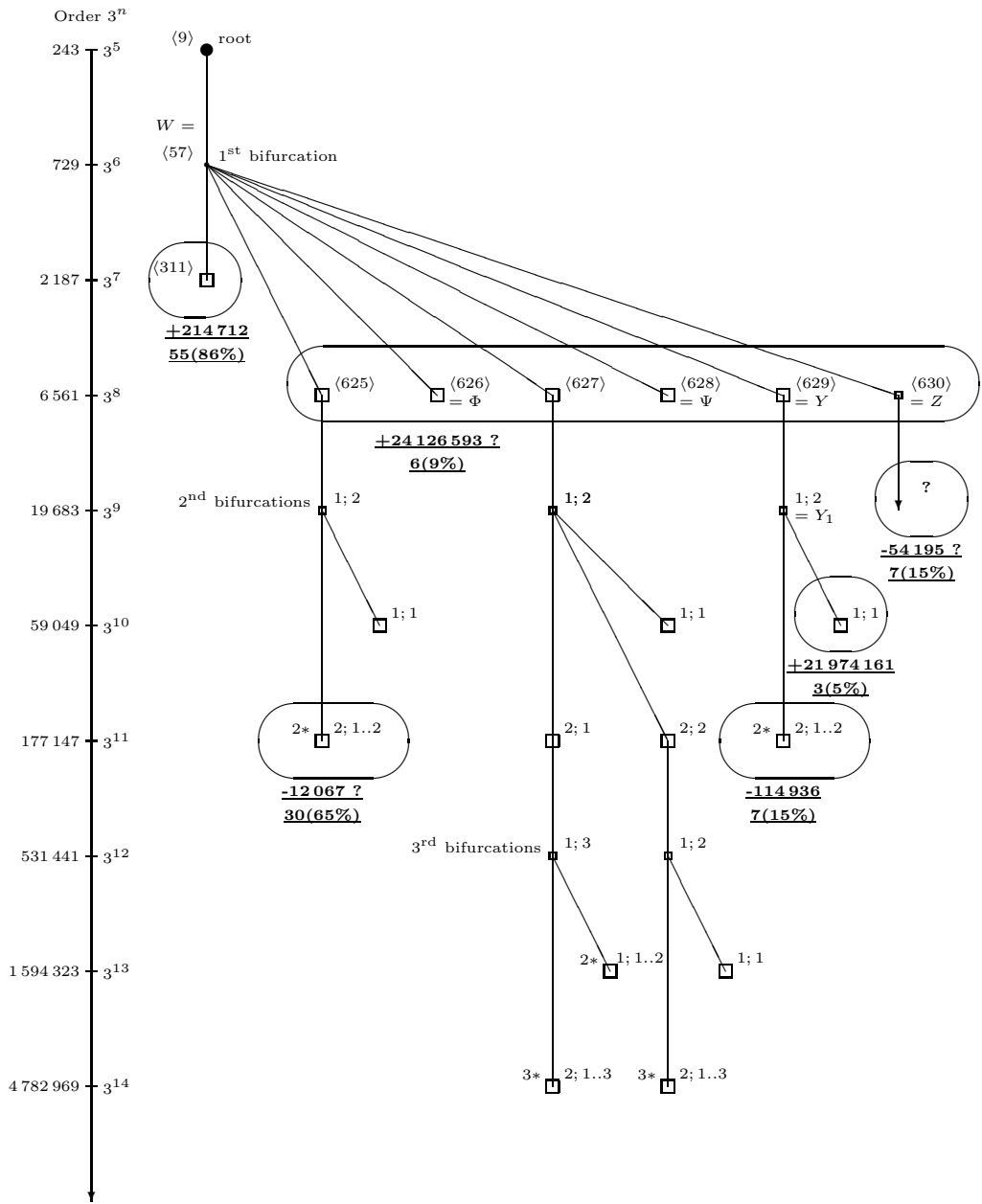

}



\subsection{Real Quadratic Fields of Type \(\mathrm{G}.19\)}
\label{ss:RQFG19Spor}

\begin{proposition}
\label{prp:RQFG19Spor}
(Fields of type \(\mathrm{G}.19\) up to \(d<10^7\)
\cite{Ma1},
\cite{Ma3}) \\
In the range \(0<d<10^7\) of fundamental discriminants \(d\)
of real quadratic fields \(K=\mathbb{Q}(\sqrt{d})\),
there exist precisely \(\mathbf{11}\) cases
with \(3\)-principalization type \(\mathrm{G}.19\), \(\varkappa_1{K}=(2143)\),
consisting of two disjoint \(2\)-cycles.
Their IPAD is uniformly given by \(\tau^{(1)}{K}=\lbrack 1^2;(21)^4\rbrack\),
in this range.
\end{proposition}

\begin{proof}
The results of
\cite[Tbl. 6.3, p. 452]{Ma3}
were computed in \(2010\) by means of PARI/GP
\cite{PARI}
using an implementation of our principalization algorithm, as described in
\cite[\S\ 5, pp. 446--450]{Ma3}.
The frequency \(11\)
in the last column \lq\lq freq.\rq\rq\
of the first row
concerns type \(\mathrm{G}.19\).
\end{proof}



\begin{theorem} 
\label{thm:RQFG19Spor}
(\(3\)-Class towers of type \(\mathrm{G}.19\) up to \(d<10^7\)) \\
The \(\mathbf{11}\) real quadratic fields \(K=\mathbb{Q}(\sqrt{d})\) in Proposition
\ref{prp:RQFG19Spor}
with discriminants
\[
\begin{aligned}
d\in\lbrace    214\,712 &,&    943\,077 &,& 1\,618\,493 &,& 2\,374\,077 &,& 3\,472\,653 &,& 4\,026\,680, \\
            4\,628\,117 &,& 5\,858\,753 &,& 6\,405\,317 &,& 7\,176\,477 &,& 7\,582\,988 &\rbrace
\end{aligned}
\]
have the unique \(3\)-class tower group 
\(G\simeq\langle 3^7,\mathbf{311}\rangle=\langle 3^6,57\rangle-\#1;1\)
and \(3\)-tower length \(\ell_3{K}=\mathbf{3}\).
\end{theorem}

\begin{proof}
Extensions of absolute degrees \(6\) and \(18\) were constructed with MAGMA
\cite{MAGMA},
using Fieker's class field package 
\cite{Fi}.
The resulting uniform iterated IPAD of second order
\(\tau^{(2)}{K}=\) \\
\(\lbrack 1^2; (21;1^4,(\mathbf{2}1^2)^3), (21;1^4,(\mathbf{21})^3)^3\rbrack\)
was used for the identification of \(G\), according to Table
\ref{tbl:3GroupsG19Spor}.
\end{proof}

Since real quadratic fields of type \(\mathrm{G}.19\) seemed to have a very rigid behaviour
with respect to their \(3\)-class field tower, admitting no variation at all,
we were curious about the continuation of these discriminants beyond the range \(d<10^7\).
Fortunately, M. R. Bush granted access to his extended numerical results for \(d<10^9\)
\cite{Bu},
and so we are able to state the following unexpected answer to our question
\lq\lq Is the \(3\)-class tower group \(G\) of real quadratic fields with
type \(\mathrm{G}.19\) and IPAD \(\tau^{(1)}(K)=\lbrack 1^2;(21)^4\rbrack\)
always isomorphic to \(\langle 3^7,\mathbf{311}\rangle\) in the SmallGroups Library?\rq\rq



\begin{proposition}
\label{prp:RQFG19SporExt}
(Fields of type \(\mathrm{G}.19\) up to \(d<5\cdot 10^7\)
\cite{Bu})
\noindent
In the range \(0<d<5\cdot 10^7\) of fundamental discriminants \(d\)
of real quadratic fields \(K=\mathbb{Q}(\sqrt{d})\),
there exist precisely \(\mathbf{64}\) cases
with \(3\)-principalization type \(\mathrm{G}.19\), \(\varkappa_1{K}=(2143)\),
and with IPAD \(\tau^{(1)}{K}=\lbrack 1^2;(21)^4\rbrack\).
\end{proposition}

\begin{proof}
Private communication by M. R. Bush
\cite{Bu}.
\end{proof}



\begin{theorem} 
\label{thm:RQFG19SporExt}
(\(3\)-Class towers of type \(\mathrm{G}.19\) up to \(d<5\cdot 10^7\)) \\
Among the \(64\) real quadratic fields \(K=\mathbb{Q}(\sqrt{d})\) with type \(\mathrm{G}.19\) in Proposition
\ref{prp:RQFG19SporExt},
\begin{itemize}
\item
the \(\mathbf{11}\) fields with discriminants \(d\) in Theorem
\ref{thm:RQFG19Spor}
and the \(\mathbf{44}\) fields with discriminants
\[
\begin{aligned}
d\in\lbrace 10\,169\,729 &,& 11\,986\,573 &,& 14\,698\,056 &,& 14\,836\,573 &,& 16\,270\,305 &,& 16\,288\,424, \\
            18\,195\,889 &,& 19\,159\,368 &,& 21\,519\,660 &,& 21\,555\,097 &,& 22\,296\,941 &,& 22\,431\,068, \\
            24\,229\,337 &,& 25\,139\,461 &,& 26\,977\,089 &,& 27\,696\,973 &,& 29\,171\,832 &,& 29\,523\,765, \\
            30\,019\,333 &,& 31\,921\,420 &,& 32\,057\,249 &,& 33\,551\,305 &,& 35\,154\,857 &,& 35\,846\,545, \\
            36\,125\,177 &,& 36\,409\,821 &,& 37\,344\,053 &,& 37\,526\,493 &,& 37\,796\,984 &,& 38\,691\,433, \\
            39\,693\,865 &,& 40\,875\,944 &,& 42\,182\,968 &,& 42\,452\,445 &,& 42\,563\,029 &,& 43\,165\,432, \\
            43\,934\,584 &,& 44\,839\,889 &,& 44\,965\,813 &,& 45\,049\,001 &,& 46\,180\,124 &,& 46\,804\,541, \\
            46\,971\,381 &,& 48\,628\,533 &\rbrace
\end{aligned}
\]
(that is, together \(\mathbf{55}\) fields or \(\mathbf{86}\%\))
have
\(\tau^{(2)}{K}=\lbrack 1^2; (21;1^4,(\mathbf{2}1^2)^3), (21;1^4,(\mathbf{21})^3)^3\rbrack\),
the unique \(3\)-class tower group 
\(G\simeq\langle 3^7,\mathbf{311}\rangle\),
and \(3\)-tower length \(\ell_3{K}=\mathbf{3}\),
\item
the \(\mathbf{3}\) fields (\(\mathbf{5}\%\)) with discriminants
\[d\in\lbrace 21\,974\,161,\ 22\,759\,557,\ 35\,327\,365\rbrace\] 
have 
IPAD of second order
\(\tau^{(2)}{K}=\lbrack 1^2; (21;1^4,(\mathbf{3}1^2)^3), (21;1^4,(\mathbf{21^2})^3)^3\rbrack\),
the unique \(3\)-tower group
\(G\simeq\langle 3^8,\mathbf{629}\rangle-\mathbf{\#1;2-\#1;1}\) of order \(3^{10}\),
and \(3\)-tower length \(\ell_3{K}=\mathbf{3}\),
\item
the \(\mathbf{6}\) fields (\(\mathbf{9}\%\)) with discriminants
\[d\in\lbrace 24\,126\,593,\ 29\,739\,477,\ 31\,353\,229,\ 35\,071\,865,\ 40\,234\,205,\ 40\,706\,677\rbrace\]  
have iterated IPAD of second order
\(\tau^{(2)}{K}=\lbrack 1^2; (21;1^4,(\mathbf{21^2})^3)^4\rbrack\),
a \(3\)-class tower group of order at least \(3^8\),
and \(3\)-tower length \(\ell_3{K}\in\lbrace\mathbf{3,4,\ldots}\rbrace\).
\end{itemize}
\end{theorem}

\begin{proof}
Similar to the proof of Theorem
\ref{thm:RQFG19Spor},
using Table
\ref{tbl:3GroupsG19Spor},
but now applied to the more extensive range of discriminants
and various iterated IPADs of second order.
\end{proof}



\subsection{Imaginary Quadratic Fields of Type \(\mathrm{G}.19\)}
\label{ss:IQFG19Spor}

\begin{proposition}
\label{prp:IQFG19Spor}
(Fields of type \(\mathrm{G}.19\) down to \(d>-5\cdot 10^5\)
\cite{Ma1},
\cite{Ma3}) \\
In the range \(-5\cdot 10^5<d<0\) of fundamental discriminants \(d\)
of imaginary quadratic fields \(K=\mathbb{Q}(\sqrt{d})\),
there exist precisely \(\mathbf{46}\) cases
with \(3\)-principalization type \(\mathrm{G}.19\), \(\varkappa_1{K}=(2143)\),
consisting of two disjoint \(2\)-cycles,
and with IPAD \(\tau^{(1)}{K}=\lbrack 1^2;(21)^4\rbrack\).
\end{proposition}

\begin{proof}
The results of
\cite[Tbl. 6.4, p. 452]{Ma3}
were computed in \(2010\) by means of PARI/GP
\cite{PARI}
using an implementation of our principalization algorithm, as described in
\cite[\S\ 5, pp. 446--450]{Ma3}.
The frequency \(94\)
in the last column \lq\lq freq.\rq\rq\
of the first row
concerns type \(\mathrm{G}.19\)
in the bigger range \(-10^6<d<0\).
Reduced to the first half of this range, we have \(46\) occurrences.
\end{proof}



\begin{theorem} 
\label{thm:IQFG19Spor}
(\(3\)-Class towers of type \(\mathrm{G}.19\) down to \(d>-5\cdot 10^5\)) \\
Among the \(46\) imaginary quadratic fields \(K=\mathbb{Q}(\sqrt{d})\) with type \(\mathrm{G}.19\) in Proposition
\ref{prp:IQFG19Spor},
\begin{itemize}
\item
the \(\mathbf{30}\) fields \(\mathbf{65}\%\)) with discriminants
\[
\begin{aligned}
d\in\lbrace  -12\,067 &,&  -49\,924 &,&  -60\,099 &,&  -83\,395 &,&  -86\,551 &,&  -93\,067, \\
            -152\,355 &,& -153\,771 &,& -161\,751 &,& -168\,267 &,& -195\,080 &,& -235\,491, \\
            -243\,896 &,& -251\,723 &,& -283\,523 &,& -310\,376 &,& -316\,259 &,& -337\,816, \\
            -339\,459 &,& -344\,823 &,& -350\,483 &,& -407\,983 &,& -421\,483 &,& -431\,247, \\
            -433\,732 &,& -442\,367 &,& -444\,543 &,& -453\,463 &,& -458\,724 &,& -471\,423 \rbrace \\
\end{aligned}
\] 
have iterated IPAD of second order
\(\tau^{(2)}{K}=\lbrack 1^2; (21;1^4,(\mathbf{21^2})^3)^4\rbrack\).
Conjecturally, most of them have \(3\)-class tower group 
\(G\simeq\langle 3^8,\mathbf{625}\rangle-\mathbf{\#1;2-\#2;1\vert 2}\) of order \(3^{11}\),
and \(3\)-tower length \(\ell_3{K}=\mathbf{3}\),
but \(\lvert G\rvert\ge 3^{14}\) and \(\ell_3{K}\ge 4\) cannot be excluded.
\item
The \(\mathbf{7}\) fields (\(\mathbf{15}\%\)) with discriminants
\[d\in\lbrace -54\,195,\ -96\,551,\ -104\,659,\ -133\,139,\ -222\,392,\ -313\,207,\ -420\,244\rbrace\] 
have iterated IPAD of second order
\(\tau^{(2)}{K}=\lbrack 1^2; (21;1^4,(\mathbf{2^21})^3)^4\rbrack\),
a \(3\)-class tower group of order at least \(3^{11}\),
and \(3\)-tower length \(\ell_3{K}\in\lbrace\mathbf{3,4,\ldots}\rbrace\),
\item
the \(\mathbf{7}\) fields (\(\mathbf{15}\%\)) with discriminants
\[d\in\lbrace -114\,936,\ -118\,276,\ -272\,659,\ -317\,327,\ -328\,308,\ -339\,563,\ -485\,411\rbrace\]  
have iterated IPAD of second order
\(\tau^{(2)}{K}=\lbrack 1^2; (21;1^4,(\mathbf{31^2})^3)^4\rbrack\),
a proven \(3\)-tower group
\(G\simeq\langle 3^8,\mathbf{629}\rangle-\mathbf{\#1;2-\#2;1\vert 2}\) of order \(3^{11}\),
and \(3\)-tower length \(\ell_3{K}=\mathbf{3}\),
\item
the \textbf{unique} field with discriminant \(d=-91\,643\)
has iterated IPAD of second order
\(\tau^{(2)}{K}=\lbrack 1^2; (21;1^4,(\mathbf{2^3})^3)^2, (21;1^4,(\mathbf{321})^3)^2\rbrack\),
unknown \(3\)-tower group and \(3\)-tower length \(\ell_3{K}\ge 3\),
\item
the \textbf{unique} field with discriminant \(d=-221\,944\)
has iterated IPAD of second order
\(\tau^{(2)}{K}=\lbrack 1^2; (21;1^4,(\mathbf{3^21})^3)^4\rbrack\),
but unknown \(3\)-tower group and \(3\)-tower length \(\ell_3{K}\ge 3\).
\end{itemize}
\end{theorem}

\begin{proof}
Similar to the proof of Theorem
\ref{thm:RQFG19Spor},
using Table
\ref{tbl:3GroupsG19Spor},
but now applied to the different range of discriminants
and various iterated IPADs of second order.
\end{proof}



\section{Imaginary quadratic fields of type \((3,3,3)\) and multi-layered IPADs}
\label{s:IQF3x3x3}
\noindent
In the final section \S\ 7 of
\cite{Ma7},
we proved that the second \(3\)-class groups \(\mathfrak{M}=\mathrm{G}_3^2{K}\)
of the \(14\) imaginary quadratic fields \(K=\mathbb{Q}(\sqrt{d})\)
with fundamental discriminants \(-10^7<d<0\)
and \(3\)-class group \(\mathrm{Cl}_3(K)\) of type \((3,3,3)\)
are pairwise non-isomorphic
\cite[Thm. 7.1, p. 307]{Ma7}.
For the proof of this theorem in 
\cite[\S\ 7.3, p. 311]{Ma7},
the IPADs of the \(14\) fields were not sufficient,
since the three fields with discriminants
\[d\in\lbrace -4\,447\,704, -5\,067\,967, -8\,992\,363\rbrace\]
share the common accumulated (unordered) IPAD
\[\tau^{(1)}{K}=\lbrack\tau_0{K};\tau_1{K}\rbrack=\lbrack 1^3;32^21;(21^4)^5,(2^21^2)^7\rbrack.\]
To complete the proof
we had to use information on the occupation numbers of the accumulated (unordered) IPODs,\\
\(\varkappa_1{K}=\lbrack 1,2,6,(8)^6,9,(10)^2,13\rbrack\) with maximal occupation number \(6\) for \(d=-4\,447\,704\),\\
\(\varkappa_1{K}=\lbrack 1,2,(3)^2,(4)^2,6,(7)^2,8,(9)^2,12\rbrack\) with maximal occupation number \(2\) for \(d=-5\,067\,967\),\\
\(\varkappa_1{K}=\lbrack (2)^2,5,6,7,(9)^2,(10)^3,(12)^3\rbrack\) with maximal occupation number \(3\) for \(d=-8\,992\,363\).

Meanwhile we succeeded in computing the \textit{second layer} of the transfer target type, \(\tau_2{K}\),
for the three critical fields with the aid of the computational algebra system MAGMA
\cite{MAGMA}
by determining the structure of the \(3\)-class groups \(\mathrm{Cl}_3{L}\)
of the \(13\) unramified bicyclic bicubic extensions \(L/K\) with relative degree \(\lbrack L:K\rbrack=3^2\)
and absolute degree \(18\).
In accumulated (unordered) form the second layer of the TTT is given by\\
\(\tau_2{K}=\lbrack 32^51^2;4321^5;2^51^3,(3^221^5)^2;2^41^4,32^21^5;(2^21^7)^3,(2^31^5)^3\rbrack\) for \(d=-4\,447\,704\),\\
\(\tau_2{K}=\lbrack 3^22^21^4;(3^221^5)^3;32^21^5;(2^31^5)^8\rbrack\) for \(d=-5\,067\,967\), and\\
\(\tau_2{K}=\lbrack 32^21^6,(3^221^5)^3;2^41^4,32^21^5;2^21^7,(2^31^5)^6\rbrack\) for \(d=-8\,992\,363\).

These results admit incredibly powerful conclusions,
which bring us closer to the ultimate goal to determine the precise isomorphism type of \(\mathrm{G}_3^2{K}\).
Firstly, they clearly show that the second \(3\)-class groups of the three critical fields are pairwise non-isomorphic
without using the IPODs.
Secondly, the component with the biggest order establishes an impressively sharpened estimate
for the order of \(\mathrm{G}_3^2{K}\) from below.
The background is explained by the following lemma.



\begin{lemma}
\label{lem:Estimates}
Let \(G\) be a finite \(p\)-group with abelianization \(G/G^\prime\) of type \((p,p,p)\)
and denote by \(\mathrm{lo}_p(G):=\log_p(\mathrm{ord}(G))\) the logarithmic order of \(G\)
with respect to the prime number \(p\).
Then the abelianizations \(H/H^\prime\) of subgroups \(H<G\) in various layers of \(G\)
admit lower bounds for \(\mathrm{lo}_p(G)\):
\begin{enumerate}
\item
\(\mathrm{lo}_p(G)\ge 1+\max\lbrace\mathrm{lo}_p(H/H^\prime)\mid H\in\mathrm{Lyr}_1{G}\rbrace\).
\item
\(\mathrm{lo}_p(G)\ge 2+\max\lbrace\mathrm{lo}_p(H/H^\prime)\mid H\in\mathrm{Lyr}_2{G}\rbrace\).
\item
\(\mathrm{lo}_p(G)\ge 3+\mathrm{lo}_p(G^\prime/G^{\prime\prime})\), and in particular we have an equation\\
\(\mathrm{lo}_p(G)=3+\mathrm{lo}_p(G^\prime)\) if \(G\) is metabelian.
\end{enumerate}
\end{lemma}

\begin{proof}
The Lagrange formula for the order of \(G\) in terms of the index of a subgroup \(H\le G\) reads
\[\mathrm{ord}(G)=(G:H)\cdot\mathrm{ord}(H),\]
and taking the \(p\)-logarithm yields
\[\mathrm{lo}_p(G)=\log_p((G:H))+\mathrm{lo}_p(H).\]
In particular, we have \(\log_p((G:H))=\log_p(p^n)=n\) for \(H\in\mathrm{Lyr}_n{G}\), \(0\le n\le 3\),
and again by the Lagrange formula
\[\mathrm{ord}(H)=(H:H^\prime)\cdot\mathrm{ord}(H^\prime)\ge (H:H^\prime),\]
respectively
\[\mathrm{lo}_p(H)=\log_p((H:H^\prime))+\mathrm{lo}_p(H^\prime)\ge\mathrm{lo}_p(H/H^\prime),\]
with equality if and only if \(H^\prime=1\), that is, \(H\) is abelian.\\
Finally, \(G\) is metabelian if and only if \(G^\prime\) is abelian.
\end{proof}



Let us first draw weak conclusions from the first layer of the TTT, i.e. the IPAD,
with the aid of Lemma
\ref{lem:Estimates}.

\begin{theorem}
\label{thm:WeakEstimates}
(Coarse estimate \cite{Ma7}) \\
The order of \(\mathfrak{M}=\mathrm{G}_3^2{K}\) for the three critical fields \(K\)
is bounded from below by \(\mathrm{ord}(\mathfrak{M})\ge 3^9\).\\
If the maximal subgroup \(H<\mathfrak{M}\) with the biggest order of \(H/H^\prime\) is abelian,
i.e. \(H^\prime=1\),
then the precise logarithmic order of \(\mathfrak{M}\) is given by \(\mathrm{lo}_3(\mathfrak{M})=9\).
\end{theorem}

\begin{proof}
The three critical fields with discriminants
\(d\in\lbrace -4\,447\,704, -5\,067\,967, -8\,992\,363\rbrace\)
share the common accumulated IPAD
\(\tau^{(1)}{K}=\lbrack\tau_0{K};\tau_1{K}\rbrack=\lbrack 1^3;32^21;(21^4)^5,(2^21^2)^7\rbrack\).

Consequently, Lemma
\ref{lem:Estimates}
yields a uniform lower bound for each of the three fields:
\[\mathrm{lo}_3(\mathfrak{M})\ge 1+\max\lbrace\mathrm{lo}_3(H/H^\prime)\mid H\in\mathrm{Lyr}_1{\mathfrak{M}}\rbrace
=1+\mathrm{lo}_3(32^21)=1+3+2\cdot 2+1=9.\]

The assumption that a maximal subgroup \(U<\mathfrak{M}\) having not the biggest order of \(U/U^\prime\)
were abelian (with \(U/U^\prime\simeq U\))
immediately yields the contradiction that
\[\mathrm{lo}_3(\mathfrak{M})=\log_3((\mathfrak{M}:U))+\mathrm{lo}_3(U)=1+\mathrm{lo}_3(U/U^\prime)
<1+\max\lbrace\mathrm{lo}_3(H/H^\prime)\mid H\in\mathrm{Lyr}_1{\mathfrak{M}}\rbrace\le\mathrm{lo}_3(\mathfrak{M}).\]
\end{proof}



\noindent
It is illuminating that much stronger estimates and conclusions are possible
by applying Lemma
\ref{lem:Estimates}
to the second layer of the TTT.

\begin{theorem}
\label{thm:StrongEstimates}
(Finer estimates) \\
None of the maximal subgroups of \(\mathfrak{M}=\mathrm{G}_3^2{K}\)
for the three critical fields \(K\) can be abelian.\\
The logarithmic order of \(\mathfrak{M}\) is bounded from below by\\
\(\mathrm{lo}_3(\mathfrak{M})\ge 17\) for \(d=-4\,447\,704\),\\
\(\mathrm{lo}_3(\mathfrak{M})\ge 16\) for \(d=-5\,067\,967\),\\
\(\mathrm{lo}_3(\mathfrak{M})\ge 15\) for \(d=-8\,992\,363\).
\end{theorem}

\begin{proof}
As mentioned earlier already,
computations with MAGMA
\cite{MAGMA}
have shown that the accumulated second layer of the TTT is given by\\
\(\tau_2{K}=\lbrack 32^51^2;4321^5;2^51^3,(3^221^5)^2;2^41^4,32^21^5;(2^21^7)^3,(2^31^5)^3\rbrack\) for \(d=-4\,447\,704\),\\
\(\tau_2{K}=\lbrack 3^22^21^4;(3^221^5)^3;32^21^5;(2^31^5)^8\rbrack\) for \(d=-5\,067\,967\), and\\
\(\tau_2{K}=\lbrack 32^21^6,(3^221^5)^3;2^41^4,32^21^5;2^21^7,(2^31^5)^6\rbrack\) for \(d=-8\,992\,363\).

Consequently the maximal logarithmic order \(M:=\max\lbrace\mathrm{lo}_3(H/H^\prime)\mid H\in\mathrm{Lyr}_2{\mathfrak{M}}\rbrace\) is\\
\(M=\mathrm{lo}_3(32^51^2)=3+5\cdot 2+2\cdot 1=15\) for \(d=-4\,447\,704\),\\
\(M=\mathrm{lo}_3(3^22^21^4)=2\cdot 3+2\cdot 2+4\cdot 1=14\) for \(d=-5\,067\,967\),\\
\(M=\mathrm{lo}_3(32^21^6)=3+2\cdot 2+6\cdot 1=13\) for \(d=-8\,992\,363\).\\
According to Lemma
\ref{lem:Estimates},
we have
\(\mathrm{lo}_3(\mathfrak{M})\ge 2+\max\lbrace\mathrm{lo}_3(H/H^\prime)\mid H\in\mathrm{Lyr}_2{\mathfrak{M}}\rbrace=2+M\).

Finally, if one of the maximal subgroups of \(\mathfrak{M}\) were abelian,
then Theorem
\ref{thm:WeakEstimates}
would give the contradiction that \(\mathrm{lo}_3(\mathfrak{M})=9\).
\end{proof}



Unfortunately, it was impossible for any of the three critical fields \(K\)
to compute the third layer of the TTT, \(\tau_3{K}\),
that is the structure of the \(3\)-class group of the Hilbert \(3\)-class field \(\mathrm{F}_3^1{K}\) of \(K\),
which is of absolute degree \(54\).
This would have given the precise order of the metabelian group
\(\mathfrak{M}=\mathrm{G}_3^2{K}=\mathrm{Gal}(\mathrm{F}_3^2{K}/K)\),
according to Lemma
\ref{lem:Estimates},
since \(\mathfrak{M}^\prime=\mathrm{Gal}(\mathrm{F}_3^2{K}/\mathrm{F}_3^1{K})\simeq\mathrm{Cl}_3(\mathrm{F}_3^1{K})\).

We also investigated
whether the complete iterated IPAD of second order, \(\tau^{(2)}{\mathfrak{M}}\),
is able to improve the lower bounds in Theorem
\ref{thm:StrongEstimates}
further.
It turned out that, 
firstly none of the additional non-normal components of \((\tau_1{H})_{H\in\mathrm{Lyr}_1{\mathfrak{M}}}\)
seems to have bigger order than the normal components of \(\tau_2{\mathfrak{M}}\),
and secondly,
due to the huge \(3\)-ranks of the involved groups,
the number of required class group computations enters astronomic regions.

To give an impression, we show the results for five of the \(13\) maximal subgroups
in the case of \(d=-4\,447\,704\):\\
\(\tau^{(1)}{H_1}=\lbrack 2^21^2;\ 32^51^2;(2^31^5)^3;(3^221^2)^3;(321^4)^9,(32^21^2)^{24}\rbrack\), with \(40\) components,\\
\(\tau^{(1)}{H_2}=\lbrack 21^4;\ 32^51^2;2^51^3;2^41^4;2^21^7;(31^6)^3,(321^4)^{33};(321^2)^{81}\rbrack\), with \(121\) components,\\
\(\tau^{(1)}{H_3}=\lbrack 2^21^2;\ 32^51^2;32^21^5;(2^21^7)^2;(321^5)^3,(32^21^3)^6,(3^221^2)^3,(32^21^2)^{24}\rbrack\), with \(40\) comp.,\\
\(\tau^{(1)}{H_4}=\lbrack 32^21;\ 32^51^2;4321^5;(3^221^5)^2;(4321^3)^6;(431^4)^6,(3^221^3)^6,(4321^2)^9,(3^31^2)^9\rbrack\), \(40\) comp.,\\
\(\tau^{(1)}{H_5}=\lbrack 2^21^2;\ 3^221^5;32^21^5,2^41^4;2^31^5;(321^3)^{36}\rbrack\), with \(40\) components.



\section{Acknowledgements}
\label{s:Acknowledgements}
\noindent
We gratefully acknowledge that our research is supported by
the Austrian Science Fund (FWF): P 26008-N25.

Sincere thanks are given to Michael R. Bush
(Washington and Lee University, Lexington, VA)
for making available
numerical results on IPADs of real quadratic fields \(K=\mathbb{Q}(\sqrt{d})\),
and the distribution of discriminants \(d<10^9\) over these IPADs
\cite{Bu}.

We are indebted to Nigel Boston, Michael R. Bush and Farshid Hajir
for kindly making available an unpublished database containing
numerical results of their paper
\cite{BBH}
and a related paper on real quadratic fields, which is still in preparation.


A succinct version of the present article has been delivered on July \(09\), \(2015\),
within the frame of the
\(29\)i\`emes Journ\'ees Arithme\'etiques at the University of Debrecen, Hungary
\cite{Ma7b}.



\bigskip
\noindent
\textbf{Appendix: Corrigenda in
\cite{Ma1,Ma3,Ma4}.}
\begin{enumerate}
\item
The restriction of the numerical results in Proposition
\ref{prp:a1a2a3}
to the range \(0<d<10^7\)
is in perfect accordance with our machine calculations
by means of PARI/GP
\cite{PARI}
in \(2010\),
thus providing the first independent verification of data in
\cite{Ma1,Ma3,Ma4}.\\
However,
in the manual evaluation of this extensive data material
for the ground state of the types \(\mathrm{a}.1\), \(\mathrm{a}.2\), \(\mathrm{a}.3\), and \(\mathrm{a}.3^\ast\),
a few errors crept in,
which must be corrected at three locations:
in the tables
\cite[Tbl. 2, p. 496]{Ma1}
and
\cite[Tbl. 6.1, p. 451]{Ma3},
and in the tree diagram
\cite[Fig. 3.2, p. 422]{Ma4}.

The absolute frequency of the ground state is actually given by\\
\(1\,382\) instead of the incorrect \(1\,386\) for the union of types \(\mathrm{a}.2\) and \(\mathrm{a}.3\),\\
\(698\) instead of the incorrect \(697\) for type \(\mathrm{a}.3^\ast\),\\
\(2\,080\) instead of the incorrect \(2\,083\) for the union of types \(\mathrm{a}.2\), \(\mathrm{a}.3\), and \(\mathrm{a}.3^\ast\), and\\
\(150\) instead of the incorrect \(147\) for type \(\mathrm{a}.1\).\\
(The three discriminants \(d\in\lbrace 7\,643\,993, 7\,683\,308, 8\,501\,541\rbrace\)
were erroneously classified as type \(\mathrm{a}.2\) or \(\mathrm{a}.3\) instead of \(\mathrm{a}.1\).)

In the second table,
two relative frequencies (percentages) should be updated:\\
\(\frac{1382}{2303}\approx 60.0\%\) instead of \(\frac{1386}{2303}\approx 60.2\%\)
and\\
\(\frac{698}{2303}\approx 30.3\%\) instead of \(\frac{697}{2303}\approx 30.3\%\).
\item
Incidentally, although it does not concern the section \(\mathrm{a}\) of IPODs,
the single field with discriminant \(d=2\,747\,001\) was erroneously classified
as type \(\mathrm{c}.18\), \(\varkappa_1=(0313)\), instead of \(\mathrm{H}.4\), \(\varkappa_1=(3313)\).
This has consequences at four locations:
in the tables
\cite[Tbl. 4--5, pp. 498--499]{Ma1}
and
\cite[Tbl. 6.5, p. 452]{Ma3},
and in the tree diagram
\cite[Fig. 3.6, p. 442]{Ma4}.

The absolute frequency of these types is actually given by\\
\(28\) instead of the incorrect \(29\) for type \(\mathrm{c}.18\) (see also
\cite[Prop. 7.2]{Ma10}),\\
\(4\) instead of the incorrect \(3\) for type \(\mathrm{H}.4\).

In the first two tables,
the total frequencies should be updated, correspondingly:\\
\(207\) instead of the incorrect \(206\) in
\cite[Tbl. 4, p. 498]{Ma1},\\
\(66\) instead of the incorrect \(67\) in
\cite[Tbl. 5, p. 499]{Ma1}.
\end{enumerate}




\end{document}

%% file: AbsFreqTyp33TreeCc1.tex

\setlength{\unitlength}{0.8cm}
\begin{picture}(18,21)(-11,-20)

\put(-10,0.5){\makebox(0,0)[cb]{order \(3^n\)}}

\put(-10,0){\line(0,-1){16}}
\multiput(-10.1,0)(0,-2){9}{\line(1,0){0.2}}

\put(-10.2,0){\makebox(0,0)[rc]{\(9\)}}
\put(-9.8,0){\makebox(0,0)[lc]{\(3^2\)}}
\put(-10.2,-2){\makebox(0,0)[rc]{\(27\)}}
\put(-9.8,-2){\makebox(0,0)[lc]{\(3^3\)}}
\put(-10.2,-4){\makebox(0,0)[rc]{\(81\)}}
\put(-9.8,-4){\makebox(0,0)[lc]{\(3^4\)}}
\put(-10.2,-6){\makebox(0,0)[rc]{\(243\)}}
\put(-9.8,-6){\makebox(0,0)[lc]{\(3^5\)}}
\put(-10.2,-8){\makebox(0,0)[rc]{\(729\)}}
\put(-9.8,-8){\makebox(0,0)[lc]{\(3^6\)}}
\put(-10.2,-10){\makebox(0,0)[rc]{\(2\,187\)}}
\put(-9.8,-10){\makebox(0,0)[lc]{\(3^7\)}}
\put(-10.2,-12){\makebox(0,0)[rc]{\(6\,561\)}}
\put(-9.8,-12){\makebox(0,0)[lc]{\(3^8\)}}
\put(-10.2,-14){\makebox(0,0)[rc]{\(19\,683\)}}
\put(-9.8,-14){\makebox(0,0)[lc]{\(3^9\)}}
\put(-10.2,-16){\makebox(0,0)[rc]{\(59\,049\)}}
\put(-9.8,-16){\makebox(0,0)[lc]{\(3^{10}\)}}

\put(-10,-16){\vector(0,-1){2}}

\put(-8,0.5){\makebox(0,0)[cb]{\(\tau(1)=\)}}

\put(-8,0){\makebox(0,0)[cc]{\((1)\)}}
\put(-8,-2){\makebox(0,0)[cc]{\((1^2)\)}}
\put(-8,-4){\makebox(0,0)[cc]{\((21)\)}}
\put(-8,-6){\makebox(0,0)[cc]{\((2^2)\)}}
\put(-8,-8){\makebox(0,0)[cc]{\((32)\)}}
\put(-8,-10){\makebox(0,0)[cc]{\((3^2)\)}}
\put(-8,-12){\makebox(0,0)[cc]{\((43)\)}}
\put(-8,-14){\makebox(0,0)[cc]{\((4^2)\)}}
\put(-8,-16){\makebox(0,0)[cc]{\((54)\)}}

\put(-8,-17){\makebox(0,0)[cc]{\textbf{IPAD}}}
\put(-8.5,-17.2){\framebox(1,18){}}

\put(3.7,-3){\vector(0,1){1}}
\put(3.9,-3){\makebox(0,0)[lc]{depth \(1\)}}
\put(3.7,-3){\vector(0,-1){1}}

\put(-5.1,-8){\vector(0,1){2}}
\put(-5.3,-7){\makebox(0,0)[rc]{period length \(2\)}}
\put(-5.1,-8){\vector(0,-1){2}}

\put(-0.1,-0.1){\framebox(0.2,0.2){}}
\put(-2.1,-0.1){\framebox(0.2,0.2){}}

\multiput(0,-2)(0,-2){8}{\circle*{0.2}}
\multiput(-2,-2)(0,-2){8}{\circle*{0.2}}
\multiput(-4,-4)(0,-2){7}{\circle*{0.2}}
\multiput(-6,-4)(0,-4){4}{\circle*{0.2}}

\multiput(2,-6)(0,-2){6}{\circle*{0.1}}
\multiput(4,-6)(0,-2){6}{\circle*{0.1}}
\multiput(6,-6)(0,-2){6}{\circle*{0.1}}

\multiput(0,0)(0,-2){8}{\line(0,-1){2}}
\multiput(0,0)(0,-2){8}{\line(-1,-1){2}}
\multiput(0,-2)(0,-2){7}{\line(-2,-1){4}}
\multiput(0,-2)(0,-4){4}{\line(-3,-1){6}}
\multiput(0,-4)(0,-2){6}{\line(1,-1){2}}
\multiput(0,-4)(0,-2){6}{\line(2,-1){4}}
\multiput(0,-4)(0,-2){6}{\line(3,-1){6}}

\put(0,-16){\vector(0,-1){2}}
\put(0.2,-17.5){\makebox(0,0)[lc]{infinite}}
\put(0.2,-18){\makebox(0,0)[lc]{mainline}}
\put(-0.2,-18.3){\makebox(0,0)[rc]{\(\mathcal{T}^1(\langle 9,2\rangle)\)}}

\put(0,0.1){\makebox(0,0)[lb]{\(=C_3\times C_3\)}}
\put(-2.5,0.1){\makebox(0,0)[rb]{\(C_9=\)}}
\put(-0.8,0.3){\makebox(0,0)[rb]{abelian}}
\put(0.1,-1.9){\makebox(0,0)[lb]{\(=G^3_0(0,0)\)}}
\put(-2.5,-1.9){\makebox(0,0)[rb]{\(G^3_0(0,1)=\)}}
\put(-4,-3.9){\makebox(0,0)[lb]{\(=\mathrm{Syl}_3A_9\)}}
\put(-4,-4.5){\makebox(0,0)[cc]{\textbf{IPAD} \quad \(\tau(1)=(1^3)\)}}
\put(-5.5,-4.7){\framebox(2.9,0.5){}}

\put(0.2,-2.2){\makebox(0,0)[lt]{bifurcation from \(\mathcal{G}(3,1)\)}}
\put(1.5,-2.5){\makebox(0,0)[lt]{to \(\mathcal{G}(3,2)\)}}

\put(-2.1,0.1){\makebox(0,0)[rb]{\(\langle 1\rangle\)}}
\put(-0.1,0.1){\makebox(0,0)[rb]{\(\langle 2\rangle\)}}

\put(-1.5,-1){\makebox(0,0)[rc]{branch \(\mathcal{B}(2)\)}}
\put(-2.1,-1.9){\makebox(0,0)[rb]{\(\langle 4\rangle\)}}
\put(-0.1,-1.9){\makebox(0,0)[rb]{\(\langle 3\rangle\)}}

\put(-4.5,-3){\makebox(0,0)[cc]{\(\mathcal{B}(3)\)}}
\put(-6.1,-3.9){\makebox(0,0)[rb]{\(\langle 8\rangle\)}}
\put(-4.1,-3.9){\makebox(0,0)[rb]{\(\langle 7\rangle\)}}
\put(-2.1,-3.9){\makebox(0,0)[rb]{\(\langle 10\rangle\)}}
\put(-0.1,-3.9){\makebox(0,0)[rb]{\(\langle 9\rangle\)}}

\put(-4.1,-5.9){\makebox(0,0)[rb]{\(\langle 25\rangle\)}}
\put(-2.1,-5.9){\makebox(0,0)[rb]{\(\langle 27\rangle\)}}
\put(-0.1,-5.9){\makebox(0,0)[rb]{\(\langle 26\rangle\)}}

\put(4.5,-5){\makebox(0,0)[cc]{\(\mathcal{B}(4)\)}}
\put(2.1,-5.9){\makebox(0,0)[lb]{\(\langle 28\rangle\)}}
\put(4.1,-5.9){\makebox(0,0)[lb]{\(\langle 30\rangle\)}}
\put(6.1,-5.9){\makebox(0,0)[lb]{\(\langle 29\rangle\)}}

\put(-6.1,-7.9){\makebox(0,0)[rb]{\(\langle 98\rangle\)}}
\put(-4.1,-7.9){\makebox(0,0)[rb]{\(\langle 97\rangle\)}}
\put(-2.1,-7.9){\makebox(0,0)[rb]{\(\langle 96\rangle\)}}
\put(-0.1,-7.9){\makebox(0,0)[rb]{\(\langle 95\rangle\)}}

\put(4.5,-7){\makebox(0,0)[cc]{\(\mathcal{B}(5)\)}}
\put(2.1,-7.9){\makebox(0,0)[lb]{\(\langle 100\rangle\)}}
\put(4.1,-7.9){\makebox(0,0)[lb]{\(\langle 99\rangle\)}}
\put(6.1,-7.9){\makebox(0,0)[lb]{\(\langle 101\rangle\)}}

\put(-4.1,-9.9){\makebox(0,0)[rb]{\(\langle 388\rangle\)}}
\put(-2.1,-9.9){\makebox(0,0)[rb]{\(\langle 387\rangle\)}}
\put(-0.1,-9.9){\makebox(0,0)[rb]{\(\langle 386\rangle\)}}

\put(4.5,-9){\makebox(0,0)[cc]{\(\mathcal{B}(6)\)}}
\put(2.1,-9.9){\makebox(0,0)[lb]{\(\langle 390\rangle\)}}
\put(4.1,-9.9){\makebox(0,0)[lb]{\(\langle 389\rangle\)}}
\put(6.1,-9.9){\makebox(0,0)[lb]{\(\langle 391\rangle\)}}

\put(-5.8,-11.9){\makebox(0,0)[rb]{\(\langle 2224\rangle\)}}
\put(-4,-11.9){\makebox(0,0)[rb]{\(\langle 2223\rangle\)}}
\put(-2.1,-11.9){\makebox(0,0)[rb]{\(\langle 2222\rangle\)}}
\put(-0.1,-11.9){\makebox(0,0)[rb]{\(\langle 2221\rangle\)}}

\put(4.5,-11){\makebox(0,0)[cc]{\(\mathcal{B}(7)\)}}
\put(2.1,-11.9){\makebox(0,0)[lb]{\(\langle 2226\rangle\)}}
\put(4,-11.9){\makebox(0,0)[lb]{\(\langle 2225\rangle\)}}
\put(5.8,-11.9){\makebox(0,0)[lb]{\(\langle 2227\rangle\)}}

\put(4.5,-13){\makebox(0,0)[cc]{\(\mathcal{B}(8)\)}}
\put(4.5,-15){\makebox(0,0)[cc]{\(\mathcal{B}(9)\)}}

\put(0.1,-16.2){\makebox(0,0)[ct]{\(G^n_0(0,0)\)}}
\put(-2,-16.2){\makebox(0,0)[ct]{\(G^n_0(0,1)\)}}
\put(-4,-16.2){\makebox(0,0)[ct]{\(G^n_0(1,0)\)}}
\put(-6,-16.2){\makebox(0,0)[ct]{\(G^n_0(-1,0)\)}}
\put(2,-16.2){\makebox(0,0)[ct]{\(G^n_1(0,-1)\)}}
\put(4,-16.2){\makebox(0,0)[ct]{\(G^n_1(0,0)\)}}
\put(6,-16.2){\makebox(0,0)[ct]{\(G^n_1(0,1)\)}}

\put(-3,-17.7){\makebox(0,0)[ct]{with abelian maximal subgroup}}
\put(4,-17.7){\makebox(0,0)[ct]{without abelian maximal subgroup}}

\put(2.5,0){\makebox(0,0)[cc]{\textbf{IPOD}}}
\put(3.5,0){\makebox(0,0)[cc]{a.1}}
\put(2.5,-0.5){\makebox(0,0)[cc]{\(\varkappa=\)}}
\put(3.5,-0.5){\makebox(0,0)[cc]{\((0000)\)}}
\put(1.8,-0.7){\framebox(2.6,1){}}
\put(-6,-2){\makebox(0,0)[cc]{\textbf{IPOD}}}
\put(-5,-2){\makebox(0,0)[cc]{A.1}}
\put(-6,-2.5){\makebox(0,0)[cc]{\(\varkappa=\)}}
\put(-5,-2.5){\makebox(0,0)[cc]{\((1111)\)}}
\put(-6.7,-2.7){\framebox(2.6,1){}}

\put(-8,-19){\makebox(0,0)[cc]{\textbf{IPOD}}}
\put(0,-19){\makebox(0,0)[cc]{a.1}}
\put(-2,-19){\makebox(0,0)[cc]{a.2}}
\put(-4,-19){\makebox(0,0)[cc]{a.3}}
\put(-6,-19){\makebox(0,0)[cc]{a.3}}
\put(2,-19){\makebox(0,0)[cc]{a.1}}
\put(4,-19){\makebox(0,0)[cc]{a.1}}
\put(6,-19){\makebox(0,0)[cc]{a.1}}
\put(-8,-19.5){\makebox(0,0)[cc]{\(\varkappa=\)}}
\put(0,-19.5){\makebox(0,0)[cc]{\((0000)\)}}
\put(-2,-19.5){\makebox(0,0)[cc]{\((1000)\)}}
\put(-4,-19.5){\makebox(0,0)[cc]{\((2000)\)}}
\put(-6,-19.5){\makebox(0,0)[cc]{\((2000)\)}}
\put(2,-19.5){\makebox(0,0)[cc]{\((0000)\)}}
\put(4,-19.5){\makebox(0,0)[cc]{\((0000)\)}}
\put(6,-19.5){\makebox(0,0)[cc]{\((0000)\)}}
\put(-8.7,-19.7){\framebox(15.4,1){}}


\put(-4,-4){\oval(1.5,2)}
\multiput(-4,-4)(0,-4){4}{\oval(5.6,1.5)}
\multiput(4,-8)(0,-4){3}{\oval(5.9,1.5)}
\put(-6,-5.1){\makebox(0,0)[cc]{\underbar{\textbf{331\,191(79.7\%)}}}}
\put(-3.5,-5.2){\makebox(0,0)[cc]{\underbar{\textbf{122\,955(29.6\%)}}}}
\put(-6,-9.1){\makebox(0,0)[cc]{\underbar{\textbf{11\,780(2.8\%)}}}}
\put(6,-9.1){\makebox(0,0)[cc]{\underbar{\textbf{26\,678(6.4\%)}}}}
\put(-6,-13.1){\makebox(0,0)[cc]{\underbar{\textbf{391}}}}
\put(6,-13.1){\makebox(0,0)[cc]{\underbar{\textbf{921}}}}
\put(-6,-17.1){\makebox(0,0)[cc]{\underbar{\textbf{12}}}}
\put(6,-17.1){\makebox(0,0)[cc]{\underbar{\textbf{25}}}}

\end{picture}

%% file: MinDiscTyp33TreeCc1.tex

\setlength{\unitlength}{0.8cm}
\begin{picture}(18,21)(-11,-20)

\put(-10,0.5){\makebox(0,0)[cb]{order \(3^n\)}}

\put(-10,0){\line(0,-1){16}}
\multiput(-10.1,0)(0,-2){9}{\line(1,0){0.2}}

\put(-10.2,0){\makebox(0,0)[rc]{\(9\)}}
\put(-9.8,0){\makebox(0,0)[lc]{\(3^2\)}}
\put(-10.2,-2){\makebox(0,0)[rc]{\(27\)}}
\put(-9.8,-2){\makebox(0,0)[lc]{\(3^3\)}}
\put(-10.2,-4){\makebox(0,0)[rc]{\(81\)}}
\put(-9.8,-4){\makebox(0,0)[lc]{\(3^4\)}}
\put(-10.2,-6){\makebox(0,0)[rc]{\(243\)}}
\put(-9.8,-6){\makebox(0,0)[lc]{\(3^5\)}}
\put(-10.2,-8){\makebox(0,0)[rc]{\(729\)}}
\put(-9.8,-8){\makebox(0,0)[lc]{\(3^6\)}}
\put(-10.2,-10){\makebox(0,0)[rc]{\(2\,187\)}}
\put(-9.8,-10){\makebox(0,0)[lc]{\(3^7\)}}
\put(-10.2,-12){\makebox(0,0)[rc]{\(6\,561\)}}
\put(-9.8,-12){\makebox(0,0)[lc]{\(3^8\)}}
\put(-10.2,-14){\makebox(0,0)[rc]{\(19\,683\)}}
\put(-9.8,-14){\makebox(0,0)[lc]{\(3^9\)}}
\put(-10.2,-16){\makebox(0,0)[rc]{\(59\,049\)}}
\put(-9.8,-16){\makebox(0,0)[lc]{\(3^{10}\)}}

\put(-10,-16){\vector(0,-1){2}}

\put(-8,0.5){\makebox(0,0)[cb]{\(\tau(1)=\)}}

\put(-8,0){\makebox(0,0)[cc]{\((1)\)}}
\put(-8,-2){\makebox(0,0)[cc]{\((1^2)\)}}
\put(-8,-4){\makebox(0,0)[cc]{\((21)\)}}
\put(-8,-6){\makebox(0,0)[cc]{\((2^2)\)}}
\put(-8,-8){\makebox(0,0)[cc]{\((32)\)}}
\put(-8,-10){\makebox(0,0)[cc]{\((3^2)\)}}
\put(-8,-12){\makebox(0,0)[cc]{\((43)\)}}
\put(-8,-14){\makebox(0,0)[cc]{\((4^2)\)}}
\put(-8,-16){\makebox(0,0)[cc]{\((54)\)}}

\put(-8,-17){\makebox(0,0)[cc]{\textbf{IPAD}}}
\put(-8.5,-17.2){\framebox(1,18){}}

\put(3.7,-3){\vector(0,1){1}}
\put(3.9,-3){\makebox(0,0)[lc]{depth \(1\)}}
\put(3.7,-3){\vector(0,-1){1}}

\put(-5.1,-8){\vector(0,1){2}}
\put(-5.3,-7){\makebox(0,0)[rc]{period length \(2\)}}
\put(-5.1,-8){\vector(0,-1){2}}

\put(-0.1,-0.1){\framebox(0.2,0.2){}}
\put(-2.1,-0.1){\framebox(0.2,0.2){}}

\multiput(0,-2)(0,-2){8}{\circle*{0.2}}
\multiput(-2,-2)(0,-2){8}{\circle*{0.2}}
\multiput(-4,-4)(0,-2){7}{\circle*{0.2}}
\multiput(-6,-4)(0,-4){4}{\circle*{0.2}}

\multiput(2,-6)(0,-2){6}{\circle*{0.1}}
\multiput(4,-6)(0,-2){6}{\circle*{0.1}}
\multiput(6,-6)(0,-2){6}{\circle*{0.1}}

\multiput(0,0)(0,-2){8}{\line(0,-1){2}}
\multiput(0,0)(0,-2){8}{\line(-1,-1){2}}
\multiput(0,-2)(0,-2){7}{\line(-2,-1){4}}
\multiput(0,-2)(0,-4){4}{\line(-3,-1){6}}
\multiput(0,-4)(0,-2){6}{\line(1,-1){2}}
\multiput(0,-4)(0,-2){6}{\line(2,-1){4}}
\multiput(0,-4)(0,-2){6}{\line(3,-1){6}}

\put(0,-16){\vector(0,-1){2}}
\put(0.2,-17.5){\makebox(0,0)[lc]{infinite}}
\put(0.2,-18){\makebox(0,0)[lc]{mainline}}
\put(-0.2,-18.3){\makebox(0,0)[rc]{\(\mathcal{T}^1(\langle 9,2\rangle)\)}}

\put(0,0.1){\makebox(0,0)[lb]{\(=C_3\times C_3\)}}
\put(-2.5,0.1){\makebox(0,0)[rb]{\(C_9=\)}}
\put(-0.8,0.3){\makebox(0,0)[rb]{abelian}}
\put(0.1,-1.9){\makebox(0,0)[lb]{\(=G^3_0(0,0)\)}}
\put(-2.5,-1.9){\makebox(0,0)[rb]{\(G^3_0(0,1)=\)}}
\put(-4,-3.9){\makebox(0,0)[lb]{\(=\mathrm{Syl}_3A_9\)}}
\put(-4,-4.5){\makebox(0,0)[cc]{\textbf{IPAD} \quad \(\tau(1)=(1^3)\)}}
\put(-5.5,-4.7){\framebox(2.9,0.5){}}

\put(0.2,-2.2){\makebox(0,0)[lt]{bifurcation from \(\mathcal{G}(3,1)\)}}
\put(1.5,-2.5){\makebox(0,0)[lt]{to \(\mathcal{G}(3,2)\)}}

\put(-2.1,0.1){\makebox(0,0)[rb]{\(\langle 1\rangle\)}}
\put(-0.1,0.1){\makebox(0,0)[rb]{\(\langle 2\rangle\)}}

\put(-1.5,-1){\makebox(0,0)[rc]{branch \(\mathcal{B}(2)\)}}
\put(-2.1,-1.9){\makebox(0,0)[rb]{\(\langle 4\rangle\)}}
\put(-0.1,-1.9){\makebox(0,0)[rb]{\(\langle 3\rangle\)}}

\put(-4.5,-3){\makebox(0,0)[cc]{\(\mathcal{B}(3)\)}}
\put(-6.1,-3.9){\makebox(0,0)[rb]{\(\langle 8\rangle\)}}
\put(-4.1,-3.9){\makebox(0,0)[rb]{\(\langle 7\rangle\)}}
\put(-2.1,-3.9){\makebox(0,0)[rb]{\(\langle 10\rangle\)}}
\put(-0.1,-3.9){\makebox(0,0)[rb]{\(\langle 9\rangle\)}}

\put(-4.1,-5.9){\makebox(0,0)[rb]{\(\langle 25\rangle\)}}
\put(-2.1,-5.9){\makebox(0,0)[rb]{\(\langle 27\rangle\)}}
\put(-0.1,-5.9){\makebox(0,0)[rb]{\(\langle 26\rangle\)}}

\put(4.5,-5){\makebox(0,0)[cc]{\(\mathcal{B}(4)\)}}
\put(2.1,-5.9){\makebox(0,0)[lb]{\(\langle 28\rangle\)}}
\put(4.1,-5.9){\makebox(0,0)[lb]{\(\langle 30\rangle\)}}
\put(6.1,-5.9){\makebox(0,0)[lb]{\(\langle 29\rangle\)}}

\put(-6.1,-7.9){\makebox(0,0)[rb]{\(\langle 98\rangle\)}}
\put(-4.1,-7.9){\makebox(0,0)[rb]{\(\langle 97\rangle\)}}
\put(-2.1,-7.9){\makebox(0,0)[rb]{\(\langle 96\rangle\)}}
\put(-0.1,-7.9){\makebox(0,0)[rb]{\(\langle 95\rangle\)}}

\put(4.5,-7){\makebox(0,0)[cc]{\(\mathcal{B}(5)\)}}
\put(2.1,-7.9){\makebox(0,0)[lb]{\(\langle 100\rangle\)}}
\put(4.1,-7.9){\makebox(0,0)[lb]{\(\langle 99\rangle\)}}
\put(6.1,-7.9){\makebox(0,0)[lb]{\(\langle 101\rangle\)}}

\put(-4.1,-9.9){\makebox(0,0)[rb]{\(\langle 388\rangle\)}}
\put(-2.1,-9.9){\makebox(0,0)[rb]{\(\langle 387\rangle\)}}
\put(-0.1,-9.9){\makebox(0,0)[rb]{\(\langle 386\rangle\)}}

\put(4.5,-9){\makebox(0,0)[cc]{\(\mathcal{B}(6)\)}}
\put(2.1,-9.9){\makebox(0,0)[lb]{\(\langle 390\rangle\)}}
\put(4.1,-9.9){\makebox(0,0)[lb]{\(\langle 389\rangle\)}}
\put(6.1,-9.9){\makebox(0,0)[lb]{\(\langle 391\rangle\)}}

\put(-5.8,-11.9){\makebox(0,0)[rb]{\(\langle 2224\rangle\)}}
\put(-4,-11.9){\makebox(0,0)[rb]{\(\langle 2223\rangle\)}}
\put(-1.8,-11.9){\makebox(0,0)[rb]{\(\langle 2222\rangle\)}}
\put(-0.1,-11.9){\makebox(0,0)[rb]{\(\langle 2221\rangle\)}}

\put(4.5,-11){\makebox(0,0)[cc]{\(\mathcal{B}(7)\)}}
\put(2.1,-11.9){\makebox(0,0)[lb]{\(\langle 2226\rangle\)}}
\put(4,-11.9){\makebox(0,0)[lb]{\(\langle 2225\rangle\)}}
\put(5.8,-11.9){\makebox(0,0)[lb]{\(\langle 2227\rangle\)}}

\put(4.5,-13){\makebox(0,0)[cc]{\(\mathcal{B}(8)\)}}
\put(4.5,-15){\makebox(0,0)[cc]{\(\mathcal{B}(9)\)}}

\put(0.1,-16.2){\makebox(0,0)[ct]{\(G^n_0(0,0)\)}}
\put(-2,-16.2){\makebox(0,0)[ct]{\(G^n_0(0,1)\)}}
\put(-4,-16.2){\makebox(0,0)[ct]{\(G^n_0(1,0)\)}}
\put(-6,-16.2){\makebox(0,0)[ct]{\(G^n_0(-1,0)\)}}
\put(2,-16.2){\makebox(0,0)[ct]{\(G^n_1(0,-1)\)}}
\put(4,-16.2){\makebox(0,0)[ct]{\(G^n_1(0,0)\)}}
\put(6,-16.2){\makebox(0,0)[ct]{\(G^n_1(0,1)\)}}

\put(-3,-17.7){\makebox(0,0)[ct]{with abelian maximal subgroup}}
\put(4,-17.7){\makebox(0,0)[ct]{without abelian maximal subgroup}}

\put(2.5,0){\makebox(0,0)[cc]{\textbf{IPOD}}}
\put(3.5,0){\makebox(0,0)[cc]{a.1}}
\put(2.5,-0.5){\makebox(0,0)[cc]{\(\varkappa=\)}}
\put(3.5,-0.5){\makebox(0,0)[cc]{\((0000)\)}}
\put(1.8,-0.7){\framebox(2.6,1){}}
\put(-6,-2){\makebox(0,0)[cc]{\textbf{IPOD}}}
\put(-5,-2){\makebox(0,0)[cc]{A.1}}
\put(-6,-2.5){\makebox(0,0)[cc]{\(\varkappa=\)}}
\put(-5,-2.5){\makebox(0,0)[cc]{\((1111)\)}}
\put(-6.7,-2.7){\framebox(2.6,1){}}

\put(-8,-19){\makebox(0,0)[cc]{\textbf{IPOD}}}
\put(0,-19){\makebox(0,0)[cc]{a.1}}
\put(-2,-19){\makebox(0,0)[cc]{a.2}}
\put(-4,-19){\makebox(0,0)[cc]{a.3}}
\put(-6,-19){\makebox(0,0)[cc]{a.3}}
\put(2,-19){\makebox(0,0)[cc]{a.1}}
\put(4,-19){\makebox(0,0)[cc]{a.1}}
\put(6,-19){\makebox(0,0)[cc]{a.1}}
\put(-8,-19.5){\makebox(0,0)[cc]{\(\varkappa=\)}}
\put(0,-19.5){\makebox(0,0)[cc]{\((0000)\)}}
\put(-2,-19.5){\makebox(0,0)[cc]{\((1000)\)}}
\put(-4,-19.5){\makebox(0,0)[cc]{\((2000)\)}}
\put(-6,-19.5){\makebox(0,0)[cc]{\((2000)\)}}
\put(2,-19.5){\makebox(0,0)[cc]{\((0000)\)}}
\put(4,-19.5){\makebox(0,0)[cc]{\((0000)\)}}
\put(6,-19.5){\makebox(0,0)[cc]{\((0000)\)}}
\put(-8.7,-19.7){\framebox(15.4,1){}}


\put(-4,-4){\oval(1.5,2)}
\multiput(-6,-4)(4,0){2}{\oval(1.5,1.5)}
\multiput(-5,-8)(0,-4){3}{\oval(3.6,1.5)}
\multiput(-2,-8)(0,-4){3}{\oval(1.5,1.5)}
\multiput(4,-8)(0,-4){3}{\oval(5.9,1.5)}
\put(-6,-5.1){\makebox(0,0)[cc]{\underbar{\textbf{+32\,009}}}}
\put(-4,-5.2){\makebox(0,0)[cc]{\underbar{\textbf{+142\,097}}}}
\put(-2,-5.1){\makebox(0,0)[cc]{\underbar{\textbf{+72\,329}}}}
\put(-6,-9.1){\makebox(0,0)[cc]{\underbar{\textbf{+494\,236}}}}
\put(-2,-9.1){\makebox(0,0)[cc]{\underbar{\textbf{+790\,085}}}}
\put(6,-9.1){\makebox(0,0)[cc]{\underbar{\textbf{+62\,501}}}}
\put(-6,-13.1){\makebox(0,0)[cc]{\underbar{\textbf{+10\,200\,108}}}}
\put(-2,-13.1){\makebox(0,0)[cc]{\underbar{\textbf{+14\,458\,876}}}}
\put(6,-13.1){\makebox(0,0)[cc]{\underbar{\textbf{+2\,905\,160}}}}
\put(-6,-17.1){\makebox(0,0)[cc]{\underbar{\textbf{+208\,540\,653}}}}
\put(-2,-17.1){\makebox(0,0)[cc]{\underbar{\textbf{+37\,304\,664}}}}
\put(6,-17.1){\makebox(0,0)[cc]{\underbar{\textbf{+40\,980\,808}}}}

\end{picture}

%% file: TreeQSecE.tex

\setlength{\unitlength}{0.8cm}
\begin{picture}(17,22.5)(-9,-21.5)

\put(-8,0.5){\makebox(0,0)[cb]{order \(3^n\)}}
\put(-8,0){\line(0,-1){20}}
\multiput(-8.1,0)(0,-2){11}{\line(1,0){0.2}}
\put(-8.2,0){\makebox(0,0)[rc]{\(243\)}}
\put(-7.8,0){\makebox(0,0)[lc]{\(3^5\)}}
\put(-8.2,-2){\makebox(0,0)[rc]{\(729\)}}
\put(-7.8,-2){\makebox(0,0)[lc]{\(3^6\)}}
\put(-8.2,-4){\makebox(0,0)[rc]{\(2\,187\)}}
\put(-7.8,-4){\makebox(0,0)[lc]{\(3^7\)}}
\put(-8.2,-6){\makebox(0,0)[rc]{\(6\,561\)}}
\put(-7.8,-6){\makebox(0,0)[lc]{\(3^8\)}}
\put(-8.2,-8){\makebox(0,0)[rc]{\(19\,683\)}}
\put(-7.8,-8){\makebox(0,0)[lc]{\(3^9\)}}
\put(-8.2,-10){\makebox(0,0)[rc]{\(59\,049\)}}
\put(-7.8,-10){\makebox(0,0)[lc]{\(3^{10}\)}}
\put(-8.2,-12){\makebox(0,0)[rc]{\(177\,147\)}}
\put(-7.8,-12){\makebox(0,0)[lc]{\(3^{11}\)}}
\put(-8.2,-14){\makebox(0,0)[rc]{\(531\,441\)}}
\put(-7.8,-14){\makebox(0,0)[lc]{\(3^{12}\)}}
\put(-8.2,-16){\makebox(0,0)[rc]{\(1\,594\,323\)}}
\put(-7.8,-16){\makebox(0,0)[lc]{\(3^{13}\)}}
\put(-8.2,-18){\makebox(0,0)[rc]{\(4\,782\,969\)}}
\put(-7.8,-18){\makebox(0,0)[lc]{\(3^{14}\)}}
\put(-8.2,-20){\makebox(0,0)[rc]{\(14\,348\,907\)}}
\put(-7.8,-20){\makebox(0,0)[lc]{\(3^{15}\)}}
\put(-8,-20){\vector(0,-1){2}}

\put(-6,0.5){\makebox(0,0)[cb]{\(\tau_1(1)=\)}}
\put(-6,0){\makebox(0,0)[cc]{\((21)\)}}
\put(-6,-2){\makebox(0,0)[cc]{\((2^2)\)}}
\put(-6,-4){\makebox(0,0)[cc]{\((32)\)}}
\put(-6,-6){\makebox(0,0)[cc]{\((3^2)\)}}
\put(-6,-8){\makebox(0,0)[cc]{\((43)\)}}
\put(-6,-10){\makebox(0,0)[cc]{\((4^2)\)}}
\put(-6,-12){\makebox(0,0)[cc]{\((54)\)}}
\put(-6,-14){\makebox(0,0)[cc]{\((5^2)\)}}
\put(-6,-16){\makebox(0,0)[cc]{\((65)\)}}
\put(-6,-18){\makebox(0,0)[cc]{\((6^2)\)}}
\put(-6,-20){\makebox(0,0)[cc]{\((76)\)}}
\put(-6,-21){\makebox(0,0)[cc]{\textbf{IPAD}}}
\put(-6.5,-21.2){\framebox(1,22){}}

\put(7.6,-7){\vector(0,1){3}}
\put(7.8,-7){\makebox(0,0)[lc]{depth \(3\)}}
\put(7.6,-7){\vector(0,-1){3}}

\put(-3.1,-8){\vector(0,1){2}}
\put(-3.3,-8){\makebox(0,0)[rc]{period length \(2\)}}
\put(-3.1,-8){\vector(0,-1){2}}

\put(0.7,-2){\makebox(0,0)[lc]{bifurcation from}}
\put(0.7,-2.3){\makebox(0,0)[lc]{\(\mathcal{G}(3,2)\) to \(\mathcal{G}(3,3)\)}}

\multiput(0,0)(0,-2){10}{\circle*{0.2}}
\multiput(0,0)(0,-2){9}{\line(0,-1){2}}
\multiput(-1,-2)(0,-2){10}{\circle*{0.2}}
\multiput(-2,-2)(0,-2){10}{\circle*{0.2}}
\multiput(1.95,-4.05)(0,-2){9}{\framebox(0.1,0.1){}}
\multiput(3,-2)(0,-2){10}{\circle*{0.2}}
\multiput(0,0)(0,-2){10}{\line(-1,-2){1}}
\multiput(0,0)(0,-2){10}{\line(-1,-1){2}}
\multiput(0,-2)(0,-2){9}{\line(1,-1){2}}
\multiput(0,0)(0,-2){10}{\line(3,-2){3}}
\multiput(-3.05,-4.05)(-1,0){2}{\framebox(0.1,0.1){}}
\multiput(3.95,-4.05)(0,-2){9}{\framebox(0.1,0.1){}}
\multiput(5,-4)(0,-2){9}{\circle*{0.1}}
\multiput(6,-4)(0,-2){9}{\circle*{0.1}}
\multiput(-1,-2)(-1,0){2}{\line(-1,-1){2}}
\multiput(3,-2)(0,-2){9}{\line(1,-2){1}}
\multiput(3,-2)(0,-2){9}{\line(1,-1){2}}
\multiput(3,-2)(0,-2){9}{\line(3,-2){3}}
\multiput(6.95,-6.05)(0,-2){8}{\framebox(0.1,0.1){}}
\multiput(6,-4)(0,-2){8}{\line(1,-2){1}}

\put(2,-0.5){\makebox(0,0)[lc]{branch}}
\put(2,-0.8){\makebox(0,0)[lc]{\(\mathcal{B}(5)\)}}
\put(2,-2.8){\makebox(0,0)[lc]{\(\mathcal{B}(6)\)}}
\put(2,-4.8){\makebox(0,0)[lc]{\(\mathcal{B}(7)\)}}
\put(2,-6.8){\makebox(0,0)[lc]{\(\mathcal{B}(8)\)}}
\put(2,-8.8){\makebox(0,0)[lc]{\(\mathcal{B}(9)\)}}
\put(2,-10.8){\makebox(0,0)[lc]{\(\mathcal{B}(10)\)}}
\put(2,-12.8){\makebox(0,0)[lc]{\(\mathcal{B}(11)\)}}
\put(2,-14.8){\makebox(0,0)[lc]{\(\mathcal{B}(12)\)}}
\put(2,-16.8){\makebox(0,0)[lc]{\(\mathcal{B}(13)\)}}
\put(2,-18.8){\makebox(0,0)[lc]{\(\mathcal{B}(14)\)}}

\put(-0.1,0.3){\makebox(0,0)[rc]{\(\langle 6\rangle\)}}

\put(-2.1,-1.8){\makebox(0,0)[rc]{\(\langle 50\rangle\)}}
\put(-1.1,-1.8){\makebox(0,0)[rc]{\(\langle 51\rangle\)}}
\put(0.1,-1.8){\makebox(0,0)[lc]{\(\langle 49\rangle\)}}
\put(3.1,-1.8){\makebox(0,0)[lc]{\(\langle 48\rangle\)}}

\put(-4.1,-3.5){\makebox(0,0)[cc]{\(\langle 292\rangle\)}}
\put(-3.1,-3.5){\makebox(0,0)[cc]{\(\langle 293\rangle\)}}
\put(-2.1,-3.3){\makebox(0,0)[cc]{\(\langle 289\rangle\)}}
\put(-2.1,-3.5){\makebox(0,0)[cc]{\(\langle 290\rangle\)}}
\put(-1.1,-3.5){\makebox(0,0)[cc]{\(\langle 288\rangle\)}}
\put(0.1,-3.5){\makebox(0,0)[lc]{\(\langle 285\rangle\)}}
\put(2.2,-3.3){\makebox(0,0)[cc]{\(\langle 284\rangle\)}}
\put(2.2,-3.5){\makebox(0,0)[cc]{\(\langle 291\rangle\)}}
\put(3.2,-3.3){\makebox(0,0)[cc]{\(\langle 286\rangle\)}}
\put(3.2,-3.5){\makebox(0,0)[cc]{\(\langle 287\rangle\)}}
\put(4.2,-3.3){\makebox(0,0)[cc]{\(\langle 276\rangle\)}}
\put(4.2,-3.5){\makebox(0,0)[cc]{\(\langle 283\rangle\)}}
\put(5.2,-3.1){\makebox(0,0)[cc]{\(\langle 280\rangle\)}}
\put(5.2,-3.3){\makebox(0,0)[cc]{\(\langle 281\rangle\)}}
\put(5.2,-3.5){\makebox(0,0)[cc]{\(\langle 282\rangle\)}}
\put(6.2,-3.1){\makebox(0,0)[cc]{\(\langle 277\rangle\)}}
\put(6.2,-3.3){\makebox(0,0)[cc]{\(\langle 278\rangle\)}}
\put(6.2,-3.5){\makebox(0,0)[cc]{\(\langle 279\rangle\)}}

\put(-2.4,-5.8){\makebox(0,0)[cc]{\(\langle 2027\rangle\)}}
\put(-1.4,-5.8){\makebox(0,0)[cc]{\(\langle 2026\rangle\)}}
\put(0.1,-5.8){\makebox(0,0)[lc]{\(\langle 2024\rangle\)}}
\put(2.2,-5.8){\makebox(0,0)[cc]{\(\langle 2028\rangle\)}}
\put(3.2,-5.8){\makebox(0,0)[cc]{\(\langle 2025\rangle\)}}
\put(4.2,-5.6){\makebox(0,0)[cc]{\(\langle 2033\rangle\)}}
\put(4.2,-5.8){\makebox(0,0)[cc]{\(\langle 2038\rangle\)}}
\put(5.2,-4.9){\makebox(0,0)[cc]{\(\langle 2031\rangle\)}}
\put(5.2,-5.1){\makebox(0,0)[cc]{\(\langle 2032\rangle\)}}
\put(5.2,-5.3){\makebox(0,0)[cc]{\(\langle 2036\rangle\)}}
\put(5.2,-5.5){\makebox(0,0)[cc]{\(\langle 2037\rangle\)}}
\put(6.2,-4.9){\makebox(0,0)[cc]{\(\langle 2029\rangle\)}}
\put(6.2,-5.1){\makebox(0,0)[cc]{\(\langle 2030\rangle\)}}
\put(6.2,-5.3){\makebox(0,0)[cc]{\(\langle 2034\rangle\)}}
\put(6.2,-5.5){\makebox(0,0)[cc]{\(\langle 2035\rangle\)}}
\put(7.2,-4.9){\makebox(0,0)[cc]{\(\langle 2015\rangle\)}}
\put(7.2,-5.2){\makebox(0,0)[cc]{\(\cdots\)}}
\put(7.2,-5.5){\makebox(0,0)[cc]{\(\langle 2023\rangle\)}}

\put(2.1,-3.8){\makebox(0,0)[lc]{\(*2\)}}
\multiput(-2.1,-3.8)(0,-4){5}{\makebox(0,0)[rc]{\(2*\)}}
\multiput(3.1,-3.8)(0,-4){5}{\makebox(0,0)[lc]{\(*2\)}}
\put(4.1,-3.8){\makebox(0,0)[lc]{\(*2\)}}
\multiput(5.1,-5.8)(0,-4){4}{\makebox(0,0)[lc]{\(*2\)}}
\multiput(5.5,-9.3)(0,-4){3}{\makebox(0,0)[lc]{\(\#2\)}}
\multiput(6.1,-5.8)(0,-4){4}{\makebox(0,0)[lc]{\(*2\)}}
\multiput(5.1,-3.8)(0,-4){5}{\makebox(0,0)[lc]{\(*3\)}}
\multiput(6.1,-3.8)(0,-4){5}{\makebox(0,0)[lc]{\(*3\)}}
\multiput(7.1,-5.8)(0,-2){4}{\makebox(0,0)[lc]{\(*3\)}}

\put(-3,-21){\makebox(0,0)[cc]{\textbf{IPOD}}}
\put(-2,-21){\makebox(0,0)[cc]{E.14}}
\put(-1,-21){\makebox(0,0)[cc]{E.6}}
\put(0,-21){\makebox(0,0)[cc]{c.18}}
\put(2,-21){\makebox(0,0)[cc]{c.18}}
\put(3.1,-21){\makebox(0,0)[cc]{H.4}}
\put(4,-21){\makebox(0,0)[cc]{H.4}}
\put(5,-21){\makebox(0,0)[cc]{H.4}}
\put(6,-21){\makebox(0,0)[cc]{H.4}}
\put(7,-21){\makebox(0,0)[cc]{H.4}}
\put(-3,-21.5){\makebox(0,0)[cc]{\(\varkappa_1=\)}}
\put(-2,-21.5){\makebox(0,0)[cc]{\((3122)\)}}
\put(-1,-21.5){\makebox(0,0)[cc]{\((1122)\)}}
\put(0,-21.5){\makebox(0,0)[cc]{\((0122)\)}}
\put(2,-21.5){\makebox(0,0)[cc]{\((0122)\)}}
\put(3.1,-21.5){\makebox(0,0)[cc]{\((2122)\)}}
\put(4,-21.5){\makebox(0,0)[cc]{\((2122)\)}}
\put(5,-21.5){\makebox(0,0)[cc]{\((2122)\)}}
\put(6,-21.5){\makebox(0,0)[cc]{\((2122)\)}}
\put(7,-21.5){\makebox(0,0)[cc]{\((2122)\)}}
\put(-3.8,-21.7){\framebox(11.6,1){}}

\put(0,-18){\vector(0,-1){2}}
\put(0.2,-19.4){\makebox(0,0)[lc]{infinite}}
\put(0.2,-19.9){\makebox(0,0)[lc]{mainline}}
\put(1.8,-20.4){\makebox(0,0)[rc]{\(\mathcal{T}^2(\langle 243,6\rangle)\)}}


\multiput(0,-2)(0,-4){3}{\oval(1,1)}
\put(0.1,-1.3){\makebox(0,0)[lc]{\underbar{\textbf{+534\,824}}}}
\put(0.1,-5.3){\makebox(0,0)[lc]{\underbar{\textbf{+13\,714\,789}}}}
\put(0.1,-9.3){\makebox(0,0)[lc]{\underbar{\textbf{+174\,458\,681}}}}

\multiput(-1,-4)(0,-4){2}{\oval(1,1)}
\put(-1,-4.8){\makebox(0,0)[lc]{\underbar{\textbf{+5\,264\,069}}}}
\put(-1,-8.8){\makebox(0,0)[lc]{\underbar{\textbf{+75\,393\,861}}}}
\put(-1,-12.8){\makebox(0,0)[lc]{\underbar{\textbf{}}}}
\put(-1,-16.8){\makebox(0,0)[lc]{\underbar{\textbf{}}}}
\multiput(-2,-4)(0,-4){3}{\oval(1,1)}
\put(-2,-4.8){\makebox(0,0)[rc]{\underbar{\textbf{+3\,918\,837}}}}
\put(-2,-8.8){\makebox(0,0)[rc]{\underbar{\textbf{+70\,539\,596}}}}
\put(-2,-12.8){\makebox(0,0)[rc]{\underbar{\textbf{+336\,698\,284}}}}
\put(-2,-16.8){\makebox(0,0)[rc]{\underbar{\textbf{}}}}
\multiput(6,-6)(0,-4){2}{\oval(1,1)}
\put(6,-6.8){\makebox(0,0)[rc]{\underbar{\textbf{+1\,162\,949}}}}
\put(6,-10.8){\makebox(0,0)[rc]{\underbar{\textbf{+126\,691\,957}}}}
\put(6,-14.8){\makebox(0,0)[rc]{\underbar{\textbf{}}}}
\put(6,-18.8){\makebox(0,0)[rc]{\underbar{\textbf{}}}}

\end{picture}

%% file: TreeUSecE.tex

\setlength{\unitlength}{0.8cm}
\begin{picture}(17,22.5)(-9,-21.5)

\put(-8,0.5){\makebox(0,0)[cb]{order \(3^n\)}}
\put(-8,0){\line(0,-1){20}}
\multiput(-8.1,0)(0,-2){11}{\line(1,0){0.2}}
\put(-8.2,0){\makebox(0,0)[rc]{\(243\)}}
\put(-7.8,0){\makebox(0,0)[lc]{\(3^5\)}}
\put(-8.2,-2){\makebox(0,0)[rc]{\(729\)}}
\put(-7.8,-2){\makebox(0,0)[lc]{\(3^6\)}}
\put(-8.2,-4){\makebox(0,0)[rc]{\(2\,187\)}}
\put(-7.8,-4){\makebox(0,0)[lc]{\(3^7\)}}
\put(-8.2,-6){\makebox(0,0)[rc]{\(6\,561\)}}
\put(-7.8,-6){\makebox(0,0)[lc]{\(3^8\)}}
\put(-8.2,-8){\makebox(0,0)[rc]{\(19\,683\)}}
\put(-7.8,-8){\makebox(0,0)[lc]{\(3^9\)}}
\put(-8.2,-10){\makebox(0,0)[rc]{\(59\,049\)}}
\put(-7.8,-10){\makebox(0,0)[lc]{\(3^{10}\)}}
\put(-8.2,-12){\makebox(0,0)[rc]{\(177\,147\)}}
\put(-7.8,-12){\makebox(0,0)[lc]{\(3^{11}\)}}
\put(-8.2,-14){\makebox(0,0)[rc]{\(531\,441\)}}
\put(-7.8,-14){\makebox(0,0)[lc]{\(3^{12}\)}}
\put(-8.2,-16){\makebox(0,0)[rc]{\(1\,594\,323\)}}
\put(-7.8,-16){\makebox(0,0)[lc]{\(3^{13}\)}}
\put(-8.2,-18){\makebox(0,0)[rc]{\(4\,782\,969\)}}
\put(-7.8,-18){\makebox(0,0)[lc]{\(3^{14}\)}}
\put(-8.2,-20){\makebox(0,0)[rc]{\(14\,348\,907\)}}
\put(-7.8,-20){\makebox(0,0)[lc]{\(3^{15}\)}}
\put(-8,-20){\vector(0,-1){2}}

\put(-6,0.5){\makebox(0,0)[cb]{\(\tau_1(2)=\)}}
\put(-6,0){\makebox(0,0)[cc]{\((21)\)}}
\put(-6,-2){\makebox(0,0)[cc]{\((2^2)\)}}
\put(-6,-4){\makebox(0,0)[cc]{\((32)\)}}
\put(-6,-6){\makebox(0,0)[cc]{\((3^2)\)}}
\put(-6,-8){\makebox(0,0)[cc]{\((43)\)}}
\put(-6,-10){\makebox(0,0)[cc]{\((4^2)\)}}
\put(-6,-12){\makebox(0,0)[cc]{\((54)\)}}
\put(-6,-14){\makebox(0,0)[cc]{\((5^2)\)}}
\put(-6,-16){\makebox(0,0)[cc]{\((65)\)}}
\put(-6,-18){\makebox(0,0)[cc]{\((6^2)\)}}
\put(-6,-20){\makebox(0,0)[cc]{\((76)\)}}
\put(-6,-21){\makebox(0,0)[cc]{\textbf{IPAD}}}
\put(-6.5,-21.2){\framebox(1,22){}}

\put(7.6,-7){\vector(0,1){3}}
\put(7.8,-7){\makebox(0,0)[lc]{depth \(3\)}}
\put(7.6,-7){\vector(0,-1){3}}

\put(-3.1,-8){\vector(0,1){2}}
\put(-3.3,-8){\makebox(0,0)[rc]{period length \(2\)}}
\put(-3.1,-8){\vector(0,-1){2}}

\put(0.7,-2){\makebox(0,0)[lc]{bifurcation from}}
\put(0.7,-2.3){\makebox(0,0)[lc]{\(\mathcal{G}(3,2)\) to \(\mathcal{G}(3,3)\)}}

\multiput(0,0)(0,-2){10}{\circle*{0.2}}
\multiput(0,0)(0,-2){9}{\line(0,-1){2}}
\multiput(-1,-2)(0,-2){10}{\circle*{0.2}}
\multiput(-2,-2)(0,-2){10}{\circle*{0.2}}
\multiput(1.95,-4.05)(0,-2){9}{\framebox(0.1,0.1){}}
\multiput(3,-2)(0,-2){10}{\circle*{0.2}}
\multiput(0,0)(0,-2){10}{\line(-1,-2){1}}
\multiput(0,0)(0,-2){10}{\line(-1,-1){2}}
\multiput(0,-2)(0,-2){9}{\line(1,-1){2}}
\multiput(0,0)(0,-2){10}{\line(3,-2){3}}
\multiput(-3.05,-4.05)(-1,0){2}{\framebox(0.1,0.1){}}
\multiput(3.95,-6.05)(0,-2){8}{\framebox(0.1,0.1){}}
\multiput(5,-6)(0,-2){8}{\circle*{0.1}}
\multiput(6,-4)(0,-2){9}{\circle*{0.1}}
\multiput(-1,-2)(-1,0){2}{\line(-1,-1){2}}
\multiput(3,-4)(0,-2){8}{\line(1,-2){1}}
\multiput(3,-4)(0,-2){8}{\line(1,-1){2}}
\multiput(3,-2)(0,-2){9}{\line(3,-2){3}}
\multiput(6.95,-6.05)(0,-2){8}{\framebox(0.1,0.1){}}
\multiput(6,-4)(0,-2){8}{\line(1,-2){1}}

\put(2,-0.5){\makebox(0,0)[lc]{branch}}
\put(2,-0.8){\makebox(0,0)[lc]{\(\mathcal{B}(5)\)}}
\put(2,-2.8){\makebox(0,0)[lc]{\(\mathcal{B}(6)\)}}
\put(2,-4.8){\makebox(0,0)[lc]{\(\mathcal{B}(7)\)}}
\put(2,-6.8){\makebox(0,0)[lc]{\(\mathcal{B}(8)\)}}
\put(2,-8.8){\makebox(0,0)[lc]{\(\mathcal{B}(9)\)}}
\put(2,-10.8){\makebox(0,0)[lc]{\(\mathcal{B}(10)\)}}
\put(2,-12.8){\makebox(0,0)[lc]{\(\mathcal{B}(11)\)}}
\put(2,-14.8){\makebox(0,0)[lc]{\(\mathcal{B}(12)\)}}
\put(2,-16.8){\makebox(0,0)[lc]{\(\mathcal{B}(13)\)}}
\put(2,-18.8){\makebox(0,0)[lc]{\(\mathcal{B}(14)\)}}

\put(-0.1,0.3){\makebox(0,0)[rc]{\(\langle 8\rangle\)}}

\put(-2.1,-1.8){\makebox(0,0)[rc]{\(\langle 53\rangle\)}}
\put(-1.1,-1.8){\makebox(0,0)[rc]{\(\langle 55\rangle\)}}
\put(0.1,-1.8){\makebox(0,0)[lc]{\(\langle 54\rangle\)}}
\put(3.1,-1.8){\makebox(0,0)[lc]{\(\langle 52\rangle\)}}

\put(-4.1,-3.5){\makebox(0,0)[cc]{\(\langle 300\rangle\)}}
\put(-3.1,-3.5){\makebox(0,0)[cc]{\(\langle 309\rangle\)}}
\put(-2.1,-3.3){\makebox(0,0)[cc]{\(\langle 302\rangle\)}}
\put(-2.1,-3.5){\makebox(0,0)[cc]{\(\langle 306\rangle\)}}
\put(-1.1,-3.5){\makebox(0,0)[cc]{\(\langle 304\rangle\)}}
\put(0.1,-3.5){\makebox(0,0)[lc]{\(\langle 303\rangle\)}}
\put(2.2,-3.3){\makebox(0,0)[cc]{\(\langle 307\rangle\)}}
\put(2.2,-3.5){\makebox(0,0)[cc]{\(\langle 308\rangle\)}}
\put(3.2,-3.3){\makebox(0,0)[cc]{\(\langle 301\rangle\)}}
\put(3.2,-3.5){\makebox(0,0)[cc]{\(\langle 305\rangle\)}}
\put(6.2,-2.9){\makebox(0,0)[cc]{\(\langle 294\rangle\)}}
\put(6.2,-3.2){\makebox(0,0)[cc]{\(\cdots\)}}
\put(6.2,-3.5){\makebox(0,0)[cc]{\(\langle 299\rangle\)}}

\put(-2.4,-5.8){\makebox(0,0)[cc]{\(\langle 2053\rangle\)}}
\put(-1.4,-5.8){\makebox(0,0)[cc]{\(\langle 2051\rangle\)}}
\put(0.1,-5.8){\makebox(0,0)[lc]{\(\langle 2050\rangle\)}}
\put(2.2,-5.8){\makebox(0,0)[cc]{\(\langle 2054\rangle\)}}
\put(3.2,-5.8){\makebox(0,0)[cc]{\(\langle 2052\rangle\)}}
\put(4.2,-5.6){\makebox(0,0)[cc]{\(\langle 2049\rangle\)}}
\put(4.2,-5.8){\makebox(0,0)[cc]{\(\langle 2059\rangle\)}}
\put(5.2,-4.9){\makebox(0,0)[cc]{\(\langle 2045\rangle\)}}
\put(5.2,-5.1){\makebox(0,0)[cc]{\(\langle 2047\rangle\)}}
\put(5.2,-5.3){\makebox(0,0)[cc]{\(\langle 2055\rangle\)}}
\put(5.2,-5.5){\makebox(0,0)[cc]{\(\langle 2057\rangle\)}}
\put(6.2,-4.9){\makebox(0,0)[cc]{\(\langle 2046\rangle\)}}
\put(6.2,-5.1){\makebox(0,0)[cc]{\(\langle 2048\rangle\)}}
\put(6.2,-5.3){\makebox(0,0)[cc]{\(\langle 2056\rangle\)}}
\put(6.2,-5.5){\makebox(0,0)[cc]{\(\langle 2058\rangle\)}}
\put(7.2,-4.9){\makebox(0,0)[cc]{\(\langle 2039\rangle\)}}
\put(7.2,-5.2){\makebox(0,0)[cc]{\(\cdots\)}}
\put(7.2,-5.5){\makebox(0,0)[cc]{\(\langle 2044\rangle\)}}

\put(2.1,-3.8){\makebox(0,0)[lc]{\(*2\)}}
\multiput(-2.1,-3.8)(0,-4){5}{\makebox(0,0)[rc]{\(2*\)}}
\multiput(3.1,-3.8)(0,-4){5}{\makebox(0,0)[lc]{\(*2\)}}
\put(6.1,-3.8){\makebox(0,0)[lc]{\(*6\)}}
\multiput(5.1,-5.8)(0,-4){4}{\makebox(0,0)[lc]{\(*2\)}}
\multiput(5.5,-9.3)(0,-4){3}{\makebox(0,0)[lc]{\(\#4\)}}
\multiput(6.1,-5.8)(0,-4){4}{\makebox(0,0)[lc]{\(*2\)}}
\multiput(5.1,-7.8)(0,-4){4}{\makebox(0,0)[lc]{\(*3\)}}
\multiput(6.1,-7.8)(0,-4){4}{\makebox(0,0)[lc]{\(*3\)}}
\multiput(7.1,-7.8)(0,-2){4}{\makebox(0,0)[lc]{\(*2\)}}

\put(-3,-21){\makebox(0,0)[cc]{\textbf{IPOD}}}
\put(-2,-21){\makebox(0,0)[cc]{E.9}}
\put(-1,-21){\makebox(0,0)[cc]{E.8}}
\put(0,-21){\makebox(0,0)[cc]{c.21}}
\put(2,-21){\makebox(0,0)[cc]{c.21}}
\put(3.1,-21){\makebox(0,0)[cc]{G.16}}
\put(4,-21){\makebox(0,0)[cc]{G.16}}
\put(5,-21){\makebox(0,0)[cc]{G.16}}
\put(6,-21){\makebox(0,0)[cc]{G.16}}
\put(7,-21){\makebox(0,0)[cc]{G.16}}
\put(-3,-21.5){\makebox(0,0)[cc]{\(\varkappa_1=\)}}
\put(-2,-21.5){\makebox(0,0)[cc]{\((2334)\)}}
\put(-1,-21.5){\makebox(0,0)[cc]{\((2234)\)}}
\put(0,-21.5){\makebox(0,0)[cc]{\((2034)\)}}
\put(2,-21.5){\makebox(0,0)[cc]{\((2034)\)}}
\put(3.1,-21.5){\makebox(0,0)[cc]{\((2134)\)}}
\put(4,-21.5){\makebox(0,0)[cc]{\((2134)\)}}
\put(5,-21.5){\makebox(0,0)[cc]{\((2134)\)}}
\put(6,-21.5){\makebox(0,0)[cc]{\((2134)\)}}
\put(7,-21.5){\makebox(0,0)[cc]{\((2134)\)}}
\put(-3.8,-21.7){\framebox(11.6,1){}}

\put(0,-18){\vector(0,-1){2}}
\put(0.2,-19.4){\makebox(0,0)[lc]{infinite}}
\put(0.2,-19.9){\makebox(0,0)[lc]{mainline}}
\put(1.8,-20.4){\makebox(0,0)[rc]{\(\mathcal{T}^2(\langle 243,8\rangle)\)}}


\multiput(0,-2)(0,-4){3}{\oval(1,1)}
\put(0.1,-1.3){\makebox(0,0)[lc]{\underbar{\textbf{+540\,365}}}}
\put(0.1,-5.3){\makebox(0,0)[lc]{\underbar{\textbf{+1\,001\,957}}}}
\put(0.1,-9.3){\makebox(0,0)[lc]{\underbar{\textbf{+116\,043\,324}}}}

\multiput(-1,-4)(0,-4){2}{\oval(1,1)}
\put(-1,-4.8){\makebox(0,0)[lc]{\underbar{\textbf{+6\,098\,360}}}}
\put(-1,-8.8){\makebox(0,0)[lc]{\underbar{\textbf{+26\,889\,637}}}}
\put(-1,-12.8){\makebox(0,0)[lc]{\underbar{\textbf{}}}}
\put(-1,-16.8){\makebox(0,0)[lc]{\underbar{\textbf{}}}}
\multiput(-2,-4)(0,-4){3}{\oval(1,1)}
\put(-2,-4.8){\makebox(0,0)[rc]{\underbar{\textbf{+342\,664}}}}
\put(-2,-8.8){\makebox(0,0)[rc]{\underbar{\textbf{+79\,043\,324}}}}
\put(-2,-12.8){\makebox(0,0)[rc]{\underbar{\textbf{+124\,813\,084}}}}
\put(-2,-16.8){\makebox(0,0)[rc]{\underbar{\textbf{}}}}
\put(-2,-19.3){\makebox(0,0)[rc]{\underbar{\textbf{}}}}
\multiput(6,-6)(0,-4){3}{\oval(1,1)}
\put(6,-6.8){\makebox(0,0)[rc]{\underbar{\textbf{+8\,711\,453}}}}
\put(6,-10.8){\makebox(0,0)[rc]{\underbar{\textbf{+59\,479\,964}}}}
\put(6,-14.8){\makebox(0,0)[rc]{\underbar{\textbf{+157\,369\,512}}}}
\put(6,-18.8){\makebox(0,0)[rc]{\underbar{\textbf{}}}}

\end{picture}

%% file: SporCc2.tex

\setlength{\unitlength}{1.0cm}
\begin{picture}(16,17)(0,-14)

\put(0,2.5){\makebox(0,0)[cb]{Order \(3^n\)}}
\put(0,2){\line(0,-1){12}}
\multiput(-0.1,2)(0,-2){7}{\line(1,0){0.2}}
\put(-0.2,2){\makebox(0,0)[rc]{\(9\)}}
\put(0.2,2){\makebox(0,0)[lc]{\(3^2\)}}
\put(-0.2,0){\makebox(0,0)[rc]{\(27\)}}
\put(0.2,0){\makebox(0,0)[lc]{\(3^3\)}}
\put(-0.2,-2){\makebox(0,0)[rc]{\(81\)}}
\put(0.2,-2){\makebox(0,0)[lc]{\(3^4\)}}
\put(-0.2,-4){\makebox(0,0)[rc]{\(243\)}}
\put(0.2,-4){\makebox(0,0)[lc]{\(3^5\)}}
\put(-0.2,-6){\makebox(0,0)[rc]{\(729\)}}
\put(0.2,-6){\makebox(0,0)[lc]{\(3^6\)}}
\put(-0.2,-8){\makebox(0,0)[rc]{\(2\,187\)}}
\put(0.2,-8){\makebox(0,0)[lc]{\(3^7\)}}
\put(-0.2,-10){\makebox(0,0)[rc]{\(6\,561\)}}
\put(0.2,-10){\makebox(0,0)[lc]{\(3^8\)}}
\put(0,-10){\vector(0,-1){2}}

\put(2.2,2.2){\makebox(0,0)[lc]{\(C_3\times C_3\)}}
\put(1.8,2.2){\makebox(0,0)[rc]{\(\langle 2\rangle\)}}
\put(1.9,1.9){\framebox(0.2,0.2){}}
\put(2,2){\line(0,-1){2}}
\put(2,0){\circle{0.2}}
\put(2.2,0.2){\makebox(0,0)[lc]{\(G^3_0(0,0)\)}}
\put(1.8,0.2){\makebox(0,0)[rc]{\(\langle 3\rangle\)}}

\put(2,0){\line(1,-4){1}}
\put(2,0){\line(1,-2){2}}
\put(2,0){\line(1,-1){4}}
\put(2,0){\line(3,-2){6}}
\put(2,0){\line(2,-1){8}}
\put(2,0){\line(5,-2){10}}
\put(2,0){\line(3,-1){12}}
\put(7,-1){\makebox(0,0)[lc]{Edges of depth \(2\) forming the interface}}
\put(8,-1.5){\makebox(0,0)[lc]{between \(\mathcal{G}(3,1)\) and \(\mathcal{G}(3,2)\)}}

\put(2,-3.6){\makebox(0,0)[cc]{\(\Phi_6\)}}
\multiput(3,-4)(1,0){2}{\circle*{0.2}}
\put(3.1,-3.9){\makebox(0,0)[lb]{\(\langle 5\rangle\)}}
\put(4.1,-3.9){\makebox(0,0)[lb]{\(\langle 7\rangle\)}}
\multiput(6,-4)(2,0){5}{\circle*{0.2}}
\put(6.1,-3.9){\makebox(0,0)[lb]{\(\langle 9\rangle\)}}
\put(8.1,-3.9){\makebox(0,0)[lb]{\(\langle 4\rangle\)}}
\put(10.1,-3.9){\makebox(0,0)[lb]{\(\langle 3\rangle\)}}
\put(12.1,-3.9){\makebox(0,0)[lb]{\(\langle 6\rangle\)}}
\put(14.1,-3.9){\makebox(0,0)[lb]{\(\langle 8\rangle\)}}

\put(3.2,-4){\oval(0.8,1)}
\put(2.7,-4.1){\makebox(0,0)[rc]{\underbar{\textbf{+422\,573}}}}
\put(2.7,-4.5){\makebox(0,0)[rc]{\underbar{\textbf{13\,712(3.3\%)}}}}
\put(4.2,-4){\oval(0.8,1)}
\put(4.7,-4.1){\makebox(0,0)[lc]{\underbar{\textbf{+631\,769}}}}
\put(4.7,-4.5){\makebox(0,0)[lc]{\underbar{\textbf{6\,691(1.6\%)}}}}

\put(2,-5){\makebox(0,0)[cc]{\textbf{IPOD}}}
\put(3,-5){\makebox(0,0)[cc]{D.10}}
\put(4,-5){\makebox(0,0)[cc]{D.5}}
\put(2,-5.5){\makebox(0,0)[cc]{\(\varkappa=\)}}
\put(3,-5.5){\makebox(0,0)[cc]{\((3144)\)}}
\put(4,-5.5){\makebox(0,0)[cc]{\((1133)\)}}
\put(1.5,-5.75){\framebox(3,1){}}

\put(6.5,-6.3){\makebox(0,0)[cc]{\(\Phi_{43}\)}}
\put(6,-4){\line(0,-1){2}}
\put(5.9,-5.9){\makebox(0,0)[rc]{\(\langle 57\rangle\)}}
\put(6.5,-5.9){\makebox(0,0)[rc]{\(\langle 56\rangle\)}}
\put(5.9,-7.8){\makebox(0,0)[rc]{\(\langle 311\rangle\)}}
\put(6.5,-7.8){\makebox(0,0)[rc]{\(\langle 310\rangle\)}}
\put(6,-4){\line(1,-4){0.5}}
\multiput(6,-6)(0.5,0){2}{\circle*{0.1}}
\multiput(6,-6)(0.5,0){2}{\line(0,-1){2}}
\multiput(5.95,-8.05)(0.5,0){2}{\framebox(0.1,0.1){}}

\put(9,-6.3){\makebox(0,0)[cc]{\(\Phi_{42}\)}}
\put(8,-4){\line(0,-1){2}}
\put(7.9,-5.9){\makebox(0,0)[rc]{\(\langle 45\rangle\)}}
\put(8.5,-5.9){\makebox(0,0)[rc]{\(\langle 44\rangle\)}}
\put(8.9,-5.9){\makebox(0,0)[rc]{\(\langle 46\rangle\)}}
\put(9.4,-5.9){\makebox(0,0)[rc]{\(\langle 47\rangle\)}}
\put(7.9,-7.1){\makebox(0,0)[rc]{\(\langle 270\rangle\)}}
\put(7.9,-7.3){\makebox(0,0)[rc]{\(\langle 271\rangle\)}}
\put(7.9,-7.5){\makebox(0,0)[rc]{\(\langle 272\rangle\)}}
\put(7.9,-7.7){\makebox(0,0)[rc]{\(\langle 273\rangle\)}}
\put(8.5,-7.1){\makebox(0,0)[rc]{\(\langle 266\rangle\)}}
\put(8.5,-7.3){\makebox(0,0)[rc]{\(\langle 267\rangle\)}}
\put(8.5,-7.5){\makebox(0,0)[rc]{\(\langle 268\rangle\)}}
\put(8.5,-7.7){\makebox(0,0)[rc]{\(\langle 269\rangle\)}}
\put(9.0,-7.8){\makebox(0,0)[rc]{\(\langle 274\rangle\)}}
\put(9.6,-7.8){\makebox(0,0)[rc]{\(\langle 275\rangle\)}}
\put(8,-4){\line(1,-4){0.5}}
\put(8,-4){\line(1,-2){1}}
\put(8,-4){\line(3,-4){1.5}}
\multiput(8,-6)(0.5,0){4}{\circle*{0.1}}
\multiput(8,-6)(0.5,0){4}{\line(0,-1){2}}
\multiput(7.95,-8.05)(0.5,0){4}{\framebox(0.1,0.1){}}
\multiput(7.9,-7.9)(0.5,0){2}{\makebox(0,0)[rc]{\(4*\)}}

\put(5.7,-6){\oval(1,1)}
\put(5.2,-6.3){\makebox(0,0)[rc]{\underbar{\textbf{+214\,712}}}}
\put(5.2,-6.7){\makebox(0,0)[rc]{\underbar{\textbf{1\,636}}}}
\put(7.7,-6){\oval(1,1)}
\put(7.9,-4.9){\makebox(0,0)[rc]{\underbar{\textbf{+957\,013}}}}
\put(7.9,-5.3){\makebox(0,0)[rc]{\underbar{\textbf{6\,583(1.6\%)}}}}

\put(5,-8.5){\makebox(0,0)[cc]{\textbf{IPOD}}}
\put(6.25,-8.5){\makebox(0,0)[cc]{G.19}}
\put(8.75,-8.5){\makebox(0,0)[cc]{H.4}}
\put(5,-9){\makebox(0,0)[cc]{\(\varkappa=\)}}
\put(6.25,-9){\makebox(0,0)[cc]{\((2143)\)}}
\put(8.75,-9){\makebox(0,0)[cc]{\((4111)\)}}
\put(4.5,-9.25){\framebox(5,1){}}


\put(10.5,-6.3){\makebox(0,0)[cc]{\(\Phi_{40}\)}}
\put(11,-6.3){\makebox(0,0)[cc]{\(\Phi_{41}\)}}
\multiput(10,-4)(0,-2){3}{\line(0,-1){2}}
\multiput(10,-6)(0,-2){3}{\circle*{0.2}}
\put(10.1,-5.9){\makebox(0,0)[lc]{\(\langle 40\rangle\)}}
\put(10.6,-5.3){\makebox(0,0)[lc]{\(\langle 34\rangle\)}}
\put(10.6,-5.5){\makebox(0,0)[lc]{\(\langle 35\rangle\)}}
\put(10.6,-5.7){\makebox(0,0)[lc]{\(\langle 36\rangle\)}}
\put(11.1,-5.3){\makebox(0,0)[lc]{\(\langle 37\rangle\)}}
\put(11.1,-5.5){\makebox(0,0)[lc]{\(\langle 38\rangle\)}}
\put(11.1,-5.7){\makebox(0,0)[lc]{\(\langle 39\rangle\)}}
\put(10.1,-7.8){\makebox(0,0)[lc]{\(\langle 247\rangle\)}}
\put(10.1,-9.8){\makebox(0,0)[lc]{\(\langle 1988\rangle\)}}
\put(10,-4){\line(1,-4){0.5}}
\put(10,-4){\line(1,-2){1}}
\multiput(10.5,-6)(0.5,0){2}{\circle*{0.1}}
\multiput(10.6,-5.9)(0.5,0){2}{\makebox(0,0)[lc]{\(*3\)}}
\put(10,-10){\vector(0,-1){2}}
\put(9.8,-12.2){\makebox(0,0)[ct]{\(\mathcal{T}^2(\langle 729,40\rangle)\)}}


\multiput(12,-4)(0,-2){3}{\line(0,-1){2}}
\multiput(12,-6)(0,-2){3}{\circle*{0.2}}
\put(12.1,-5.9){\makebox(0,0)[lc]{\(\langle 49\rangle\)}}
\put(12.1,-7.8){\makebox(0,0)[lc]{\(\langle 285\rangle\)}}
\put(12.1,-9.8){\makebox(0,0)[lc]{\(\langle 2024\rangle\)}}
\put(12,-10){\vector(0,-1){2}}
\put(12,-12.2){\makebox(0,0)[ct]{\(\mathcal{T}^2(\langle 243,6\rangle)\)}}

\put(13.1,-6.3){\makebox(0,0)[cc]{\(\Phi_{23}\)}}
\put(12,-11){\makebox(0,0)[cc]{\(3\) metabelian mainlines}}

\multiput(14,-4)(0,-2){3}{\line(0,-1){2}}
\multiput(14,-6)(0,-2){3}{\circle*{0.2}}
\put(14.1,-5.9){\makebox(0,0)[lc]{\(\langle 54\rangle\)}}
\put(14.1,-7.8){\makebox(0,0)[lc]{\(\langle 303\rangle\)}}
\put(14.1,-9.8){\makebox(0,0)[lc]{\(\langle 2050\rangle\)}}
\put(14,-10){\vector(0,-1){2}}
\put(14,-12.2){\makebox(0,0)[ct]{\(\mathcal{T}^2(\langle 243,8\rangle)\)}}

\put(9,-13){\makebox(0,0)[cc]{\textbf{IPOD}}}
\put(10,-13){\makebox(0,0)[cc]{b.10}}
\put(12,-13){\makebox(0,0)[cc]{c.18}}
\put(14,-13){\makebox(0,0)[cc]{c.21}}
\put(9,-13.5){\makebox(0,0)[cc]{\(\varkappa=\)}}
\put(10,-13.5){\makebox(0,0)[cc]{\((2100)\)}}
\put(12,-13.5){\makebox(0,0)[cc]{\((0122)\)}}
\put(14,-13.5){\makebox(0,0)[cc]{\((2034)\)}}
\put(8.5,-13.75){\framebox(6,1){}}

\end{picture}

%% file: TreeH4.tex

\setlength{\unitlength}{0.7cm}
\begin{picture}(18,26.5)(-6,-25.5)

\put(-5,0.5){\makebox(0,0)[cb]{Order}}
\put(-5,0){\line(0,-1){24}}
\multiput(-5.1,0)(0,-2){13}{\line(1,0){0.2}}
\put(-5.2,0){\makebox(0,0)[rc]{\(243\)}}
\put(-4.8,0){\makebox(0,0)[lc]{\(3^5\)}}
\put(-5.2,-2){\makebox(0,0)[rc]{\(729\)}}
\put(-4.8,-2){\makebox(0,0)[lc]{\(3^6\)}}
\put(-5.2,-4){\makebox(0,0)[rc]{\(2\,187\)}}
\put(-4.8,-4){\makebox(0,0)[lc]{\(3^7\)}}
\put(-5.2,-6){\makebox(0,0)[rc]{\(6\,561\)}}
\put(-4.8,-6){\makebox(0,0)[lc]{\(3^8\)}}
\put(-5.2,-8){\makebox(0,0)[rc]{\(19\,683\)}}
\put(-4.8,-8){\makebox(0,0)[lc]{\(3^9\)}}
\put(-5.2,-10){\makebox(0,0)[rc]{\(59\,049\)}}
\put(-4.8,-10){\makebox(0,0)[lc]{\(3^{10}\)}}
\put(-5.2,-12){\makebox(0,0)[rc]{\(177\,147\)}}
\put(-4.8,-12){\makebox(0,0)[lc]{\(3^{11}\)}}
\put(-5.2,-14){\makebox(0,0)[rc]{\(531\,441\)}}
\put(-4.8,-14){\makebox(0,0)[lc]{\(3^{12}\)}}
\put(-5.2,-16){\makebox(0,0)[rc]{\(1\,594\,323\)}}
\put(-4.8,-16){\makebox(0,0)[lc]{\(3^{13}\)}}
\put(-5.2,-18){\makebox(0,0)[rc]{\(4\,782\,969\)}}
\put(-4.8,-18){\makebox(0,0)[lc]{\(3^{14}\)}}
\put(-5.2,-20){\makebox(0,0)[rc]{\(14\,348\,907\)}}
\put(-4.8,-20){\makebox(0,0)[lc]{\(3^{15}\)}}
\put(-5.2,-22){\makebox(0,0)[rc]{\(43\,046\,721\)}}
\put(-4.8,-22){\makebox(0,0)[lc]{\(3^{16}\)}}
\put(-5.2,-24){\makebox(0,0)[rc]{\(129\,140\,163\)}}
\put(-4.8,-24){\makebox(0,0)[lc]{\(3^{17}\)}}
\put(-5,-24){\vector(0,-1){2}}

\multiput(0,0)(0,-2){2}{\line(0,-1){2}}
\multiput(0,0)(0,-2){1}{\circle*{0.2}}
\multiput(0,-2)(0,-2){1}{\circle*{0.1}}
\put(0.1,0.2){\makebox(0,0)[lb]{\(\langle 4\rangle\)}}
\put(0.1,-1.8){\makebox(0,0)[lb]{\(\langle 45\rangle\) (not coclass-settled)}}
\put(1.1,-2.8){\makebox(0,0)[lb]{\(1^{\text{st}}\) bifurcation}}
\multiput(0,-2)(0,-4){1}{\line(-3,-2){3}}
\multiput(0,-2)(0,-4){1}{\line(-1,-1){2}}
\multiput(0,-2)(0,-4){1}{\line(-1,-2){1}}
\multiput(-3.1,-4.1)(1,0){4}{\framebox(0.2,0.2){}}
\put(-3,-4.2){\makebox(0,0)[ct]{\(\langle 270\rangle\)}}
\put(-2,-4.2){\makebox(0,0)[ct]{\(\langle 271\rangle\)}}
\put(-1,-4.2){\makebox(0,0)[ct]{\(\langle 272\rangle\)}}
\put(0,-4.2){\makebox(0,0)[ct]{\(\langle 273\rangle\)}}
\put(-3,-4.6){\makebox(0,0)[ct]{\(T_{0,1}\)}}
\put(-2,-4.6){\makebox(0,0)[ct]{\(T_{0,2}\)}}
\put(-1,-4.6){\makebox(0,0)[ct]{\(T_{0,3}\)}}
\put(0,-4.6){\makebox(0,0)[ct]{\(T_{0,4}\)}}
\put(0,-2){\line(1,-2){2}}
\put(0,-2){\line(1,-4){1}}
\multiput(0.9,-6.1)(1,0){1}{\framebox(0.2,0.2){}}
\put(1,-6.2){\makebox(0,0)[ct]{\(606\)}}
\put(1,-6.6){\makebox(0,0)[ct]{\(S_0\)}}

\put(-3,-4){\oval(0.8,4.4)}
\put(-3,-6.5){\makebox(0,0)[cc]{\underbar{\textbf{+2\,303\,112}}}}
\put(-3,-7.0){\makebox(0,0)[cc]{\underbar{\textbf{5(19\%)}}}}

\put(-1.5,-4){\oval(1.8,2.5)}
\put(-1.5,-5.5){\makebox(0,0)[cc]{\underbar{\textbf{+2\,023\,845}}}}
\put(-1.5,-6.0){\makebox(0,0)[cc]{\underbar{\textbf{8(29\%)}}}}

\put(0.3,-4){\oval(1.5,2.0)}
\put(2,-3.7){\makebox(0,0)[cc]{\underbar{\textbf{+957\,013}}}}
\put(2,-4.2){\makebox(0,0)[cc]{\underbar{\textbf{11(41\%)}}}}

\put(2.4,-6){\oval(1.3,2.0)}
\put(4.4,-5.7){\makebox(0,0)[cc]{\underbar{\textbf{+2\,852\,733\ ?}}}}
\put(4.4,-6.2){\makebox(0,0)[cc]{\underbar{\textbf{3(11\%)}}}}

\put(1,-6){\oval(1.3,2.0)}
\put(1,-7.2){\makebox(0,0)[cc]{\underbar{\textbf{-3\,896}}}}
\put(1,-7.7){\makebox(0,0)[cc]{\underbar{\textbf{3(50\%)}}}}

\put(3,-12){\oval(1.3,2.0)}
\put(3,-13.2){\makebox(0,0)[cc]{\underbar{\textbf{-6\,583\ ?}}}}
\put(3,-13.7){\makebox(0,0)[cc]{\underbar{\textbf{3(50\%)}}}}

\multiput(2,-6)(0,-2){2}{\line(0,-1){2}}
\put(1.9,-6.1){\framebox(0.2,0.2){}}
\put(1.95,-8.05){\framebox(0.1,0.1){}}
\put(2.1,-5.8){\makebox(0,0)[lb]{\(605\)}}
\put(2.1,-7.8){\makebox(0,0)[lb]{\(1;2\) (not coclass-settled)}}
\put(3.1,-8.8){\makebox(0,0)[lb]{\(2^{\text{nd}}\) bifurcation}}
\multiput(2,-8)(0,-4){1}{\line(-3,-2){3}}
\multiput(2,-8)(0,-4){1}{\line(-1,-1){2}}
\multiput(2,-8)(0,-4){1}{\line(-1,-2){1}}
\multiput(-1.1,-10.1)(1,0){4}{\framebox(0.2,0.2){}}
\put(-1,-10.2){\makebox(0,0)[ct]{\(1;1\)}}
\put(0,-10.2){\makebox(0,0)[ct]{\(1;2\)}}
\put(1,-10.2){\makebox(0,0)[ct]{\(1;3\)}}
\put(2,-10.2){\makebox(0,0)[ct]{\(1;4\)}}
\put(-1,-10.6){\makebox(0,0)[ct]{\(T_{1,1}\)}}
\put(0,-10.6){\makebox(0,0)[ct]{\(T_{1,2}\)}}
\put(1,-10.6){\makebox(0,0)[ct]{\(T_{1,3}\)}}
\put(2,-10.6){\makebox(0,0)[ct]{\(T_{1,4}\)}}
\put(2,-8){\line(1,-2){2}}
\put(2,-8){\line(1,-4){1}}
\multiput(2.9,-12.1)(1,0){1}{\framebox(0.2,0.2){}}
\put(3,-12.2){\makebox(0,0)[ct]{\(2;2\)}}
\put(3,-12.6){\makebox(0,0)[ct]{\(S_1\)}}

\multiput(4,-12)(0,-2){2}{\line(0,-1){2}}
\put(3.9,-12.1){\framebox(0.2,0.2){}}
\put(3.95,-14.05){\framebox(0.1,0.1){}}
\put(4.1,-11.8){\makebox(0,0)[lb]{\(2;1\)}}
\put(4.1,-13.8){\makebox(0,0)[lb]{\(1;2\) (not coclass-settled)}}
\put(5.1,-14.8){\makebox(0,0)[lb]{\(3^{\text{rd}}\) bifurcation}}
\multiput(4,-14)(0,-4){1}{\line(-3,-2){3}}
\multiput(4,-14)(0,-4){1}{\line(-1,-1){2}}
\multiput(4,-14)(0,-4){1}{\line(-1,-2){1}}
\multiput(0.9,-16.1)(1,0){4}{\framebox(0.2,0.2){}}
\put(1,-16.2){\makebox(0,0)[ct]{\(1;1\)}}
\put(2,-16.2){\makebox(0,0)[ct]{\(1;2\)}}
\put(3,-16.2){\makebox(0,0)[ct]{\(1;3\)}}
\put(4,-16.2){\makebox(0,0)[ct]{\(1;4\)}}
\put(4,-14){\line(1,-2){2}}
\put(4,-14){\line(1,-4){1}}
\multiput(4.9,-18.1)(1,0){1}{\framebox(0.2,0.2){}}
\put(5,-18.2){\makebox(0,0)[ct]{\(2;2\)}}
\put(5,-18.6){\makebox(0,0)[ct]{\(S_2\)}}

\multiput(6,-18)(0,-2){2}{\line(0,-1){2}}
\put(5.9,-18.1){\framebox(0.2,0.2){}}
\put(5.95,-20.05){\framebox(0.1,0.1){}}
\put(6.1,-17.8){\makebox(0,0)[lb]{\(2;1\)}}
\put(6.1,-19.8){\makebox(0,0)[lb]{\(1;1\) (not coclass-settled)}}
\put(7.1,-20.8){\makebox(0,0)[lb]{\(4^{\text{th}}\) bifurcation}}
\multiput(6,-20)(0,-4){1}{\line(-3,-2){3}}
\multiput(6,-20)(0,-4){1}{\line(-1,-1){2}}
\multiput(6,-20)(0,-4){1}{\line(-1,-2){1}}
\multiput(2.9,-22.1)(1,0){4}{\framebox(0.2,0.2){}}
\put(3,-22.2){\makebox(0,0)[ct]{\(1;1\)}}
\put(4,-22.2){\makebox(0,0)[ct]{\(1;2\)}}
\put(5,-22.2){\makebox(0,0)[ct]{\(1;3\)}}
\put(6,-22.2){\makebox(0,0)[ct]{\(1;4\)}}
\put(6,-20){\line(1,-2){2}}
\put(6,-20){\line(1,-4){1}}
\multiput(6.9,-24.1)(1,0){1}{\framebox(0.2,0.2){}}
\put(7,-24.2){\makebox(0,0)[ct]{\(2;2\)}}
\put(7,-24.6){\makebox(0,0)[ct]{\(S_3\)}}

\multiput(7.9,-24.1)(0,-2){1}{\framebox(0.2,0.2){}}
\put(8.1,-23.8){\makebox(0,0)[lb]{\(2;1\)}}

\end{picture}

%% file: TreeG19.tex

\setlength{\unitlength}{0.8cm}
\begin{picture}(15,21)(0,-20)

\put(0,0.5){\makebox(0,0)[cb]{Order \(3^n\)}}

\put(0,0){\line(0,-1){18}}
\multiput(-0.1,0)(0,-2){10}{\line(1,0){0.2}}

\put(-0.2,0){\makebox(0,0)[rc]{\(243\)}}
\put(0.2,0){\makebox(0,0)[lc]{\(3^5\)}}
\put(-0.2,-2){\makebox(0,0)[rc]{\(729\)}}
\put(0.2,-2){\makebox(0,0)[lc]{\(3^6\)}}
\put(-0.2,-4){\makebox(0,0)[rc]{\(2\,187\)}}
\put(0.2,-4){\makebox(0,0)[lc]{\(3^7\)}}
\put(-0.2,-6){\makebox(0,0)[rc]{\(6\,561\)}}
\put(0.2,-6){\makebox(0,0)[lc]{\(3^8\)}}
\put(-0.2,-8){\makebox(0,0)[rc]{\(19\,683\)}}
\put(0.2,-8){\makebox(0,0)[lc]{\(3^9\)}}
\put(-0.2,-10){\makebox(0,0)[rc]{\(59\,049\)}}
\put(0.2,-10){\makebox(0,0)[lc]{\(3^{10}\)}}
\put(-0.2,-12){\makebox(0,0)[rc]{\(177\,147\)}}
\put(0.2,-12){\makebox(0,0)[lc]{\(3^{11}\)}}
\put(-0.2,-14){\makebox(0,0)[rc]{\(531\,441\)}}
\put(0.2,-14){\makebox(0,0)[lc]{\(3^{12}\)}}
\put(-0.2,-16){\makebox(0,0)[rc]{\(1\,594\,323\)}}
\put(0.2,-16){\makebox(0,0)[lc]{\(3^{13}\)}}
\put(-0.2,-18){\makebox(0,0)[rc]{\(4\,782\,969\)}}
\put(0.2,-18){\makebox(0,0)[lc]{\(3^{14}\)}}

\put(0,-18){\vector(0,-1){2}}


\put(1.8,0.2){\makebox(0,0)[rc]{\(\langle 9\rangle\)}}
\put(2.2,0.2){\makebox(0,0)[lc]{root}}
\put(2,0){\circle*{0.2}}

\put(2,0){\line(0,-1){2}}

\put(1.8,-1.3){\makebox(0,0)[rc]{\(W=\)}}
\put(1.8,-1.8){\makebox(0,0)[rc]{\(\langle 57\rangle\)}}
\put(2.2,-1.8){\makebox(0,0)[lc]{\(1^{\text{st}}\) bifurcation}}
\put(2,-2){\circle*{0.1}}

\put(2,-2){\line(0,-1){2}}

\put(1.8,-3.8){\makebox(0,0)[rc]{\(\langle 311\rangle\)}}
\put(1.9,-4.1){\framebox(0.2,0.2){}}

\put(1.8,-4){\oval(1.6,1.3)}
\put(2,-4.9){\makebox(0,0)[cc]{\underbar{\textbf{+214\,712}}}}
\put(2,-5.3){\makebox(0,0)[cc]{\underbar{\textbf{55(86\%)}}}}



\put(2,-2){\line(1,-2){2}}

\put(4.2,-5.8){\makebox(0,0)[lc]{\(\langle 625\rangle\)}}
\put(3.9,-6.1){\framebox(0.2,0.2){}}

\put(9.4,-5.8){\oval(12,1.3)}
\put(6,-6.8){\makebox(0,0)[cc]{\underbar{\textbf{+24\,126\,593\ ?}}}}
\put(6,-7.2){\makebox(0,0)[cc]{\underbar{\textbf{6(9\%)}}}}

\put(4,-6){\line(0,-1){2}}

\put(4.2,-7.8){\makebox(0,0)[lc]{\(1;2\)}}
\put(3.8,-7.8){\makebox(0,0)[rc]{\(2^{\text{nd}}\) bifurcations}}
\put(3.95,-8.05){\framebox(0.1,0.1){}}

\put(4,-8){\line(1,-2){1}}

\put(5.2,-9.8){\makebox(0,0)[lc]{\(1;1\)}}
\put(4.9,-10.1){\framebox(0.2,0.2){}}

\put(4,-8){\line(0,-1){4}}

\put(4.2,-11.8){\makebox(0,0)[lc]{\(2;1..2\)}}
\put(3.8,-11.8){\makebox(0,0)[rc]{\(2\ast\)}}
\put(3.9,-12.1){\framebox(0.2,0.2){}}

\put(4.2,-12){\oval(2.6,1.3)}
\put(4,-12.9){\makebox(0,0)[cc]{\underbar{\textbf{-12\,067\ ?}}}}
\put(4,-13.3){\makebox(0,0)[cc]{\underbar{\textbf{30(65\%)}}}}


\put(2,-2){\line(1,-1){4}}

\put(6.2,-5.8){\makebox(0,0)[lc]{\(\langle 626\rangle\)}}
\put(6.2,-6.1){\makebox(0,0)[lc]{\(=\Phi\)}}
\put(5.9,-6.1){\framebox(0.2,0.2){}}


\put(2,-2){\line(3,-2){6}}

\put(8.2,-5.8){\makebox(0,0)[lc]{\(\langle 627\rangle\)}}
\put(7.9,-6.1){\framebox(0.2,0.2){}}

\put(8,-6){\line(0,-1){2}}

\put(8.2,-7.8){\makebox(0,0)[lc]{\(1;2\)}}
\put(7.95,-8.05){\framebox(0.1,0.1){}}

\put(8,-8){\line(1,-1){2}}

\put(10.2,-9.8){\makebox(0,0)[lc]{\(1;1\)}}
\put(9.9,-10.1){\framebox(0.2,0.2){}}

\put(8,-8){\line(0,-1){4}}

\put(8.2,-11.8){\makebox(0,0)[lc]{\(2;1\)}}
\put(7.9,-12.1){\framebox(0.2,0.2){}}

\put(8,-12){\line(0,-1){2}}

\put(8.2,-13.8){\makebox(0,0)[lc]{\(1;3\)}}
\put(7.8,-13.8){\makebox(0,0)[rc]{\(3^{\text{rd}}\) bifurcations}}
\put(7.95,-14.05){\framebox(0.1,0.1){}}

\put(8,-14){\line(1,-2){1}}

\put(9.1,-15.8){\makebox(0,0)[lc]{\(1;1..2\)}}
\put(8.8,-15.8){\makebox(0,0)[rc]{\(2\ast\)}}
\put(8.9,-16.1){\framebox(0.2,0.2){}}

\put(8,-14){\line(0,-1){4}}

\put(8.2,-17.8){\makebox(0,0)[lc]{\(2;1..3\)}}
\put(7.8,-17.8){\makebox(0,0)[rc]{\(3\ast\)}}
\put(7.9,-18.1){\framebox(0.2,0.2){}}

\put(8,-8){\line(1,-2){2}}

\put(10.2,-11.8){\makebox(0,0)[lc]{\(2;2\)}}
\put(9.9,-12.1){\framebox(0.2,0.2){}}

\put(10,-12){\line(0,-1){2}}

\put(10.2,-13.8){\makebox(0,0)[lc]{\(1;2\)}}
\put(9.95,-14.05){\framebox(0.1,0.1){}}

\put(10,-14){\line(1,-2){1}}

\put(11.2,-15.8){\makebox(0,0)[lc]{\(1;1\)}}
\put(10.9,-16.1){\framebox(0.2,0.2){}}

\put(10,-14){\line(0,-1){4}}

\put(10.2,-17.8){\makebox(0,0)[lc]{\(2;1..3\)}}
\put(9.8,-17.8){\makebox(0,0)[rc]{\(3\ast\)}}
\put(9.9,-18.1){\framebox(0.2,0.2){}}


\put(2,-2){\line(2,-1){8}}

\put(10.2,-5.8){\makebox(0,0)[lc]{\(\langle 628\rangle\)}}
\put(10.2,-6.1){\makebox(0,0)[lc]{\(=\Psi\)}}
\put(9.9,-6.1){\framebox(0.2,0.2){}}

\put(8,-6){\line(0,-1){2}}

\put(8.2,-7.8){\makebox(0,0)[lc]{\(1;2\)}}
\put(7.95,-8.05){\framebox(0.1,0.1){}}


\put(2,-2){\line(5,-2){10}}

\put(12.2,-5.8){\makebox(0,0)[lc]{\(\langle 629\rangle\)}}
\put(12.2,-6.1){\makebox(0,0)[lc]{\(=Y\)}}
\put(11.9,-6.1){\framebox(0.2,0.2){}}

\put(12,-6){\line(0,-1){2}}

\put(12.2,-7.8){\makebox(0,0)[lc]{\(1;2\)}}
\put(12.2,-8.1){\makebox(0,0)[lc]{\(=Y_1\)}}
\put(11.95,-8.05){\framebox(0.1,0.1){}}

\put(12,-8){\line(1,-2){1}}

\put(13.2,-9.8){\makebox(0,0)[lc]{\(1;1\)}}
\put(12.9,-10.1){\framebox(0.2,0.2){}}

\put(12,-8){\line(0,-1){4}}

\put(12.2,-11.8){\makebox(0,0)[lc]{\(2;1..2\)}}
\put(11.8,-11.8){\makebox(0,0)[rc]{\(2\ast\)}}
\put(11.9,-12.1){\framebox(0.2,0.2){}}

\put(13,-9.8){\oval(1.6,1.3)}
\put(13,-10.7){\makebox(0,0)[cc]{\underbar{\textbf{+21\,974\,161}}}}
\put(13,-11.1){\makebox(0,0)[cc]{\underbar{\textbf{3(5\%)}}}}

\put(12.2,-12){\oval(2.6,1.3)}
\put(12,-12.9){\makebox(0,0)[cc]{\underbar{\textbf{-114\,936}}}}
\put(12,-13.3){\makebox(0,0)[cc]{\underbar{\textbf{7(15\%)}}}}


\put(2,-2){\line(3,-1){12}}

\put(14.2,-5.8){\makebox(0,0)[lc]{\(\langle 630\rangle\)}}
\put(14.2,-6.1){\makebox(0,0)[lc]{\(=Z\)}}
\put(13.95,-6.05){\framebox(0.1,0.1){}}

\put(14,-6.0){\vector(0,-1){2}}
\put(14.4,-7.6){\makebox(0,0)[lc]{\textbf{?}}}

\put(14.4,-7.8){\oval(1.6,1.3)}
\put(14.4,-8.7){\makebox(0,0)[cc]{\underbar{\textbf{-54\,195\ ?}}}}
\put(14.4,-9.1){\makebox(0,0)[cc]{\underbar{\textbf{7(15\%)}}}}

\end{picture}

%% file: 3StageCriteria.bbl
\begin{thebibliography}{XX}
%
\bibitem{Ma7}
D. C. Mayer,
\textit{Index-\(p\) abelianization data of \(p\)-class tower groups},
Adv. Pure Math.
\textbf{5}
(2015)
no. 5,
286--313,
DOI 10.4236/apm.2015.55029,
Special Issue on Number Theory and Cryptography,
April 2015.
%
\bibitem{Ma7b}
D. C. Mayer,
\textit{Index-\(p\) abelianization data of \(p\)-class tower groups},
29i\`emes Journ\'ees Arithm\'etiques (JA 2015),
Univ. of Debrecen,
Hungary,
presentation delivered on July 09, 2015.
%
\bibitem{Ma8}
D. C. Mayer,
\textit{Periodic sequences of \(p\)-class tower groups},
J. Appl. Math. Phys.
\textbf{3}
(2015),
no. 7,
746--756,
DOI 10.4236/jamp.2015.37090.
%
%
\bibitem{Ma9}
D. C. Mayer,
\textit{Artin transfer patterns on descendant trees of finite \(p\)-groups},
Adv. Pure Math.
\textbf{6}
(2016),
no. 2,
66-104,
DOI 10.4236/apm.2016.62008,
Special Issue on Group Theory Research,
January 2016.
%
\bibitem{Ma1}
D. C. Mayer,
\textit{The second \(p\)-class group of a number field},
Int. J. Number Theory
\textbf{8}
(2012),
no. 2,
471--505,
DOI 10.1142/S179304211250025X.
%
\bibitem{SoTa}
A. Scholz und O. Taussky,
\textit{Die Hauptideale der kubischen Klassenk\"orper
imagin\"ar quadratischer Zahlk\"orper:
ihre rechnerische Bestimmung
und ihr Einflu\ss\ auf den Klassenk\"orperturm},
J. Reine Angew. Math.
\textbf{171}
(1934),
19--41.
%
\bibitem{BuMa}
M. R. Bush and D. C. Mayer,
\textit{\(3\)-class field towers of exact length \(3\)},
J. Number Theory
\textbf{147}
(2015),
766--777,
DOI 10.1016/j.jnt.2014.08.010.
%
\bibitem{KoVe}
H. Koch und B. B. Venkov,
\textit{\"Uber den \(p\)-Klassenk\"orperturm eines imagin\"ar-quadratischen Zahlk\"orpers},
Ast\'erisque
\textbf{24--25}
(1975),
57--67.
%
\bibitem{Ar1}
E. Artin,
\textit{Beweis des allgemeinen Reziprozit\"atsgesetzes},
Abh. Math. Sem. Univ. Hamburg
\textbf{5}
(1927),
353--363.
%
\bibitem{Ma5}
D. C. Mayer,
\textit{Quadratic \(p\)-ring spaces for counting dihedral fields},
Int. J. Number Theory
\textbf{10}
(2014),
no. 8,
2205--2242,
DOI 10.1142/S1793042114500754.
%
\bibitem{Ar2}
E. Artin,
\textit{Idealklassen in Oberk\"orpern und allgemeines Reziprozit\"atsgesetz},
Abh. Math. Sem. Univ. Hamburg
\textbf{7}
(1929),
46--51.
%
\bibitem{Fw}
Ph. Furtw\"angler.
\textit{Beweis des Hauptidealsatzes f\"ur die Klassenk\"orper algebraischer Zahlk\"orper},
Abh. Math. Sem. Univ. Hamburg
\textbf{7}
(1929),
14--36.
%
\bibitem{BBH}
N. Boston, M. R. Bush and F. Hajir,
\textit{Heuristics for \(p\)-class towers of imaginary quadratic fields},
to appear in Math. Annalen,
2016.
(arXiv: 1111.4679v2 [math.NT] 10 Dec 2014.)
%
\bibitem{PARI}
The PARI Group, PARI/GP, Version 2.9.0,
Bordeaux,
2016,
\verb+(http://pari.math.u-bordeaux.fr)+.
%
\bibitem{BCP}
W. Bosma, J. Cannon, and C. Playoust,
\textit{The Magma algebra system. I. The user language},
J. Symbolic Comput.
\textbf{24}
(1997),
235--265.
%
\bibitem{BCFS}
W. Bosma, J. J. Cannon, C. Fieker, and A. Steels (eds.),
\textit{Handbook of Magma functions}
(Edition 2.22,
Sydney,
2016).
%
\bibitem{MAGMA}
The MAGMA Group,
\textit{MAGMA Computational Algebra System},
Version 2.22-6,
Sydney,
2016,
\verb+(http://magma.maths.usyd.edu.au)+.
%
\bibitem{Ma2}
D. C. Mayer,
\textit{Transfers of metabelian \(p\)-groups},
Monatsh. Math.
\textbf{166}
(2012),
no. 3--4,
467--495,
DOI 10.1007/s00605-010-0277-x.
%
\bibitem{Ta1}
O. Taussky,
\textit{\"Uber eine Versch\"arfung des Haupidealsatzes f\"ur algebraische Zahlk\"orper},
J. Reine Angew. Math.
\textbf{168}
(1932),
193--210.
%
\bibitem{MaNm}
D. C. Mayer and M. F. Newman,
\textit{Finite \(3\)-groups with transfer kernel type \(\mathrm{F}\)},
in preparation.
%
\bibitem{Ma14}
D. C. Mayer,
\textit{\(p\)-Capitulation over number fields with \(p\)-class rank two},
J. Appl. Math. Phys.
\textbf{4}
(2016),
no. 7,
1280--1293,
DOI 10.4236/jamp.2016.47135.
%
\bibitem{Ma10a}
D. C. Mayer,
\textit{New number fields with known \(p\)-class tower},
22nd Czech and Slovak International Conference on Number Theory (CSICNT 2015),
Liptovsk\'y J\'an, Slovakia,
presentation delivered on August 31, 2015.
%
\bibitem{Ma10}
D. C. Mayer,
\textit{New number fields with known \(p\)-class tower},
Tatra Mt. Math. Pub.,
\textbf{64}
(2015),
21--57,
DOI 10.1515/tmmp-2015-0040,
Special Issue on Number Theory and Cryptology \lq 15.
%
\bibitem{Ma3}
D. C. Mayer,
\textit{Principalization algorithm via class group structure},
J. Th\'eor. Nombres Bordeaux
\textbf{26}
(2014),
no. 2,
415--464,
DOI 10.5802/jtnb.874.
%
\bibitem{BEO1}
H. U. Besche, B. Eick, and E. A. O'Brien,
\textit{A millennium project: constructing small groups},
Int. J. Algebra Comput.
\textbf{12}
(2002),
623-644.
%
\bibitem{BEO2}
H. U. Besche, B. Eick, and E. A. O'Brien,
\textit{The SmallGroups Library --- a Library of Groups of Small Order},
2005,
an accepted and refereed GAP package, available also in MAGMA.
%
\bibitem{Nm}
M. F. Newman,
\textit{Determination of groups of prime-power order},
pp. 73--84,
in: Group Theory, Canberra, 1975,
Lecture Notes in Math.,
vol. \textbf{573},
Springer,
Berlin,
1977.
%
%
\bibitem{Ob}
E. A. O'Brien, 
\textit{The p-group generation algorithm}, 
J. Symbolic Comput.
\textbf{9}
(1990),
677--698.
%
\bibitem{Ma4}
D. C. Mayer,
\textit{The distribution of second \(p\)-class groups on coclass graphs},
J. Th\'eor. Nombres Bordeaux
\textbf{25}
(2013),
no. 2,
401--456,
DOI 10.5802/jtnb.842.
%
\bibitem{Ma6}
D. C. Mayer,
\textit{Periodic bifurcations in descendant trees of finite \(p\)-groups},
Adv. Pure Math.
\textbf{5}
(2015),
no. 4,
162--195,
DOI 10.4236/apm.2015.54020,
Special Issue on Group Theory,
March 2015.
%
\bibitem{GNO}
G. Gamble, W. Nickel, and E. A. O'Brien,
\textit{ANU p-Quotient --- p-Quotient and p-Group Generation Algorithms},
2006,
an accepted GAP package, available also in MAGMA.
%
\bibitem{GAP}
The GAP Group,
\textit{GAP -- Groups, Algorithms, and Programming --- a System for Computational Discrete Algebra},
Version 4.8.6,
Aachen, Braunschweig, Fort Collins, St. Andrews,
2016,
\verb+(http://www.gap-system.org)+.
%
\bibitem{Ma4a}
D. C. Mayer,
\textit{The distribution of second \(p\)-class groups on coclass graphs},
27\`iemes Journ\'ees Arithm\'etiques (JA 2011),
Faculty of Mathematics and Informatics,
Univ. of Vilnius,
Lithuania,
presentation delivered on July 01, 2011.
%
\bibitem{Ma12}
D. C. Mayer,
\textit{Three-stage towers of \(5\)-class fields},
arXiv: 1604.06930v1 [math.NT] 23 Apr 2016.
%
%
\bibitem{Ma15}
D. C. Mayer,
\textit{Recent progress in determining \(p\)-class field towers},
Gulf J. Math. (Dubai, UAE).
(arXiv: 1605.09617v1 [math.NT] 31 May 2016.)
%
\bibitem{Ma15b}
D. C. Mayer,
\textit{Recent progress in determining \(p\)-class field towers},
1st International Colloquium of Algebra, Number Theory, Cryptography and Information Security (ANCI 2016),
Taza, Morocco,
invited keynote delivered on November 12, 2016.
%
\bibitem{Ne}
B. Nebelung,
\textit{Klassifikation metabelscher \(3\)-Gruppen
mit Faktorkommutatorgruppe vom Typ \((3,3)\)
und Anwendung auf das Kapitulationsproblem}
(Inauguraldissertation,
Universit\"at zu K\"oln,
1989).
%
\bibitem{AHL}
J. A. Ascione, G. Havas and C. R. Leedham-Green,
\textit{A computer aided classification of certain groups of prime power order},
Bull. Austral. Math. Soc.
\textbf{17}
(1977),
257--274, Corrigendum 317--319, Microfiche Supplement p. 320,
DOI 10.1017/s0004972700010467.
%
\bibitem{As}
J. A. Ascione,
\textit{On \(3\)-groups of second maximal class},
Ph.D. Thesis,
Australian National University,
Canberra,
1979.
%
\bibitem{HeSm}
F.-P. Heider und B. Schmithals,
\textit{Zur Kapitulation der Idealklassen
in unverzweigten primzyklischen Erweiterungen},
J. Reine Angew. Math.
\textbf{336}
(1982),
1--25.
%
\bibitem{Ma0}
D. C. Mayer,
\textit{List of discriminants \(d_L<200\,000\) of totally real cubic fields \(L\),
arranged according to their multiplicities \(m\) and conductors \(f\)}
(Computer Centre, Department of Computer Science, University of Manitoba, Winnipeg, Canada,
1991, Austrian Science Fund, Project Nr. J0497-PHY).
%
\bibitem{Bu}
M. R. Bush,
\textit{IPADs of real quadratic fields with \(3\)-class rank two and discriminants up to \(10^9\)},
\(11\) July \(2015\),
private communication.
%
\bibitem{Bl}
N. Blackburn,
\textit{On prime-power groups in which the derived group has two generators},
Proc. Camb. Phil. Soc.
\textbf{53}
(1957),
19--27.
%
\bibitem{Sh}
I. R. Shafarevich,
\textit{Extensions with prescribed ramification points} (Russian),
Publ. Math., Inst. Hautes \'Etudes Sci.
\textbf{18}
(1964),
71--95.
(English transl. by J. W. S. Cassels in
Amer. Math. Soc. Transl.,
II. Ser.,
\textbf{59}
(1966),
128--149.)
%
%
\bibitem{BaBu}
L. Bartholdi and M. R. Bush,
\textit{Maximal unramified \(3\)-extensions of imaginary quadratic fields and \(\mathrm{SL}_2\mathbb{Z}_3\)},
J. Number Theory
\textbf{124}
(2007),
159--166.
%
%
%
%
%
%
%
%
%
%
\bibitem{Fi}
C. Fieker,
\textit{Computing class fields via the Artin map},
Math. Comp.
\textbf{70}
(2001),
no. 235,
1293--1303.
%
%
%
\bibitem{Ma}
D. C. Mayer,
\textit{Principalization in complex \(S_3\)-fields},
Congressus Numerantium
\textbf{80}
(1991),
73--87.
(Proceedings of the Twentieth Manitoba Conference on Numerical Mathematics and Computing,
Univ. of Manitoba, Winnipeg, Canada, 1990.)
%
\bibitem{So}
A. Scholz,
\textit{Idealklassen und Einheiten
in kubischen K\"orpern},
Monatsh. Math. Phys.
\textbf{40}
(1933),
211--222.
%
%
%
%
%
\end{thebibliography}
